\documentclass[final,hidelinks]{amsart}

\usepackage[title,toc,titletoc,page]{appendix}
\usepackage{pdfpages}

\usepackage{booktabs}
\newcommand{\forarray}[1]{\renewcommand{\arraystretch}{#1}}

\usepackage{grffile}
\usepackage{mathtools,trimclip}
\usepackage{blkarray}
\usepackage{amsfonts}
\usepackage{amsmath}
\usepackage{amssymb}
\usepackage{mathtools}
\usepackage{multicol}
\usepackage{latexsym}
\usepackage[all]{xy}
\usepackage{xcolor}
\usepackage{multicol}
\usepackage{bm}
\usepackage{extarrows}
\usepackage{enumerate, enumitem}
\usepackage{tikz}
\usetikzlibrary{calc}
\usepackage{hyperref}
\usepackage{soul,cancel}
\usepackage{graphicx}
\graphicspath{{./FinalPaper/}}

\usepackage[mathlines]{lineno}

\hypersetup{
  colorlinks   = true,    % Colours links instead of ugly boxes
  urlcolor     = blue,    % Colour for external hyperlinks
  linkcolor    = blue,    % Colour of internal links
  citecolor    = red      % Colour of citations
}

\usepackage[left=1in,right=1in,top=1in,bottom=1in]{geometry}

\usetikzlibrary{positioning}
\usepackage{caption}
\usepackage{qtree}

\usepackage[normalem]{ulem}

\makeatletter
\newcommand{\addresseshere}{%
  \enddoc@text\let\enddoc@text\relax
}
\makeatother

\newtheorem{theorem}{Theorem}[section]
\newtheorem{lemma}[theorem]{Lemma}
\newtheorem{proposition}[theorem]{Proposition}
\newtheorem{corollary}[theorem]{Corollary}
\newtheorem*{theorem*}{Theorem}

\theoremstyle{definition}
\newtheorem{definition}[theorem]{Definition}
\newtheorem{example}[theorem]{Example}
\newtheorem{conjecture}[theorem]{Conjecture}
\theoremstyle{remark}
\newtheorem{remark}[theorem]{Remark}
\newtheorem{notation}[theorem]{Notation}
\newtheorem{assumption}[theorem]{Assumption}

\newcommand{\R}{\mathbb{R}}

\newcommand{\msc}{\mathsf{c}}

\newcommand{\ds}{\displaystyle}

\newcommand{\In}{\textit{In}}
\newcommand{\Out}{\textit{Out}}
\newcommand{\Leak}{\textit{Leak}}

\usepackage{nicematrix}
\definecolor{TAMU}{RGB}{140,0,0}
\definecolor{myblue}{RGB}{0,0,198}
\definecolor{myred}{RGB}{182,0,0}
\newcommand{\defcolor}[1]{{\color{TAMU}#1}}
\newcommand{\jordy}[1]{\defcolor{{#1}}}

\allowdisplaybreaks

\title{Identifiability of directed-cycle and catenary \\ linear compartmental models} 

\author[Ahmed]{Saber Ahmed} 
\address{Saber~Ahmed\\ 
         Department of Mathematics and Statistics\\ 
         Hamilton College\\ 
         Clinton\\ 
         New York \ 13323\\ 
         USA} 
\email{smahmed@hamilton.edu}
\urladdr{https://sites.google.com/view/saber-ahmed}
\author[Crepeau]{Natasha Crepeau} 
\address{Natasha~Crepeau\\ 
         Department of Mathematics\\ 
         University of Washington\\ 
         Seattle\\ 
         Washington \ 98195\\ 
         USA} 
\email{ncrepeau@uw.edu}
\urladdr{https://sites.google.com/view/natasha-crepeau-math/home}
\author[Dessauer]{Paul R. Dessauer Jr.} 
\address{Paul~R. Dessauer Jr.\\ 
         Department of Mathematical Sciences\\ 
         The University of Texas at El Paso\\ 
         El Paso\\ 
         Texas \ 79968\\ 
         USA; 
         {\em Present address:}
         Department of Mathematics\\ 
         Texas A\&M University\\ 
         College Station\\ 
         Texas \ 77843\\ 
         USA
         } 
\email{pdessauer@tamu.edu}
\author[Edozie]{Alexis Edozie} 
\address{Alexis~Edozie\\ 
         Department of Mathematics and Statistics\\ 
         University of Massachusetts Amherst\\ 
         Amherst\\ 
         Massachusetts \ 01003\\ 
         USA;
         {\em Present address:}
        Department of Statistics\\ 
        University of Michigan\\
        Ann Arbor\\
        Michigan \ 48109\\
        USA
         } 
\email{aedozie@umich.edu}
\author[Garcia-Lopez]{Odalys Garcia-Lopez} 
\address{Odalys~Garcia-Lopez\\ 
         Department of Mathematics \& Statistics\\ 
         The College of New Jersey\\ 
         Ewing\\ 
         New Jersey \ 08628\\ 
         USA;
         {\em Present address:}
         Department of Mathematics\\
         Tufts University\\ 
         Medford\\ 
         Massachusetts \ 02155\\
         USA
         } 
\email{odalys.garcia\textunderscore lopez@tufts.edu}
\author[Grimsley]{Tanisha Grimsley} 
\address{Tanisha~Grimsley\\ 
         Mathematics Department\\ 
         Juniata College\\ 
         Huntingdon\\ 
         Pennsylvania\ 16652\\ 
         USA} 
\email{grimstn19@gmail.com}
\author[Lopez]{Jordy Lopez Garcia} 
\address{Jordy~Lopez Garcia\\ 
         Department of Mathematics\\ 
         Texas A\&M University\\ 
         College Station\\ 
         Texas \ 77843\\ 
         USA} 
\email{jordy.lopez@tamu.edu}
\urladdr{https://jordylopez27.github.io/}
\author[Neri]{Viridiana Neri} 
\address{Viridiana~Neri\\ 
         Department of Mathematics\\ 
         Columbia University\\ 
         New York\\ 
         New York \ 10027\\ 
         USA} 
\email{vjn2108@columbia.edu}
\author[Shiu]{Anne Shiu} 
\address{Anne~Shiu\\ 
         Department of Mathematics\\ 
         Texas A\&M University\\ 
         College Station\\ 
         Texas \ 77843\\ 
         USA} 
\email{annejls@tamu.edu}
\urladdr{https://people.tamu.edu/~annejls/}
\date{\today}

\begin{document}

\newcommand{\BibTeX}{{\scshape Bib}\TeX\xspace} %

\begin{abstract} 
A parameter of a mathematical model is structurally identifiable if it can be determined from noiseless experimental data.  Here, we examine the identifiability properties of two important classes of linear compartmental models: directed-cycle models and catenary models (models for which the underlying graph is a directed cycle or a bidirected path, respectively).  Our main result is a complete characterization of the directed-cycle models for which every parameter is (generically locally) identifiable.  Additionally, for catenary models, we give a formula for their input-output equations.  Such equations are used to analyze identifiability, so we expect our formula to support future analyses into the identifiability of catenary models.  Our proofs rely on prior results on input-output equations, and we also use techniques from linear algebra and graph theory.

\vspace{.1in}
\noindent
{\bf MSC Codes:} 
93B30,
37N25, 
92C45,
34A30, 
34A55,
05C50 

\vspace{.1in}
\noindent
 {\bf Keywords:} Linear compartmental model, structural identifiability, cycle model, catenary model
\end{abstract}
\maketitle

%\begin{keywords}
%    Linear compartmental model, structural identifiability, cycle model, catenary model
%\end{keywords}

%\begin{MSCcodes}
%93B30,37N25, 92C45, 34A30, 34A55, 05C50 
%\end{MSCcodes}

\section{Introduction} 
Linear compartmental models are used in a variety of biological applications \cite{BGMSS, delbary2016compartmental, eisenberg2017confidence, singularlocus, MSE}.  Each such model has an underlying directed graph that describes how a substance, such as a chemical, moves throughout a system. For instance, a model may involve a drug being administered to a patient, circulating through the body, and entering and exiting organs at a certain rate. Such rates are {parameters} of the model.

A parameter of a model is (generically locally) {\em structurally identifiable} (or {\em identifiable}, for short) if it can be determined from noiseless experimental data, and a model is itself 
(generically locally) 
{\em identifiable} if all of its parameters are identifiable.  
Our motivation is a recent characterization of  identifiability for a large class of linear compartmental models, %namely, those (with one input and one output) 
in which the underlying graph is a bidirected tree.  This classification, which is due to Bortner {\em et al.}~\cite{BGMSS}, is as follows: {\em A bidirected-tree model with one input and one output is generically locally identifiable if and only if it has at most one leak, and the distance from the input to the output is at most one} (see Proposition~\ref{prop:classification-identifiable-models}).

Several questions arise:
\begin{enumerate}[label=({\bf Q\arabic*})]
    \item In {\em unidentifiable} bidirected-tree models, which individual parameters (if any) are identifiable?
    \item Is there a characterization of identifiability -- analogous to that of Bortner {\em et al.} -- for models that are not bidirected-tree models?
    \item How does Bortner {\em et al.}'s classification generalize to allow for models with more than one input and/or more than one output?
\end{enumerate}

Our work provides partial answers to these questions.  For~({\bf Q1}), we consider an important class of bidirected-tree models, namely, {\em catenary models}; these are the models in which the underlying graph is a bidirected path. 
An example is shown in Figure~\ref{fig: directedcyclex}.  
See~\cite{delbary2016compartmental,scussolini2017physiology} for instances of catenary models  appearing in applications.

\begin{figure}[ht]
\centering
\begin{tikzpicture}[scale=0.9]
        \tikzset{node basic/.style={draw, ultra thick, text=black, minimum size=2em}}
        \tikzset{node circle/.style={node basic, circle}}
        \tikzset{line basic/.style={very thick, ->}}
        
        \node[node circle, minimum size=1.1cm] (1) at (0, 0) {\(1 \)};
        \node[node circle, minimum size = 1.1cm] (2) at (1.5, 1.5) {\(2\)};
        \node[node circle, minimum size=1.1cm] (3) at (4.5, 1.5) {\(3\)};
        \node[node circle] (n-1) at (4.5,-1.5) {\small \(n-1\)};
        \node[node circle, minimum size=1.1cm] (n) at (1.5, -1.5) {\(n\)};
        
        \draw[line basic] (n.north west)--(1.south east) node[xshift=.1cm, yshift=-.55cm] {\(k_{1n}\)};
        \draw[line basic] (1.north east)--(2.south west) node[xshift= -.55cm, yshift=-0.1cm] {\(k_{21}\)};
        \draw[line basic] (2.east)--(3.west) node[midway, above] {\(k_{32}\)};
        \draw[line basic] (n-1.west)--(n.east) node[midway, below] {\(k_{n,n-1}\)};
        % input
        \draw[line basic] (-1.5,1.5)--(1.north west) node[midway, left] {\(\text{input}\)};
        %output
        \draw (-1.5,-1.5)--(1.south west) node[midway, left] {\(\text{output}\)};
        \node at (-1.58,-1.58) [circle, draw, fill=none,  inner sep=2pt]{};
        %leak
        \draw[line basic] (3.east)--(7,1.5) node[midway, above] {\(k_{03}\) (leak)};
        \draw[line basic] (5.5,-.5)--(n-1.north east) node[yshift=.2cm, xshift=1.1cm] {\(k_{n-1,n-2}\)};
        \draw[line basic] (3.south east)--(5.5, .5) node[xshift = .1cm, yshift = .5cm] {\(k_{43}\)};
        % dot-dot-dot
        \node at (5.5, .25)  [circle, fill, inner sep=1pt]{};
        \node at (5.5, 0)  [circle, fill, inner sep=1pt]{};
        \node at (5.5, -.25)  [circle, fill, inner sep=1pt]{};
        %
        %----------------
        % CATENARY MODEL
        %----------------
        \node[node circle] (1cat) at (-11, 0) {\(1\)};
        \node[node circle] (2cat) at (-9, 0) {\(2\)};
        \node[node circle] (3cat) at (-7, 0) {\(3\)};
        \node[node circle] (ncat) at (-5, 0) {\(n\)};
        % Edges
        \draw[line basic] (1cat.north east)--(2cat.north west) node[midway, above] {\(k_{21}\)};
        \draw[line basic] (2cat.south west)--(1cat.south east) node[midway, below] {\(k_{12}\)};
        \draw[line basic] (2cat.north east)--(3cat.north west) node[midway, above] {\(k_{32}\)};
        \draw[line basic] (3cat.south west)--(2cat.south east) node[midway, below] {\(k_{23}\)};
        \node at (-6.25, 0)  [circle, fill, inner sep=1pt]{};
        \node at (-6, 0)  [circle, fill, inner sep=1pt]{};
       \node at (-5.75, 0)  [circle, fill, inner sep=1pt]{};
        % input
        \draw[line basic] (-10.5,1.5)--(2cat.north) node[midway, right] {\(\text{input}\)};
        %output
        \draw (-6.5,-1.5)--(ncat.south west) node[midway, right] {\(\text{output}\)};;
        \node at (-6.58,-1.58) [circle, draw, fill=none,  inner sep=2pt]{};
\end{tikzpicture}
\caption{A catenary model (left), and a directed-cycle model (right).
} 

\label{fig: directedcyclex}
\end{figure}

We give a formula for 
certain equations arising from catenary models (Theorem~\ref{thm:catenary-coefficient-map}).  Specifically, these are the {\em input-output equations}, which are used to analyze identifiability.  
It is our hope that, in the future, our formula will lead to a complete answer to question~({\bf Q1}) for catenary models.

For question~({\bf Q2}), we focus on a class of non-bidirected-tree models: {\em directed-cycle models}.  One such model appears in Figure~\ref{fig: directedcyclex}.  
In contrast to the classification of 
Bortner {\em et al.} (Proposition~\ref{prop:classification-identifiable-models}), 
directed-cycle models with one input and one output can be identifiable and yet (i) have more than one leak and/or (ii) be such that the distance from the input to output is arbitrarily large.  

Our main result 
completely characterizes (generic, local) identifiability for all directed-cycle models, even those with more than one input or output:
\begin{theorem*}[Theorem~\ref{thm:main-cycle}] 
A directed-cycle model is generically locally
identifiable if and only if it is ``leak-interlacing''.
\end{theorem*}
 
\noindent 
Informally, a model is leak-interlacing if its leaks ``interlace'' with the inputs and outputs, in the sense that between any two leaks there is at least one input or output (see Definition~\ref{def:interlacing}). 

Theorem~\ref{thm:main-cycle} partially answers question~({\bf Q2}) and in fact also addresses question~({\bf Q3}).  Indeed, Theorem~\ref{thm:main-cycle} concerns directed-cycle models with any number of inputs and outputs (and leaks). 
In fact, to our knowledge, {\em this is the first result on the identifiability of a family of models that allows for any number of inputs and outputs.}  

Some cases of Theorem~\ref{thm:main-cycle} were previously known~\cite{GOS,singularlocus}. These cases pertain to directed-cycle models with only one input and one output. Indeed, Theorem~\ref{thm:main-cycle} extends results of 
Gerberding, Obatake, and Shiu~\cite{GOS}, as summarized in Table~\ref{tab:summary-cycle}. 

\begin{table*}
\centering
\forarray{1.6}
\begin{tabular}{@{}llll@{}}
\toprule
Output ($p$)                          & Leaks                                                                                                 & Result                                       & Reference                                                                               \\
\hline

(any)                                    & $\lvert Leak \rvert \leq 1$                                                                                                          & identifiable                                 & \cite[Theorem 3.4]{GOS}                                            \\
    $p=1$ or $p=n$  
& (see result)                                                                                                             & 

identifiable $\Leftrightarrow$  $\lvert Leak \rvert \leq 1$         

& 
    \begin{minipage}{1.3in}
    \cite[Theorem 3.7 and Corollary~3.11]{GOS}         
    \end{minipage}
   \\
%---------------------
% p between 2 and n-1
%---------------------
 $2 \leq p \leq n-1$ 
& 
(see result)
& 
identifiable $\Leftrightarrow$ leak-interlacing condition
& 
Theorem~\ref{thm:main-cycle}
    \\
(any)                                    & $\lvert Leak \rvert \geq 3$                                                                                                          & unidentifiable                               & Corollary~\ref{cor:cycle-3plus-leaks}   \\                                             
\bottomrule
\end{tabular}
\caption{Classification of directed-cycle models with one input and one output.  Specifically, this table displays results on (generic, local) identifiability of 
directed-cycle models $(G, \In, \Out, \Leak)$ with $n \geq 3$ compartments, $\In=\{1\}$, and $\Out=\{p\}$.
The leak-interlacing condition is described in Definition~\ref{def:interlacing}.
\label{tab:summary-cycle}
}
\end{table*}

A final direction of our work builds on Theorem~\ref{thm:main-cycle} (but does not 
 address questions~({\bf Q1})--({\bf Q3})).   
We investigate, for directed-cycle models that are identifiable, their ``singular loci''.   The singular locus refers to the measure-zero set of parameter vectors that, informally speaking, are not identifiable (for details, see Definition~\ref{def:singular-locus})~\cite{singularlocus}. 
In other words, this set represents the parameters that can {\em not} be estimated from data.

In the case of directed-cycle models with input, output, and leak in a single compartment, the singular locus was computed by Gross {\em et al.}~\cite{GHMS}; in particular, this set is a union of hyperplanes.  Here we consider additional directed-cycle models, and although we can not compute the full set in all cases, we nevertheless prove that the singular locus contains certain hyperplanes 
(Theorems~\ref{thm:singular-locus-cycle} and~\ref{thm:1-1-p-singular-locus}).

Our results, like many of those in prior works~\cite{BGMSS,
BM-2022,
BM-indisting,
CJSSS,GOS,GHMS, singularlocus}, are proven using combinatorial tools involving graph theory, as well as algebraic methods.  The reason for this is that we take a differential-algebra approach to identifiability, and use prior results to reduce the question of identifiability (for linear compartmental models) to one of combinatorics and linear algebra.

This paper is structured as follows. 
Section \ref{sec:background} introduces linear compartmental models and recalls some useful results about such models.  In Sections \ref{sec:cycle-models}--%, \ref{sec:SL}, and 
\ref{sec:catenary}, we prove our results for directed-cycle models and catenary models. Lastly, we discuss remaining open questions in Section~\ref{sec:Conj}. 

\section{Background}\label{sec:background}
This section introduces linear compartmental models (Section~\ref{sec:models}), their ODEs and input-output equations (Section~\ref{sec:i-o-equations}), and identifiability (Section~\ref{sec:ident}), following the notation in prior work~\cite{GOS,MSE}.  We also recall some useful results on symmetric polynomials (Section~\ref{sec:sym-polyn}).

\subsection{Linear compartmental models} \label{sec:models}
Before giving precise definitions, we begin with a motivating example: the linear compartmental model shown in Figure~\ref{fig:cycle-running-example}, which has three compartments.  The three (directed) edges between compartments form a (directed) cycle, while the arrows labeled $k_{01}$ and $k_{03}$ are called {\bf leaks}. Each leak represents a loss from a compartment (by degradation, for instance, or from an outflow).
At compartment-1, there is an inward-pointing arrow, which is the {\bf input}, representing (for instance) the intake of a drug.  
Finally, the symbol 
\begin{tikzpicture}
    \draw (0,0) circle (0.06);
    \draw (0.06,0.06) -- (.16,.19);
\end{tikzpicture}
represents the \textbf{output} of the system; 
output compartments represent, for instance, the locations at which we measure the drug concentration. 
%The parameters are denoted by $k_{ij}$. 

\begin{figure}[ht]
\begin{center}
\begin{tikzpicture}[
roundnode/.style={circle, draw=black, very thick, minimum size=10mm},
arrowbasic/.style={very thick, ->},
]

%Nodes
\node[roundnode](middle){2};
\node[roundnode](leftcomp)  [left=of middle] {1};
\node[roundnode](rightcomp) [below=of middle] {3};

%Lines
\draw[arrowbasic] (leftcomp.east) -- (middle.west) node[pos=.5, above] {\(k_{21}\)};
\draw[arrowbasic] (middle.south) -- (rightcomp.north) node[pos=.5, right] {\(k_{32}\)};

%------------------
%======INPUT/OUTPUT/LEAKS 
%------------------
% Input
\draw[arrowbasic] (-2.05,1.35) -- node[right] {In}(leftcomp.north);
\draw(0,1.25) circle (0.1) node[right] { }; %{$y$};
\draw(0,1.15) -- (middle.north);
% LEAK
\draw[arrowbasic] (rightcomp.east)  -- node[below]{\(k_{03}\)}(1.5, -2.03);
% LEAK
\draw[arrowbasic] (leftcomp.west) -- node[below] {\(k_{01}\)}(-3.5, 0);

\draw[arrowbasic] (rightcomp.north west) -- (leftcomp.south east) node[pos=.5, left] {\(k_{13}\)};
\end{tikzpicture}
\end{center}
    \caption{A linear compartmental model.}
    \label{fig:cycle-running-example}
\end{figure}

\begin{definition}
A \textbf{linear compartmental model}, denoted by $(G, \In, \Out, \Leak)$, consists of a finite, directed graph $G = (V, E)$, where each vertex (or node) $i\in V$ represents a \textbf{compartment} of the model, with three sets $\In, \Out, \Leak \subseteq V$ denoting the sets of {\bf input}, {\bf output}, and {\bf leak} compartments, respectively. 

In a linear compartmental model, a directed edge $j\to i$ represents the flow or transfer of material, with associated parameter $k_{ij}$. Each leak compartment $\ell \in \Leak$ also has an associated parameter, $k_{0\ell}$ (as shown, for instance, in Figure~\ref{fig:cycle-running-example}).
\end{definition}

\begin{assumption} \label{assumption}
All linear compartmental models in this work are assumed to have {\color{purple} at} least one input and 
at least one output (that is, $\lvert In \rvert \geq 1$ and
$\lvert Out \rvert \geq 1$).
\end{assumption}

In figures, we represent a linear compartmental model by its graph, together with input and leak arrows, and the symbol \begin{tikzpicture}
    \draw (0,0) circle (0.06);
    \draw (0.06,0.06) -- (.16,.19);
\end{tikzpicture} (for outputs).
Additionally, we label compartments by positive integers. See, for instance, Figures~\ref{fig: directedcyclex}--\ref{fig: bidirectedcatex}.

\begin{example}
For the linear compartmental model in Figure \ref{fig: directedcyclex},
we have $In = Out = \{1\}$ and $Leak = \{3\}$.  For the model in Figure \ref{fig:cycle-running-example}, $In = \{1\}$, $Out = \{2\}$, and $Leak = \{1,3\}$.
\end{example}

This work focuses on certain types of linear compartmental models, which we define next.

\begin{definition} \label{def:types-of-models} 
Consider a linear compartmental model $\mathcal M = (G, \In, \Out, \Leak)$.
%, with directed graph $G = (V, E)$.
\begin{enumerate}
    \item $\mathcal M$ is a \textbf{bidirected-tree model} if $G$ is a bidirected-tree graph (that is, $G$ is obtained from an undirected tree by making all edges bidirected).
    \item $\mathcal M$ is a \textbf{catenary model} if $G = (V, E)$ is a bidirected-path graph: $V = \{1,2, \ldots, n \}$ (for some positive integer $n$) and 
    $E = \{1 \leftrightarrows 2,~ 2 \leftrightarrows 3,~ \dots, (n{-}1) \leftrightarrows n \}$.
    \item $\mathcal M$ is a \textbf{directed-cycle model} (or {\bf cycle}, for short) if $G = (V, E)$ 
    is a directed cycle: $V = \{1,2, \ldots, n \}$, where $n \geq 3$, and 
        $E = \{1\to 2,~ 2\to 3,~ \dots, (n{-}1) \to n,~ n \to 1\}$.
    \item $\mathcal{M}$ is \textbf{strongly connected} if for any pair of vertices $i,j$ in $G$, there is a directed path from $i$ to $j$ and also a directed path from $j$ to $i$.
\end{enumerate}
\end{definition}

\noindent
Catenary models are bidirected-tree models, but directed-cycle models are not.  
Both catenary models (and, more generally, bidirected-tree models) and directed-cycle models are 
strongly connected.

\begin{figure}[h]
\centering
\begin{tikzpicture}
        \tikzset{node basic/.style={draw, ultra thick, text=black, minimum size=2em}}
        \tikzset{node circle/.style={node basic, circle}}
        \tikzset{line basic/.style={very thick, ->}}
        
        \node[node circle] (1) at (0, 0) {\(1\)};
        \node[node circle] (2) at (2, 0) {\(2\)};
        \node[node circle] (3) at (4, 0) {\(3\)};
        \node[node circle] (4) at (6, 0) {\(4\)};

        \draw[line basic] (1.north east)--(2.north west) node[midway, above] {\(k_{21}\)};
        \draw[line basic] (2.south west)--(1.south east) node[midway, below] {\(k_{12}\)};
        \draw[line basic] (2.north east)--(3.north west) node[midway, above] {\(k_{32}\)};
        \draw[line basic] (3.south west)--(2.south east) node[midway, below] {\(k_{23}\)};
        \draw[line basic] (3.north east)--(4.north west) node[midway, above] {\(k_{43}\)};
        \draw[line basic] (4.south west)--(3.south east) node[midway, below] {\(k_{34}\)};
        \draw[line basic] (-1,1)--(1.north west) node[midway, left] {\(\text{In}\)};
        % OUTPUT
        \draw (1.2,-1)--(2.south) ;
        \draw(1.1,-1.05) circle (0.1) node[right] {};        
        % LEAK
        \draw[line basic] (2.north)--(2,1.5) node[midway, left] {\(k_{02}\)};
        \draw[line basic] (4.north)--(6,1.5) node[midway, left] {\(k_{04}\)};
        %\node at (-1,-1) [circle, fill, inner sep=1.5pt]{};
\end{tikzpicture}
\caption{A catenary model.}
\label{fig: bidirectedcatex}
\end{figure}

\begin{example}  \label{ex:cycle-and-cat}
Directed-cycle models are shown in Figures \ref{fig: directedcyclex} and ~\ref{fig:cycle-running-example}, while catenary models are depicted in Figures~\ref{fig: directedcyclex} and~\ref{fig: bidirectedcatex}.
\end{example}

\subsection{ODEs and input-output equations} \label{sec:i-o-equations}
To define the ODEs arising from a linear compartmental model, we use the following matrix.
\begin{definition}
The \textbf{compartmental matrix} of a linear compartmental model $(G, \In, \Out, \Leak)$ with $n$ compartments, where $G=(V,E)$, is the $n\times n$ matrix $A = (a_{ij})$ given by:
    \begin{equation*}
        a_{ij} ~:=~
            \begin{cases}
                \ds -k_{0i} - \sum_{p: i \to p \in E} k_{pi}  &\text{if $i=j$ and $i \in \Leak$},\\
                \ds - \sum_{p:i \to p\in E} k_{pi} &\text{if $i=j$ and $i \notin \Leak$},\\
                k_{ij} &\text{if $i\neq j$ and $j \to i$ is an edge of $G$},\\
                0 &\text{if $i\neq j$ and $j\to i$ is not an edge of $G$}.
            \end{cases}.
    \end{equation*}
\end{definition}

\begin{example}\label{ex:3-cycle-matrix} 
The model in Figure~\ref{fig:cycle-running-example} has the following compartmental matrix:
    \begin{align} \label{eq:A-matrix-for-0}
    A ~=~ \begin{pmatrix}
        -k_{01} - k_{21} & 0& k_{13} \\
         k_{21} & -k_{32} & 0 \\
         0 & k_{32} & -k_{03} -k_{13}
    \end{pmatrix}~.
        \end{align}
\end{example}

A linear compartmental model $\mathcal M = (G, \In, \Out, \Leak)$ with $n$ compartments defines a system of ODEs, as follows. Let $u(t)=(u_{1}(t),u_{2}(t),\dots,u_{n}(t))^{T}$ denote the vector of inputs (where $u_{j}(t)=0$ when $j\notin In$) at time $t$. 
 
The vector $x(t) = (x_1(t), x_2(t), \hdots, x_n(t))^{T}$ denotes the vector of concentrations at all compartments.  
Let $A$ denote the compartmental matrix of $\mathcal M$.  Then 
$\mathcal M$ defines the following system of linear ODEs (with inputs):
\begin{align} \label{eq:ode}
    x'(t) &= Ax(t) + u(t)  \\
    y_i(t) &= x_i(t)~,
    \notag
\end{align}
for all $i\in Out$.
\begin{example} [Example~\ref{ex:3-cycle-matrix}, continued] \label{ex:3-cycle-matrix-continued}
For the model in Figure~\ref{fig:cycle-running-example}, the ODEs~\eqref{eq:ode} are as follows: 
    \begin{align*} 
    x_1'(t) &~=~
        -(k_{01}+k_{21}) x_{1}(t)
        + k_{13}x_{3}{\color{blue}(t)} + u_1(t)\\
    x_2'(t)&~=~ k_{21}x_{1}(t)-k_{32}x_{2}(t)\\
    x_3'(t)&~=~ k_{32}x_{2}(t)-(k_{03}+k_{13})x_{3}(t) \\ %\text{ with}\\
    y_2(t)&~=~x_{2}(t)~.
    \end{align*}
\end{example}
An equation that holds along every solution $x(t)$ of the ODEs~\eqref{eq:ode} is an {\em input-output equation} if it involves only the parameters $k_{ij}$, input variables $u_i\vcentcolon = u_i(t)$, output variables $y_i\vcentcolon = y_i(t)$, and their derivatives.
The next result recalls a formula for input-output equations, which was proven by Meshkat, Sullivant, and Eisenberg~\cite[Theorem 2]{MSE}. 

\begin{notation} \label{notation:submatrix}
    We use $(B)^{j,i}$ to denote the submatrix obtained from a matrix $B$ by removing row $j$ and column $i$.
\end{notation}

\begin{proposition}[Input-output equations,~\cite{MSE}] \label{prop:in/out} Let $\mathcal{M} = (G, \In, \Out, \Leak)$ be a linear compartmental model with $n$ compartments and at least one input, and let $A$ be its compartmental matrix. For $i\in \Out$, the following equation is an {input-output equation} for $\mathcal{M}$:
\begin{equation} \label{eq:in-out}
    \det(\partial I-A)y_{i} ~=~ \sum_{j\in In}(-1)^{i+j}\det \left[ (\partial I-A)^{j,i} \right] u_j~,
\end{equation}
\noindent 
where $\partial I$ denotes the $n{\times}n$ matrix where each diagonal entry is the differential operator $d/dt$ and each off-diagonal entry is 0.
\end{proposition}

\begin{example}[Example~\ref{ex:3-cycle-matrix-continued}, continued] \label{ex:3-cycle-again}
For the model in Figure~\ref{fig:cycle-running-example}, we use the matrix $A$ shown in~\eqref{eq:A-matrix-for-0} to compute the input-output equation~\eqref{eq:in-out} (where $i=2$ and $j=1$): 
    \begin{align*}
        \det 
        \begin{pmatrix}
        \frac{d}{dt} +k_{01} + k_{21} & 0& -k_{13} \\
         -k_{21} & \frac{d}{dt} +k_{32} & 0 \\
         0 & -k_{32} & \frac{d}{dt} +k_{03} +k_{13}
    \end{pmatrix}
        y_2
        ~{=}~
        (-1)^{2+1}
        \det 
        \begin{pmatrix}
         -k_{21} & 0 \\
         0  & \frac{d}{dt} +k_{03} +k_{13}
    \end{pmatrix}
    u_1~.
    \end{align*}
This equation expands as follows:
    \begin{align} \label{eq:i-o-running-ex}
    y_2^{(3)}
    & +
    (k_{01}  +k_{03} +k_{13} + k_{21}  +k_{32} ) y_2^{(2)}
     \\
    \notag
    & +
    (k_{01}k_{03}+
    k_{01}k_{13}+
    k_{01}k_{32}+
    k_{03}k_{21}+
    k_{03}k_{32}+
    k_{13}k_{21}+
    k_{13}k_{32}+
    k_{21}k_{32})
    y_2^{(1)}
     \\
    \notag
    & +
    (k_{01}k_{03}k_{32}+ 
    k_{01}k_{13}k_{32}+
    k_{03}k_{21}k_{32}
    )
    y_2
         \\
    \notag
    & =~
    k_{21} u_1^{(1)} + 
    (k_{03} +k_{13}) k_{21} u_1~.
    \end{align}
\end{example}

From the input-output equation(s)~\eqref{eq:in-out}, %in Proposition \ref{prop:in/out}, 
we derive a \textbf{coefficient map}:
$$\msc = (\msc_1,\msc_2,\dots, \msc_m) : \R^{\lvert E\rvert + \lvert \Leak\rvert} \rightarrow \R ^m~,$$ 
which
evaluates each vector of parameters $k_{ij}$ (arising from edges $(j,i)\in E$, or $i=0$ and $j\in\Leak$) at 
a vector of a subset of the coefficients of the input-output equation(s), 
such that all non-constant coefficients are included.
(Here, $m$ denotes the number of some choice of 
such coefficients.)

\begin{remark}
In prior works
(such as~\cite{BGMSS,CJSSS}), constant coefficients are excluded in the coefficient map.  In contrast, we allow for some constant coefficients.  This freedom makes it easier to state some results, without affecting identifiability properties.
\end{remark}

\begin{example}[Example~\ref{ex:3-cycle-again}, continued] \label{ex:input-output}
We revisit the model in Figure~\ref{fig:cycle-running-example}.  The input-output equation~\eqref{eq:i-o-running-ex} yields the coefficient map 
$\msc : \R^{5} \rightarrow \R ^5$, defined by: 
\begin{align} \label{eq:coeff-map-for-main-ex}
    \notag
    \msc_1~&=~ k_{01}  +k_{03} +k_{13} + k_{21}  +k_{32} \\
    \notag
    \msc_2~&=~
    k_{01}k_{03}+
    k_{01}k_{13}+
    k_{01}k_{32}+
    k_{03}k_{21}+
    k_{03}k_{32}+
    k_{13}k_{21}+
    k_{13}k_{32}+
    k_{21}k_{32}
    \\
    %\notag
    \msc_3~&=~
    k_{01}k_{03}k_{32}+ 
    k_{01}k_{13}k_{32}+
    k_{03}k_{21}k_{32}
    \\
    \notag
    \msc_4~&=~
     k_{21}
    \\
    \notag
    \msc_5~&=~(k_{03} +k_{13}) k_{21}~.
\end{align}
\end{example}    

We end this subsection by showing, for models with two or more inputs or outputs, how to view each coefficient of the coefficient map as 
%a coefficient
coming from 
of a model with one input and one output. 

\begin{proposition}[Input-output equations] \label{prop:coeff-many-in-or-out} 
Let $\mathcal{M} = (G, \In, \Out, \Leak)$ be a linear compartmental model 
with at least one input, and let $A$ be the corresponding compartmental matrix. 
Let $i\in \Out$, and consider the corresponding 
input-output equation~\eqref{eq:in-out} for $\mathcal{M}$, which we rewrite below, for convenience:
\begin{equation} \label{eq:in-out-again}
    \det(\partial I-A)y_{i} ~=~ \sum_{j\in In}(-1)^{i+j}\det\left[(\partial I-A)^{j,i}\right]u_j~.
\end{equation}
For every non-negative integer $d\vcentcolon$
\begin{enumerate}
    \item The coefficient of $y_i^{(d)}$ 
    in the left-hand side of equation~\eqref{eq:in-out-again}
    equals the coefficient of $y_i^{(d)}$ in the (unique) input-output equation~\eqref{eq:in-out} for the one-input, one-output model $(G, \{j^*\}, \{i\}, Leak)$, where 
    $j^*$ is any input compartment (that is, $j^* \in \In$).
    \item For $j \in \In$, the coefficient of $u_j^{(d)}$ 
        in the right-hand side of equation~\eqref{eq:in-out-again}
    equals the coefficient of $u_j^{(d)}$ 
    in the (unique) input-output equation~\eqref{eq:in-out} for the one-input, one-output model $(G, \{j\}, \{i\}, Leak)$.
\end{enumerate}

\end{proposition}

\begin{proof}
This result follows directly from 

the fact that the input-output equation~\eqref{eq:in-out} for the model 
$(G, \{j\}, \{i\}, Leak)$
is as follows:
\begin{equation*} 
    \det(\partial I-A)y_{i} ~=~ (-1)^{i+j}\det\left[(\partial I-A)^{j,i}\right]u_j~.
\end{equation*}
\vspace{-0.1in}
\end{proof}

We use Proposition~\ref{prop:coeff-many-in-or-out} later in some of our proofs in Section~\ref{sec:cycle-models}.

\subsection{Identifiability} \label{sec:ident}

From time-series input-output data $(u(t);y(t))$ arising from a model, we want to identify the values of the parameters $k_{ij}$. 
More precisely, from a generic choice of the inputs and initial conditions, we wish to be able to recover the parameters (at least locally, that is, up to a finite set) from exact measurements of both the inputs and the outputs. This is the notion of {\em generic local identifiability}. 

This concept of identifiability was proven by 
Ovchinnikov {\em et al.}
to be equivalent (for strongly connected models with at least one input) to the following definition, which assesses identifiability by way of coefficient maps~\cite[Main Result 3]{OPT}. 

\begin{definition}[Identifiability]\label{def:ident} 
Consider
a linear compartmental model $ \mathcal M =(G, \In, \Out, \Leak)$ that is strongly connected and has one or more inputs.  Let $G=(V,E)$, and let $\msc :  \R^{|E| + |\Leak|} \rightarrow \R ^m$ denote the coefficient map for $\mathcal M$.
\begin{enumerate}
    \item  $\mathcal M$ is \textbf{generically locally identifiable} 
    (or {\bf identifiable}, for short)
    if, except possibly on a set of measure zero, every point in $\R^{|E| + |\Leak|}$ admits a neighborhood on which     
$\msc$ is injective.
    \item $\mathcal M$ is {\bf unidentifiable} if $\mathcal M$ is not generically locally identifiable.
\end{enumerate}
\end{definition}

We rely on (yet another) equivalent definition of identifiability, as follows~\cite[Proposition~2]{MSE}.

\begin{proposition} \label{prop:rank-matrix-for-identifiability}
Consider
a linear compartmental model $ \mathcal M =(G, \In, \Out, \Leak)$ that is strongly connected and has one or more inputs.  Let $G=(V,E)$, and let $\msc :  \R^{|E| + |\Leak|} \rightarrow \R ^m$ denote the coefficient map for $\mathcal M$. Then $\mathcal M$ is {generically locally identifiable} if and only if the rank of the Jacobian matrix of $\msc$, when evaluated at a generic point, is equal to $|E| + |Leak|$.
\end{proposition}

\begin{example}[Example~\ref{ex:input-output}, continued] \label{ex:jacobian-matrix}
For the coefficient map $\msc$ shown in equation~\eqref{eq:coeff-map-for-main-ex},
the Jacobian matrix of $\msc$ 
--
where the columns correspond to the parameters 
$ k_{21}, k_{32}, k_{13}, k_{01}, k_{03}$, 
respectively
-- 
is the following $5 \times 5$ matrix:
    \begin{align}       
    \label{eq:matrix-for-example}
    \begin{pmatrix}
        1 & 1 & 1 & 1 & 1 \\
        k_{03}+k_{13}+k_{32}
            & k_{01}+k_{03}+k_{13}+k_{21}
            & k_{01}+k_{21}+k_{32}
            &
          k_{03}+k_{13}+k_{32} 
        & k_{01}+k_{21}+k_{32} 
            \\
            % Row 3
        k_{03} k_{32}
        &
        k_{01}(k_{03}+k_{13}) + k_{03}k_{21}
        &
        k_{01}k_{32}
        &
        k_{32}(k_{03}+k_{13}) 
        &
        k_{32}(k_{01}+k_{21})
\\
        1 & 0 & 0 & 0 & 0 \\
        k_{03}+k_{13}& 0 & k_{21} & 0 & k_{21}  \\
    \end{pmatrix}~.
    \end{align}  
It is straightforward to check (by hand or computer) that 
the determinant of this matrix is the following nonzero polynomial (in the parameters $k_{ij}$):
\begin{align} \label{eq:s-locus-example}
    k_{21}^2  k_{32} (k_{01} + k_{21} - k_{32}) ~.
\end{align}  
Thus, by Proposition~\ref{prop:rank-matrix-for-identifiability}, the model in Figure~\ref{fig:cycle-running-example} is generically locally identifiable.
\end{example}

Proposition~\ref{prop:rank-matrix-for-identifiability} 
translates the concept of identifiability to a property concerning Jacobian matrices, so the following well-known result will be useful~\cite[Theorem 2.2]{ER}.  

\begin{lemma}[Algebraic independence and the Jacobian] \label{lem:alg-indep-general}
Consider $n$ polynomials in $n$ variables, $f_1,f_2,\dots, f_n \in \mathbb{C}[x_1,x_2,\dots,x_n]$.
The Jacobian matrix of these polynomials 
has nonzero determinant if and only if the polynomials are algebraically
independent. 
\end{lemma}

Next, we state a result that will allow us to analyze modified versions of coefficient maps.

\begin{lemma}[Replace a coefficient by a quotient] \label{lem:modify-coeffs-by-division}
    Consider a map $\msc = (\msc_1,\msc_2, \dots, \msc_m): \mathbb{R}^{\ell} \to \mathbb{R}^m$, 
with  $\msc_1, \msc_2, \dots, \msc_m \in \mathbb{R}(x_1,x_2,\dots, x_{\ell}) \smallsetminus \{0\} $ (that is, each $\msc_i $ is a 
nonzero rational function in $\ell$ variables over $\mathbb{R}$).    
    Consider a map $\mathsf{d}$ obtained from $\mathsf{c}$ by replacing some $\msc_i$ (where $1 \leq i \leq m$) by $\msc_i / \msc_j$ (where $1 \leq j \leq m$, with $j \neq i$).   
    Then the 
    rank of the 
    Jacobian matrix of 
    $\msc$, when evaluated at a generic point, 
    equals the rank of the 
    Jacobian matrix of 
    $\mathsf{d}$, when evaluated at a generic point. 
\end{lemma}

\begin{proof}
For $1 \leq k \leq m$, let $r_k$ denote the $k$-th row of the Jacobian matrix of $\msc$.
It is straightforward to check that the 
    Jacobian matrix of 
    $\mathsf{d}$ is obtained from the
    Jacobian matrix of 
    $\msc$ by replacing 
    $r_i$ (the $i$-th row) by 
    the linear combination
    $(c_j)^{-1} r_i + -c_i (c_j)^{-2} r_j$.   
The result now follows directly.
\end{proof}

\subsection{Identifiability results concerning model operations and certain model classes} \label{sec:identifiability-prior-results}

In this section, we recall several results on identifiability.  
The first states that 
identifiability is preserved when inputs or outputs are added~\cite[Proposition~4.1]{GHMS}.

\begin{proposition}[Adding inputs or outputs,~\cite{GHMS}] \label{prop:add-in-out}
Consider a model $\widetilde{\mathcal{M}}$  obtained from a strongly connected linear compartmental model $\mathcal M$
by adding one or more inputs and/or one or more outputs.     
If $\mathcal M$ is generically locally
identifiable, 
then so is $\widetilde{\mathcal{M}}$. 
\end{proposition}

The next result, which is 
due to Bortner~{\em et~al.}~\cite[Theorem~5.2]{BGMSS}, characterizes identifiability in bidirected-tree models (Definition~\ref{def:types-of-models}(1)). 

\begin{proposition}[Bidirected-tree models,~\cite{BGMSS}]
\label{prop:classification-identifiable-models} 
Let $\mathcal M = (G, \In, \Out, \Leak)$ be a bidirected-tree model with one input and one output: $\In = \{j\}$ and $\Out = \{p\}$.
The following are equivalent:
\begin{enumerate}
    \item $\mathcal M$
is generically locally identifiable, and 
    \item 
    $|\Leak| \leq 1$ and $\text{dist}_G(j,p) \leq 1$, where $\text{dist}_G(j,p)$ is the number of edges in the 
    unique (directed) path in $G$ from compartment-$j$ to compartment-$p$.
\end{enumerate}
\end{proposition}

Next, we recall two results of Gerberding, Obatake, and Shiu, which pertain to directed-cycle models 
(Definition~\ref{def:types-of-models}~(3)).
One result concerns directed-cycle models 
with up to one leak~\cite[Theorem~3.4]{GOS}, and the other pertains to models 
with (up to relabeling) input in the first compartment and output in last compartment~\cite[Corollary~3.11]{GOS}.

\begin{proposition}[Directed-cycle models with at most one leak,~\cite{GOS}]
\label{prop:cycle-1-in-1-out}
Let $n \geq 3$, 
and let 
$\mathcal M = (G, \In, \Out, \Leak)$ be an $n$-compartment directed-cycle model (with 
$\lvert \In \rvert \geq 1$ and
$\lvert \Out \rvert \geq 1$).
\begin{enumerate}
    \item If $\mathcal M $ has at most one leak (i.e., $\lvert \Leak \rvert \leq 1$), then $\mathcal{M}$ is generically locally identifiable.
    \item 
    If $\In = \{i\}$ and $\Out = \{i-1\}$ (here, $i-1$ is taken mod $n$), for some $1 \leq i  \leq n$,  
    then $\mathcal M$ is generically locally identifiable if and only if $\lvert \Leak \rvert \leq 1$.
\end{enumerate}
\end{proposition}

\subsection{Symmetric polynomials} \label{sec:sym-polyn}
Symmetric polynomials often appear in coefficient maps.  In this subsection, we recall their definition and some results on their algebraic independence.  These results are used in later sections. In what follows, we use the standard notation $[n]:=\{1,2,\dots,n\}$ for positive integers $n$.

\begin{definition} \label{def:elem-sym-poly}
Consider positive integers $k$ and $n$, with $1 \leq k \leq n$.
Let $X_n$ denote the set of variables $x_1, x_2,\hdots, x_n$. 
The $k$-th \textbf{elementary symmetric polynomial} on $X_n$, which we denote by $e_k(X_n)$, is given by
\begin{align*}
e_k(X_n) ~:=~ \sum_{\substack{I \subseteq [n] \\ |I| = k}} 
    \left( \prod_{i\in I} x_i\right) ~.
\end{align*}
\end{definition}

\begin{remark} We use the notation $e_k$, rather than $e_k(X_n)$, when the set $X_n$ can be inferred without confusion. 

\end{remark}

\begin{example} \label{ex:elem-sym-poly}
    The elementary symmetric polynomials on $X_3=\{x_1,x_2,x_3\}$ are $e_1=x_1+x_2+x_3$, $e_2=x_1x_2 + x_1 x_3 + x_2x_3$, and $e_3=x_1 x_2 x_3$.
\end{example}

\begin{lemma}[Algebraic independence] \label{lem:alg-indep}
    Let $n \geq 1$.  
    Let $e_1, e_2,\dots, e_n$ denote the elementary symmetric polynomials on the set $X_n = \{x_1, x_2,\hdots, x_n\}$.
    \begin{enumerate}
        \item The following set of (all) elementary symmetric polynomials on $X_n$ is algebraically independent over $\mathbb C$: 
    \[
    \{ e_1, e_2, \dots, e_n \}~.
    \]
        \item Assume $n \geq 2$, and let $1 \leq k \leq n-1 $.  The following set of polynomials is algebraically independent over $\mathbb C$:  
    \[
    \{ e_1, e_2, \dots, e_{n-1},~ x_1+x_2+\dots + x_k \}~.
    \]       
    \end{enumerate}    
\end{lemma}

\begin{proof}
    Part~(1) is well known; for instance, see \cite{symmetric_fns}.

    Next, we prove Part~(2). By part (1), the set $\{ e_1, e_2, \dots, e_{n-1} \}$ is algebraically independent, so it suffices to show that the polynomial  $
s := x_1+x_2+\dots + x_k $
is transcendental (that is, not algebraic) over the field $\mathbb{C} (e_1, e_2, ..., e_{n-1})$.  

We proceed by contradiction. Assume that $s$ is algebraic over  $\mathbb{C} (e_1, e_2, ..., e_{n-1})$.  Basic results from field theory imply that $s$ has a minimal polynomial $f \in 
\mathbb{C} (e_1, e_2, ..., e_{n-1})[y]$ (so, $f(s)=0$). 
By construction, the coefficients of $f$ are in the field $\mathbb{C} (e_1, e_2, ..., e_{n-1})$ and so are symmetric functions (that is, they are fixed under permutations of the variables $x_1,x_2,\dots, x_n$).  It follows that the roots of $f$ also are fixed under such permutations.  
Hence, all permuted versions of $s$, that is, the sums of $k$ distinct variables $x_i$, are also roots of $f$.  
For instance, one such sum is $x_2 + x_3 + \dots + x_{k+1}$ (here, the assumption $k \leq n-1$ is used).  It follows that all such sums are in the splitting field of $f$ over $\mathbb{C} (e_1, e_2, ..., e_{n-1})$, which we denote by $F$ (and which is an algebraic extension of $\mathbb{C} (e_1, e_2, ..., e_{n-1})$).

Next, we claim that every difference $x_i - x_j$ (with $i \neq j$) is in $F$.
One instance of this claim can be seen as follows: 
$x_1 - x_{k+1} = (x_1+x_2+\dots + x_k) - (x_2 + x_3 + \dots + x_{k+1})$, and so this difference is the sum of two elements in $F$.  By symmetry, all other such differences are also in $F$.

Now we claim that each variable $x_i $ (for $1 \leq i \leq n$) is in $F$.
For instance, 
$kx_1= (x_1-x_2) + (x_1- x_3) + \dots + (x_1- x_k) + ( x_1 + x_2 + \dots +x_k)
$ shows that $k x_1$ is a sum of $k$ elements of $F$, which implies that $x_1$ is in $F$.  By symmetry, each variable $x_2,x_3,\dots, x_n$ is also in $F$.

We conclude that %the product 
$x_1 x_2 \dots x_n$, that is, the elementary symmetric polynomial $e_n$, is in~$F$.  Thus, $e_n$ is algebraic over $\mathbb{C} (e_1, e_2, ..., e_{n-1})$, which contradicts part~(1) and ends the proof.    
\end{proof}

\section{Identifiability of directed-cycle models} \label{sec:cycle-models}

The aim of this section is to completely categorize identifiability in cycle models  (with any number of inputs, outputs, and leaks).  This aim is achieved in  Theorem~\ref{thm:main-cycle} below. 
This result states that, 
outside of one exceptional case (Definition~\ref{def:exceptional-model}),
identifiability in cycle models is entirely characterized by a ``leak-interlacing'' condition (Definition~\ref{def:interlacing}). 

\begin{definition}
\label{def:exceptional-model}
An $n$-compartment directed-cycle model 
$\mathcal{M} = (G, \In, \Out, \Leak)$ 
is an {\bf exceptional model}
if, 
for some $1 \leq i \leq n$, 
we have
$In = \{i\},$ $ \Out = \{i - 1\},$  
$\lvert \Leak \rvert =2$, 
and
$ i-1 \in \Leak$ (where $i-1$ is taken mod $n$).
\end{definition}

Two exceptional models are shown in Figure~\ref{fig:cycle-exceptional}.  Each has $n=3$ compartments and an input in the first compartment ($i=1$ in the notation of Definition~\ref{def:exceptional-model}).

\begin{figure}[ht]
\begin{center}
\begin{tikzpicture}[
roundnode/.style={circle, draw=black, very thick, minimum size=10mm},
arrowbasic/.style={very thick, ->},
]

%Nodes
\node[roundnode](middle){2};
\node[roundnode](leftcomp)  [left=of middle] {1};
\node[roundnode](rightcomp) [below=of middle] {3};

%Lines
\draw[arrowbasic] (leftcomp.east) -- (middle.west) node[pos=.5, above] {\(k_{21}\)};
\draw[arrowbasic] (middle.south) -- (rightcomp.north) node[pos=.5, right] {\(k_{32}\)};

%======INPUT/OUTPUT/LEAKS 
\draw[arrowbasic] (-2.05,1.35) -- node[right] {In}(leftcomp.north);
\draw(rightcomp) -- (1.5,-2.03);
\draw(1.6,-2.03) circle (0.1) node[right] {};
\draw[arrowbasic] (leftcomp.west) -- node[below] {\(k_{01}\)}(-3.5, 0);
\draw[arrowbasic]  (rightcomp.west) -- node[below] {\(k_{03}\)} (-1.6,-2.03);

\draw[arrowbasic] (rightcomp.north west) -- (leftcomp.south east) node[pos=.5, left] {\(k_{13}\)};
\end{tikzpicture}
\ \ 
\begin{tikzpicture}[
roundnode/.style={circle, draw=black, very thick, minimum size=10mm},
arrowbasic/.style={very thick, ->},
]

%Nodes
\node[roundnode](middle){2};
\node[roundnode](leftcomp)  [left=of middle] {1};
\node[roundnode](rightcomp) [below=of middle] {3};

%Lines
\draw[arrowbasic] (leftcomp.east) -- (middle.west) node[pos=.5, above] {\(k_{21}\)};
\draw[arrowbasic] (middle.south) -- (rightcomp.north) node[pos=.5, right] {\(k_{32}\)};

%======INPUT/OUTPUT/LEAKS 
\draw[arrowbasic] (-2.05,1.35) -- node[right] {In}(leftcomp.north);
\draw(rightcomp) -- (1.5,-2.03);
\draw(1.6,-2.03) circle (0.1) node[right] {};
\draw[arrowbasic] (middle.east) -- node[below] {\(k_{02}\)} (1.5, 0);
\draw[arrowbasic]  (rightcomp.west) -- node[below] {\(k_{03}\)} (-1.6,-2.03);

\draw[arrowbasic] (rightcomp.north west) -- (leftcomp.south east) node[pos=.5, left] {\(k_{13}\)};
\end{tikzpicture}
\end{center}
    \caption{Two exceptional models.}
    \label{fig:cycle-exceptional}
\end{figure}

\begin{remark} \label{rem:exceptional}
Exceptional models are so named, because they must be excluded from the ``leak-interlacing'' condition (Definition~\ref{def:interlacing} below) in order for the classification result (Theorem~\ref{thm:main-cycle} below) to hold.  Indeed, by part~(2) of Proposition~\ref{prop:cycle-1-in-1-out}, exceptional models are unidentifiable.
\end{remark}

The idea behind ``leak-interlacing'' in the
following definition is that between any two leaks, there must exist at least one input or one output. 

\begin{definition} \label{def:interlacing} 
    Let $n \geq 3$.  
    An $n$-compartment directed-cycle model $\mathcal{M}=(G, \In, \Out, \Leak)$
    is {\bf leak-interlacing} if it is \uline{not} an exceptional model and, additionally, 
    one of the following holds: 
    \begin{enumerate}
        \item $\lvert Leak \rvert \leq 1$, or
        \item  $\lvert Leak \rvert \geq 2$, with $Leak = \{\ell_1, \ell_2, ..., \ell_z \}$ where $1 \leq \ell_1 < \ell_2 < \dots < \ell_z \leq n$, and, additionally,
        for all $\alpha \in \{1,2,\dots,z \}$, 
        we have
         $\{\ell_{\alpha}+1, \ell_{\alpha}+2, \dots, \ell_{\alpha+1} \} \cap (\In \cup \Out) \neq \varnothing$.
         (Here, $\ell_{z+1}:=\ell_1$, and the compartments $\ell_{z}{+}1, \ell_{z}{+}2,\dots, \ell_1$ are taken mod $n$.) 
    \end{enumerate}
\end{definition}

\begin{remark} \label{rem:path}
Part~(2) of Definition~\ref{def:interlacing} can be restated as follows:
we require that $\lvert Leak \rvert \geq 2$ and that 
        the 
        directed path
        in $G$ from $(\ell_{\alpha}{+}1)$ to $\ell_{\alpha +1}$
        contains at least one input or output, where ``path'' here is generalized to allow for paths of length $0$ 
        (such a length-$0$ path arises when $\ell_{\alpha}{+}1  = \ell_{\alpha +1}$).
\end{remark}

\begin{example}[Example~\ref{ex:jacobian-matrix}, continued] \label{ex:leak-interlacing}
Recall that the directed-cycle model in Figure~\ref{fig:cycle-running-example} has $3$ compartments, with $In=\{1\}$, $Out=\{2\}$, and $Leak = \{1, 3\}$.
This model is leak-interlacing; indeed, the path in the cycle from compartment-$2$ (here, $2=1+1$) to compartment-$3$ contains the output (at compartment-$2$), and the path 
(in the generalized sense to allow for length-0 paths, as in Remark~\ref{rem:path})
from $3+1(\mathrm{mod}~3)=1$ to $1$ contains the input (at $1$). The next theorem, which is proven later in this section, recapitulates what we saw earlier: this model is generically locally
 identifiable. 
\end{example}

\begin{theorem}[Identifiability of directed-cycle models] \label{thm:main-cycle}
A directed-cycle model is generically locally identifiable if and only if it is leak-interlacing. 
\end{theorem}

\begin{remark} \label{rem:connect-to-GOS}
Theorem~\ref{thm:main-cycle} significantly extends results of 
Gerberding, Obatake, and Shiu~\cite[\S3]{GOS}; these results were largely summarized earlier in Proposition~\ref{prop:cycle-1-in-1-out}.
In fact, our work is motivated by theirs, and indeed they posed the problem of classifying identifiability in directed-cycle models with two or more leaks~\cite[\S4]{GOS}. 
\end{remark}

\begin{example} \label{ex:not-leak-interlacing}
Consider the $4$-compartment directed-cycle model $\mathcal{M}=
(G, \In, \Out, \Leak)$ with 
$\In = \{1\}$, $\Out = \{3\}$, and $\Leak = \{1, 2\}$, as shown in Figure~\ref{fig:cycle-exceptional}. 
There is no input or output along the (generalized) path from $1+1=2$ to $2$ (i.e., compartment-$2$ is neither an input nor an output), so $\mathcal{M}$ is not leak-interlacing.  Thus, Theorem~\ref{thm:main-cycle} implies that this model is unidentifiable.

\begin{figure}[ht]
\begin{center}
\begin{tikzpicture}[
roundnode/.style={circle, draw=black, very thick, minimum size=10mm},
arrowbasic/.style={very thick, ->},
]

%Nodes
\node[roundnode](comp1){1};
\node[roundnode](comp2) [right=of comp1] {2};
\node[roundnode](comp3) [below=of comp2] {3};
\node[roundnode](comp4) [left=of comp3] {4};

% edges
\draw[arrowbasic] (comp1) -- (comp2) node[pos=.5, above] {\(k_{21}\)};
\draw[arrowbasic] (comp2) -- (comp3) node[pos=.5, right] {\(k_{32}\)};
\draw[arrowbasic] (comp3) -- (comp4) node[pos=.5, above] {\(k_{43}\)};
\draw[arrowbasic] (comp4) -- (comp1) node[pos=.5, left] {\(k_{14}\)};

% input
\draw[arrowbasic] (0,1.35) -- node[right] {In} (comp1);

% output
\draw(comp3) -- (3.5,-2.03);
\draw(3.6,-2.03) circle (0.1) node[right] {};

% leak
\draw[arrowbasic](comp1) -- node[below] {\(k_{01}\)} (-1.6,0);
\draw[arrowbasic](comp2) -- node[below] {\(k_{02}\)} (3.5,0);
\end{tikzpicture}
\end{center}
    \caption{A directed-cycle model that is not leak-interlacing.}
    \label{fig:cycle-4-leak-interlacing}
\end{figure}
\end{example}

An immediate consequence of Theorem~\ref{thm:main-cycle} concerns directed-cycle models with 
too many leaks (cf.~\cite[Theorem~6.1]{BM-2022}).
% --------------------
% COROLLARY -- TOO MANY LEAKS
% --------------------
\begin{corollary}[Directed-cycle models with too many leaks]
     \label{cor:cycle-3plus-leaks}
    If $\mathcal{M}= (G, \In, \Out, \Leak)$ is a directed-cycle model with 
    $\lvert \Leak \rvert \geq \lvert \In \rvert + \lvert \Out \rvert + 1$,
    then  $\mathcal{M}$ is unidentifiable.
    In particular, every directed-cycle model with 
    exactly $1$ input, exactly $1$ output, and $3$ or more leaks is unidentifiable.
\end{corollary}

The remainder of this section is dedicated to proving Theorem~\ref{thm:main-cycle}, as follows. Subsection~\ref{sec:cycle-prior-results} contains results on the coefficient map of directed-cycle models, which we use in the next subsections.  
In 
Subsection~\ref{sec:cycle-non-leak-interlacing}, 
we prove one implication of 
Theorem~\ref{thm:main-cycle} 
(Proposition~\ref{prop:non-leak-interlacing-summary} states that non-leak-interlacing models are unidentifiable);
while 
Subsection~\ref{sec:cycle-leak-interlacing} 
presents the reverse implication  
(Proposition~\ref{prop:leak-interlacing} states that leak-interlacing models are generically locally
identifiable).

\subsection{Coefficient maps of directed-cycle models}  \label{sec:cycle-prior-results}
This subsection presents useful results on the coefficient maps of directed-cycle models. 
We begin with a formula for the coefficient map of models with only one input and one output (and at least one leak), which is due to Gerberding, Obatake, and Shiu~\cite[Proposition~3.8]{GOS}:

\begin{lemma}[Coefficient map for $1$ input and $1$ output, \cite{GOS}] \label{lem:coefficient-map-cycle}
Let $n \geq 3$. 
Consider an $n$-compartment directed-cycle model 
$\mathcal{M}=
(G, \In, \Out, \Leak)$, with $In = \{1\}$, $Out = \{p\}$ (for some $1 \leq p \leq n$), and $Leak = \{\ell_1, \ell_2, \ldots , \ell_z\} \neq \varnothing$. 
Let  
$\ds\kappa := \prod_{i=2}^p k_{i,i-1}$ (if $p=1$, then $\ds \kappa := 1$), 
and let $k_{n+1,n}:= k_{1n}$. 
Then the coefficient map of $\mathcal{M}$ is the map 
$\msc: \R^{n+z} \rightarrow \R^{2n - p + 1}$ defined by:
{\small
\begin{align} \label{eq:coeff-map-cycle}   
(k_{21}, k_{32},\ldots, k_{1n},~ k_{0\ell_1}, k_{0\ell_2}, \ldots, k_{0\ell_z}) \mapsto 
    (e_1, e_2, \ldots, e_{n-1}, ~ 
    e_n-\displaystyle{\prod_{i=1}^{n} k_{i+1, i}},~ \kappa,~ 
    e_1^*\kappa,
    e_2^*\kappa,
    \ldots e_{n-p}^*\kappa)~,
\end{align}
}%
\noindent 
where $e_j$ and $e_j^*$ denote the $j^{th}$ elementary symmetric polynomial on the sets $E = \{k_{i+1,i} \ | \ i \notin Leak \} \ \cup \ \{k_{i+1,i} + k_{0i}\ | \ i \in  Leak \}$ and 
$E^* = \{k_{i+1,i} \ | \ p +1 \leq i \leq n, \ i \notin Leak \} \ \cup \ \{k_{i+1,i} + k_{0i}\ | \ p +1 \leq i \leq n, \ i \in  Leak \}$, respectively.  (If $p=n$, the set $E^*$ is empty.)
\end{lemma}

\begin{example}[Example~\ref{ex:leak-interlacing}, continued] \label{ex:coeff-map}
For the model in Figure~\ref{fig:cycle-running-example}, the coefficient map was shown earlier in~\eqref{eq:coeff-map-for-main-ex}, and it matches the description in~\eqref{eq:coeff-map-cycle}.
\end{example}

\begin{remark} \label{rem:notation}
Lemma~\ref{lem:coefficient-map-cycle} 
uses the notation $\kappa$ and $e^*_1$, which -- in the more general setting of the next result
(Proposition~\ref{prop:specific-coefficients}) --
will correspond to the terms $\kappa(1, p)$ and $e^*_1(1, p)$.
\end{remark}

\begin{notation} \label{notation:types-of-coefficients}
    We label coefficients in the coefficient map~\eqref{eq:coeff-map-cycle} in Lemma~\ref{lem:coefficient-map-cycle}, as follows:
    \begin{itemize}
        \item Type I: $e_1, e_2, \ldots, e_{n-1}$
        
        \item Type II: $e_n-\displaystyle{\prod_{i=1}^{n} k_{i+1, i}}$
        \item Type III: $\kappa$
        \item Type IV: $    e_1^*\kappa,
    e_2^*\kappa,
    \ldots e_{n-p}^*\kappa$ (if $p=n$, these Type~IV coefficients do not exist).
    \end{itemize}
\end{notation}

The next result uses Lemma~\ref{lem:coefficient-map-cycle} to describe some\footnote{Proposition~\ref{prop:specific-coefficients}(3) lists $e^{*}_1(i,j)$, but not the ``higher'' terms $e^{*}_2(i,j), e^{*}_3(i,j),$ and so on, because these higher terms need not exist and, moreover, only the term $e^{*}_1(i,j)$ is needed for later proofs.} coefficients of cycle models that may have more than one input and/or output.  

\begin{proposition}[Coefficients] \label{prop:specific-coefficients}
Assume $ n \geq 3$.
Consider an $n$-compartment directed-cycle model 
$\mathcal{M}=
(G, \In, \Out, \Leak)$, 
with 
$\In \neq \varnothing$, 
$\Out \neq \varnothing$, and
$\Leak \neq \varnothing$. 

Let $k_{n+1,n}:= k_{1n}$.          
    \begin{enumerate}
        \item The coefficient map of $\mathcal{M}$ contains the coefficients 
        $e_1, e_2, \ldots, e_{n-1},~
        e_n-\displaystyle{\prod_{i=1}^{n} k_{i+1, i}}
        $, where $e_j$ denotes the $j^{th}$ elementary symmetric polynomial on the set $E = \{k_{i+1,i} \ | \ i \notin Leak \} \ \cup \ \{k_{i+1,i} + k_{0i}\ | \ i \in  Leak \}$.
        \item For all $i \in \In$ and $j \in \Out$, the coefficient map of $\mathcal{M}$ contains the following coefficient,
                        which (if $i \neq j$) is the product of edge-parameters along the path from input $i$ to output $j$:
                \begin{align} \label{eq:product-i-to-j}
                         \kappa(i,j) \quad := \quad
                         \begin{cases}
                             k_{i+1,i} k_{i+2,i+1} \cdots k_{j,j-1} \quad & \mathrm{if}~ i \neq j\\
                             1 & \mathrm{if}~ i = j~,
                         \end{cases}
                \end{align}
                %$\kappa(i,j):= k_{i+1,i} k_{i+2,i+1} \cdots k_{j,j-1}$, 
                where the indices are taken mod $n$ (that is, $k_{n+1,n}:=k_{1n}$, $k_{n+2,n+1}:=k_{21}$, etc.).
    \item For all $i \in \In$ and $j \in \Out$ with $j \neq i-1$ (mod $n$), the coefficient map of $\mathcal{M}$ contains the coefficient $e^*_1(i,j) \kappa(i,j)$, 
                where: 
                \begin{align*}  
                e^*_1(i,j) ~ := ~ 
                \left(
                    \displaystyle\sum_{q \in \{j+1,j+2,\dots, i-1\} \smallsetminus \Leak}
                    k_{q+1,q}
                \right)
                +
                \left(
                    \displaystyle\sum_{q \in \{j+1,j+2,\dots, i-1\} \cap \Leak}
                    k_{q+1,q} + k_{0q}
                \right)      
                \end{align*}
                where the indices $j+1, j+2,\dots, i-1$ are taken mod $n$.
        \end{enumerate}
\end{proposition}

\begin{proof}
Let $\mathcal{M}=
(G, \In, \Out, \Leak)$ be a directed-cycle model with 
$\In \neq \varnothing$, 
$\Out \neq \varnothing$, and
$\Leak \neq \varnothing$. 
Let $i \in \In$ and $j \in \Out$, and consider the one-input, one-output model $\mathcal{M}(i,j):=(G, \{i\}, \{j\}, Leak)$.  
By Proposition~\ref{prop:coeff-many-in-or-out}, every coefficient of the coefficient map of 
$\mathcal{M}(i,j)$ is also a coefficient of $\mathcal{M}$.  So, it suffices to prove that the coefficients in (1)--(3) in Proposition~\ref{prop:specific-coefficients} are coefficients of
$\mathcal{M}(i,j)$.  

The plan for the remainder of the proof is 
as follows.
First, we relabel the compartments 
of $\mathcal{M}(i,j)$
so that the input is in the first compartment.  Next, we use Lemma~\ref{lem:coefficient-map-cycle} to obtain the coefficients.  Finally, we return these coefficients to the original labels. 

Let $\mathcal{M}(i,j)' = (G, \{1\}, \{j-i+1\}, \Leak')$ denote the cycle model obtained from 
$\mathcal{M}(i,j)$ by relabeling each compartment $q$ by $q-i+1$ (mod $n$).  By Lemma~\ref{lem:coefficient-map-cycle}, the following are coefficients of $\mathcal{M}(i,j)'$:

    \begin{enumerate}[label=(\roman*)] %[(i)]
        \item $e_1',e_2',\dots, e_{n-1}',$
        $e_n' - \displaystyle \prod_{i=1}^n k_{i+1,i}$, where $e_j'$ is the $j^{th}$ elementary symmetric polynomial on the set $E' = \{k_{i+1,i} \ | \ i \notin Leak' \} \ \cup \ \{k_{i+1,i} + k_{0i}\ | \ i \in  Leak' \}$;
        \item $\kappa = \displaystyle 
        \prod_{q=2}^{j-i+1} k_{q,q-1}$.      
    \end{enumerate} 
Lemma~\ref{lem:coefficient-map-cycle} also implies that, if $j-i+1 \neq n$ (mod $n$), then the following is a coefficient of $\mathcal{M}(i,j)'$:
    \begin{enumerate}%[(iii)]
        \item[(iii)] $e_1^* \kappa$, where 
                       \begin{align*}  
                e^*_1 ~ := ~ 
                \left(
                    \displaystyle\sum_{q \in \{p+1,p+2,\dots, n \} \smallsetminus \Leak'}
                    k_{q+1,q}
                \right)
                +
                \left(
                    \displaystyle\sum_{q \in \{p+1,p+2,\dots, n \} \cap \Leak'}
                    k_{q+1,q} + k_{0q}
                \right)      ~.
                \end{align*}
    \end{enumerate}

From the coefficients in (i)--(iii), we obtain coefficients of $\mathcal{M}(i,j)$ by undoing the relabeling we did to obtain  $\mathcal{M}(i,j)'$.  It is now straightforward to check that the coefficients resulting from relabeling those in (i), (ii), and (iii) are precisely those in (1), (2), and (3), respectively.
\end{proof}

The following result records some partial derivatives of one of the coefficients described in Proposition~\ref{prop:specific-coefficients}; the proof is elementary and so is omitted.

\begin{lemma} \label{lem:partial-deriv}
Assume that $1 \leq i < j \leq n$, where $n \geq 3$.  Consider the coefficient $\kappa(i,j) =
k_{i+1,i} k_{i+2,i+1} \cdots k_{j,j-1}$,
from Proposition~\ref{prop:specific-coefficients}~(2).  For $1 \leq \ell \leq n$, we have the following partial derivatives:
\begin{align*}
    \frac{\partial \kappa(i,j)}{\partial k_{0 \ell}} ~=~0
    \quad \quad 
    \mathrm{and}
    \quad \quad 
    \frac{\partial \kappa(i,j)}{\partial k_{\ell+1, \ell}}
        ~=~
        \begin{cases}
                             \frac{\kappa(i,j)}{k_{\ell+1,\ell}} \quad & \mathrm{if}~ i \leq \ell \leq j-1 \\
                             0 & \mathrm{otherwise}~.
                         \end{cases}
\end{align*}
\end{lemma}
 
We end this subsection with a result in the same vein as Lemma~\ref{lem:partial-deriv}.  A key difference is that we restrict our attention to models with only one input and one output.  As with Lemma~\ref{lem:partial-deriv}, the proof is straightforward and hence is omitted.

\begin{lemma} \label{lem:partial-deriv-1-input}
Assume that $1 \leq p \leq n$, where $n \geq 3$.  
Consider the coefficients~\eqref{eq:coeff-map-cycle} in Lemma~\ref{lem:coefficient-map-cycle} for an $n$-compartment directed-cycle model 
$(G, \In, \Out, \Leak)$ with $In = \{1\}$ and $Out = \{p\}$.  
The partial derivatives of the coefficients satisfy the following: 
    \begin{enumerate}[label=(\Roman*)]%[(I)]
        % --------------------
        % TYPE I    
        % --------------------
        \item For $1 \leq i \leq n$, if $\ell \in \Leak$, then
        \begin{align*}
    \frac{\partial e_i }{\partial k_{0 \ell}} ~=~
    \frac{\partial e_i }{\partial k_{\ell+1,+\ell}}~.         
    \end{align*}        
        For $1 \leq i \leq n$, if $\ell, q \in \Leak$, with $\ell \neq q$, then
            \begin{align} \label{eq:lem-difference-partials}
            \frac{\partial e_i }{\partial k_{0 q}}
            -
            \frac{\partial e_i }{\partial k_{0 \ell}} 
            ~=~
            (k_{0 \ell} + k_{\ell+1,\ell} - k_{0q} - k_{q+1,q} ) \overline{e}_{i-2}
            ~,
            \end{align}  
        where, for $j \geq 1$, we let $\overline{e}_{j}$ denote the $j$-th elementary symmetric polynomial on the set 
        $E \smallsetminus \{k_{0 \ell}+k_{\ell +1, \ell},~ k_{0q}+k_{q+1,q} \} =  \{k_{i+1,i} \ | \ i \notin Leak \} \ \cup \ \{k_{i+1,i} + k_{0i}\ | \ i \in  Leak %\smallsetminus \{ \ell, q\} 
        ,~ i \neq \ell,~ i \neq q
        \}$; and let 
                $\overline{e}_{-1}:=0$ and 
        $\overline{e}_{0}:=1$.
        % --------------------
        % TYPE II   
        % --------------------
        \item 
        If $\ell \in \Leak$, then 
        \begin{align*}
        \left(
        \frac{\partial }{\partial k_{\ell+1,\ell}}
        -
        \frac{\partial }{\partial k_{0 \ell}} 
        \right)
        \left[e_n-\displaystyle{\prod_{i=1}^{n} k_{i+1, i}} \right]
        ~=~
            -
        \displaystyle{\prod_{1 \leq i \leq n,~ i\neq \ell} k_{i+1, i}}~.
        \end{align*}        
        If $\ell, q \in \Leak$, with $\ell \neq q$, then
            \begin{align} \label{eq:lem-difference-partials-type-II}
            \left(
            \frac{\partial  }{\partial k_{0 q}}
            -
            \frac{\partial  }{\partial k_{0 \ell}} 
            \right)
            \left[e_n-\displaystyle{\prod_{i=1}^{n} k_{i+1, i}} \right]
            ~=~
            (k_{0 \ell} + k_{\ell+1,\ell} - k_{0q} - k_{q+1,q} ) \overline{e}_{n-2}
            ~,
            \end{align}  
        where $\overline{e}_{j}$ is as defined in part~(I) above.
        % --------------------
        % TYPE III
        % --------------------
        \item For $1 \leq \ell \leq n$, we have:
    \begin{align*}
    \frac{\partial \kappa}{\partial k_{0 \ell}} ~=~0
    \quad \quad 
    \mathrm{and}
    \quad \quad 
    \frac{\partial \kappa}{\partial k_{\ell+1, \ell}}
        ~=~
        \begin{cases}
                             \frac{\kappa}{k_{\ell+1,\ell}} \quad & \mathrm{if}~ 1 \leq \ell \leq p-1 \\
                             0 & \mathrm{otherwise}~.
                         \end{cases}
    \end{align*}
        % --------------------
        % TYPE IV 
        % --------------------
        \item 
        For $1 \leq i \leq n-p$, we have: 
            \begin{enumerate}
                \item If $\ell \in \Leak$ with $\ell \geq p+1$, then
                    \begin{align*}
                        \frac{\partial [e_i^*\kappa]}{\partial k_{\ell+1,\ell}}
                        ~=~
                        \frac{\partial [e_i^*\kappa]}{\partial k_{0 \ell}}~.
                    \end{align*}
                \item 
                $\frac{\partial [e_i^*\kappa] }{\partial k_{p+1,p}} = 0$;
                and, if $\ell \in \Leak$ with $1 \leq \ell \leq p$, then
                    \begin{align*}
                        \frac{\partial [e_i^*\kappa]}{\partial k_{0 \ell}}
                        ~=~
                        0
                        ~.
                    \end{align*}
                \item If $1 \leq \ell \leq p-1$, then
                    \begin{align*}
                        \frac{\partial [e_i^*\kappa] }{\partial k_{\ell+1,\ell}}
                        ~=~
                        e_i^* \frac{\kappa}{ k_{\ell+1, \ell}}~.
                    \end{align*}            \end{enumerate}
    \end{enumerate}
\end{lemma}

The results in this subsection are used in the remainder of this section.   Specifically, we use 
Lemma~\ref{lem:coefficient-map-cycle} in Subsection~\ref{sec:cycle-non-leak-interlacing}, and use Proposition~\ref{prop:specific-coefficients} and Lemma~\ref{lem:partial-deriv} to prove the main result of Subsection~\ref{sec:cycle-leak-interlacing}
(namely, Proposition~\ref{prop:leak-interlacing}).
Also, Lemma~\ref{lem:partial-deriv-1-input} is used in Subsection~\ref{sec:cycle-non-leak-interlacing} and later in Section~\ref{sec:SL}.

\subsection{Non-leak-interlacing models} \label{sec:cycle-non-leak-interlacing}
The aim of this subsection is to prove that non-leak-interlacing cycle models are unidentifiable (Proposition~\ref{prop:non-leak-interlacing-summary}), which is one implication of 
Theorem ~\ref{thm:main-cycle}.

\begin{proposition}[Non-leak-interlacing models] \label{prop:non-leak-interlacing-summary}
If $\mathcal{M}$ is a directed-cycle model  
that is non-leak-interlacing, then $\mathcal{M}$ is unidentifiable.
\end{proposition} 

We prove Proposition~\ref{prop:non-leak-interlacing-summary} 
in two steps. 
We first prove the result for models with one input and one output (Lemma~\ref{lem:unIdentifiable-before-output}).  
The second step (at the end of this subsection) is to use Lemma~\ref{lem:unIdentifiable-before-output} to prove the general case.

In our proofs of 
Lemma~\ref{lem:unIdentifiable-before-output}
and Proposition~\ref{prop:non-leak-interlacing-summary},
the key idea is to show a linear dependence among certain columns of the Jacobian matrix of the coefficient map.
This linear dependence is illustrated in the following example.

\begin{example}[Example~\ref{ex:not-leak-interlacing}, continued] \label{ex:show-dependence} 
    Recall that the 
    $4$-compartment 
    cycle model $\mathcal{M}=
(G, \In, \Out, \Leak)$ with 
$\In = \{1\}$, $\Out = \{3\}$, and $\Leak = \{1, 2\}$
is non-leak-interlacing. 
Consider the coefficient map of $\mathcal{M}$ obtained by equation~\eqref{eq:coeff-map-cycle} in Lemma \ref{lem:coefficient-map-cycle},  
and let $J$ denote the resulting 
$(6 \times 6)$ Jacobian matrix.  
The following matrix is formed by the columns of $J$ that correspond to 
the parameters 
 $k_{21}$ and 
$ k_{01}$ (in that order):
\begin{align} \label{eq:2-columns-of-J}
        \begin{pmatrix}
            1 & 
            1 & \\
            k_{32} + k_{02} + k_{43} + k_{14}& 
            k_{32} + k_{02} + k_{43} + k_{14} & \\
           k_{32}k_{43} + k_{02}k_{43} + k_{32}k_{14} + k_{02}k_{14} + k_{43}k_{14}& 
             k_{32}k_{43} + k_{02}k_{43} + k_{32}k_{14} + k_{02}k_{14} + k_{43}k_{14} & \\
           k_{02}k_{43}k_{14} &
           k_{32}k_{43}k_{14}+ k_{02}k_{43}k_{14} &\\
            k_{32} & 
            0 & \\
           k_{14}k_{32}  & 
           0  & 
        \end{pmatrix}~.
\end{align}
Similarly, the columns of $J$ corresponding to 
 $k_{32}$ and 
$ k_{02}$ form the following matrix:
\begin{align} \label{eq:2-more-columns-of-J}
        \begin{pmatrix}
           1 & 
           1 \\
           k_{21} + k_{01} + k_{43} + k_{14}& 
            k_{21} + k_{01} + k_{43} + k_{14} \\
                       k_{21}k_{43} + k_{01}k_{43} + k_{21}k_{14} + k_{01}k_{14} + k_{43}k_{14} & 
           k_{21}k_{43} + k_{01}k_{43} + k_{21}k_{14} + k_{01}k_{14} + k_{43}k_{14} \\
                      k_{01}k_{43}k_{14} &
           k_{21}k_{43}k_{14} + k_{01}k_{43}k_{14} &
           \\
                       k_{21} & 
           0 \\
                      k_{14}k_{21}  & 
           0 
        \end{pmatrix}~.
\end{align}
Observe that the four columns of the matrices~(\ref{eq:2-columns-of-J}--\ref{eq:2-more-columns-of-J}) --
which we denote by $\vec{J}_{k_{21}}$, 
$\vec{J}_{k_{01}}$, 
$\vec{J}_{k_{32}}$, and
$\vec{J}_{k_{02}}$,  
 respectively --
 satisfy the following linear dependence:
 \begin{align} \label{eq:linear-dependence-in-example}
    k_{21} \left( \vec{J}_{k_{01}} - \vec{J}_{k_{21}} \right)
    \quad = \quad
        k_{32} \left( \vec{J}_{k_{02}} - \vec{J}_{k_{32}} \right)
        \quad = \quad
        \begin{pmatrix}
            0 \\
            0 \\
            0 \\
            k_{21} k_{32} k_{43} k_{14}\\
            - k_{21} k_{32}\\
            - k_{21} k_{32} k_{14}
        \end{pmatrix}
        ~.
\end{align}
This dependence shows (by Proposition~\ref{prop:rank-matrix-for-identifiability}) 
that $\mathcal{M}$ is unidentifiable.
Our next result asserts that a linear dependence like the one in~\eqref{eq:linear-dependence-in-example} always occurs when a non-exceptional, one-input, one-output model is non-leak-interlacing; see equation~\eqref{eq:lin-depen-col} below and Remark~\ref{rem:non-leak-interlacing}.
\end{example}

\begin{lemma}[Linear dependence for one-input, one-output models] \label{lem:unIdentifiable-before-output}
    Let $n \geq 3$.
    Let $\mathcal{M}=
(G, \In, \Out, \Leak)$ 
    be an $n$-compartment directed-cycle model with $In = \{1\}$ and $Out = \{p\}$ (for some $1 \leq p \leq n$).
    Assume that
    there exist $a, b \in Leak$
    such that 
    $1 \leq  a < b  \leq p-1$ 
    or $p \leq a < b \leq n$.
Let $J$ denote the 
%$(2n-p+1) \times (n+ |\Leak|)$
Jacobian matrix of the coefficient map of $\mathcal{M}$.
%(where $t\geq 2$ is the number of leaks).  
Let $\vec{J}_{k_{0a}}$, 
$\vec{J}_{k_{a+1, a}}$, 
$\vec{J}_{k_{0b}}$, and
$\vec{J}_{k_{b+1, b}}$ denote the columns of $J$ that correspond to %to the partial derivative with respect to 
the parameters
 $k_{0a}$, 
$ k_{a+1, a}$, 
$k_{0b}$,
and $k_{b+1, b}$,  
 respectively.
Then these four column vectors satisfy the 
following linear dependence: 
\begin{align}
    \label{eq:lin-depen-col}        
    k_{a+1,a} \left( \vec{J}_{k_{0a}} - \vec{J}_{k_{a+1, a}} \right)
\quad = \quad
k_{b+1,b} \left( \vec{J}_{k_{0b}} - \vec{J}_{k_{b+1, b}} \right)~.
   \end{align}
In particular, $\mathcal{M}$ is unidentifiable.
\end{lemma}

\begin{remark} \label{rem:non-leak-interlacing}
    In the context of Lemma~\ref{lem:unIdentifiable-before-output}, the assumption on the leaks (that is, the existence of $a, b \in Leak$
    with 
    $1 \leq  a < b  \leq p-1$ 
    or $p \leq a < b \leq n$)
    is equivalent to the condition that the model $\mathcal{M}$
    is not an exceptional model  
    and is non-leak-interlacing.
\end{remark}

\begin{proof}[Proof of Lemma~\ref{lem:unIdentifiable-before-output}] 
Consider $a, b \in Leak$, as in the statement of the lemma.  
The coefficient map for $\mathcal{M}$ is given in equation~\eqref{eq:coeff-map-cycle} in Lemma \ref{lem:coefficient-map-cycle}. 
In particular, we use the notation 
introduced there: 
\begin{enumerate}[label=(\roman*)]%[(i)]
    \item $\kappa:=k_{21} k_{32} \dots k_{p,p-1}$ is the product of edge-parameters along the path from input to output,
    unless $p=1$, in which case we define $\kappa:=1$,
    \item $e_j$ is the $j^{th}$ elementary symmetric polynomial on the set
    $E = \{k_{i+1,i} \ | \ i \notin Leak \} \ \cup \ \{ k_{i+1,i} + k_{0i}\ | \ i \in  Leak \}$,
    and
    \item $e_j^*$ is the $j^{th}$ elementary symmetric polynomial on a set, namely $E^*$, that involves 
    parameters of edges and leaks \uline{after} the output (compartment-$p$).
\end{enumerate}

We consider two cases.

{\bf Case 1:} The leaks $a$ and $b$ satisfy the inequalities
    $1 \leq  a < b  \leq p-1$.

Consider the following claim.

\noindent {\bf Claim:}
For $\ell \in \{a,b\}$, the following equalities of vectors hold\footnote{An instance of this claim appears in equation~\eqref{eq:linear-dependence-in-example}
in Example~\ref{ex:show-dependence}, with $a=1$ and $b=2$.}: 

\begin{align} \label{eq:col-vec-reln}
k_{\ell+1,\ell} \left( \vec{J}_{k_{0\ell}} - \vec{J}_{k_{\ell+1, \ell}} \right)
    \quad = \quad
    k_{\ell+1,\ell}
    \begin{pmatrix}
        \left( \frac{\partial}{\partial k_{0 \ell }} - \frac{\partial}{\partial k_{\ell +1, \ell}} \right)  [e_1] \\
        \left( \frac{\partial}{\partial k_{0 \ell }} - \frac{\partial}{\partial k_{\ell +1, \ell}} \right)  [e_2] \\
    \vdots \\
        \left( \frac{\partial}{\partial k_{0 \ell }} - \frac{\partial}{\partial k_{\ell +1, \ell}} \right)  [e_{n-1}] \\
        \left( \frac{\partial}{\partial k_{0 \ell }} - \frac{\partial}{\partial k_{\ell +1, \ell}} \right)  
        [ e_{n} - \displaystyle{\prod_{i=1}^{n} k_{i+1, i}} ] \\
        \left( \frac{\partial}{\partial k_{0 \ell }} - \frac{\partial}{\partial k_{\ell +1, \ell}} \right)  [\kappa] \\
        \left( \frac{\partial}{\partial k_{0 \ell }} - \frac{\partial}{\partial k_{\ell +1, \ell}} \right)  [e_1^* \kappa] \\
        \left( \frac{\partial}{\partial k_{0 \ell }} - \frac{\partial}{\partial k_{\ell +1, \ell}} \right)  [e_2^* \kappa] \\
    \vdots \\
        \left( \frac{\partial}{\partial k_{0 \ell }} - \frac{\partial}{\partial k_{\ell +1, \ell}} \right)  [e_{n-p}^* \kappa] \\
    \end{pmatrix}
    ~ = ~
    \begin{pmatrix}
    0  \\
    0  \\
    \vdots  \\
    0  \\
     \displaystyle{\prod_{i=1}^{n} k_{i+1, i}}
    \\
    - \kappa
    \\
    - e_1^*\kappa 
 \\ 
    - e_2^*\kappa 
     \\
    \vdots
    \\
    - e_{n-p}^*\kappa  
      \\
    \end{pmatrix}~.
\end{align}
In~\eqref{eq:col-vec-reln}, the right-most vector does {\em not} depend on $\ell$, and so the claim  implies the desired equality~\eqref{eq:lin-depen-col}. 

Also, observe that
the first equality 
in~\eqref{eq:col-vec-reln} 
is by definition.  Accordingly, the remainder of the proof (for Case~1) is dedicated to proving the second equality in~\eqref{eq:col-vec-reln}.  We do so by considering each entry -- or certain subsets of entries -- individually (these subsets correspond to the coefficients of Types I--IV, as in Notation~\ref{notation:types-of-coefficients}).  Assume in what follows (for this case) that we have fixed $\ell \in \{a,b\}$.

\noindent
{\bf First $n-1$ entries of~\eqref{eq:col-vec-reln}.}
For $1 \leq j \leq n-1$, the following equality follows from 
Lemma~\ref{lem:partial-deriv-1-input}(I):
\begin{align*} %\label{eq:partial-elemen}
     \left( \frac{\partial}{\partial k_{0 \ell }} - \frac{\partial}{\partial k_{\ell +1, \ell}} \right)  [e_j] ~=~ 0~.
\end{align*}

\noindent
{\bf The $n$-th entry of~\eqref{eq:col-vec-reln}.}
By Lemma~\ref{lem:partial-deriv-1-input}(II), we obtain the first equality here (and the second is straightforward):
\begin{align} \label{eq:partial-elemen-minus-cycle}
    k_{\ell+1,\ell}
     \left( \frac{\partial}{\partial k_{0 \ell }} - \frac{\partial}{\partial k_{\ell +1, \ell}} \right)  
     \left[ e_n - \displaystyle{\prod_{i=1}^{n} k_{i+1, i}} \right] ~=~    
     k_{\ell+1,\ell} \displaystyle{\prod_{1 \leq i \leq n,~ i\neq \ell} k_{i+1, i}}
         ~=~
    \displaystyle{\prod_{i=1}^{n} k_{i+1, i}}~.
\end{align}
We remark that the above equalities did not use the assumption underlying this case ($1 \leq  a < b  \leq p-1$), only the fact that $\ell$ is a leak.  This will be useful when we consider Case~2.

\noindent
{\bf The $n+1$ entry of~\eqref{eq:col-vec-reln}.}
Lemma~\ref{lem:partial-deriv-1-input}(III) yields the following (the assumption $1 \leq a < b \leq p-1$ is used here to assure that $1 \leq \ell \leq p-1 $):
\begin{align} \label{eq:kappa-partial}
    \frac{\partial}{\partial k_{0 \ell }} [\kappa] ~=~0
    \quad 
    {\rm and}
    \quad 
        k_{\ell+1,\ell}
    \frac{\partial}{\partial k_{\ell + 1, \ell }} [\kappa] ~=~
    \kappa~.
\end{align}
The entry in~\eqref{eq:col-vec-reln} corresponding to $\kappa$ now follows.

\noindent
{\bf Last $n-p$ entries of~\eqref{eq:col-vec-reln}.}
Parts (b) and (c) of Lemma~\ref{lem:partial-deriv-1-input}(IV) imply the following (the inequalities $1 \leq \ell \leq p-1 $ are again used here):
\begin{align*} %\label{eq:final-terms}
    \notag
k_{\ell+1, \ell}
    \left( \frac{\partial}{\partial k_{0 \ell }} - \frac{\partial}{\partial k_{\ell +1, \ell}} \right)  [e_{j}^* \kappa]
     ~&=~ - e_j^* \kappa ~.
     \notag
\end{align*}

We conclude that the claim holds, and so the proposition is true for Case~1.

{\bf Case 2:} The leaks $a$ and $b$ satisfy $p \leq a < b \leq n$.
In this case, it was shown in \cite[proof of Theorem 3.10]{GOS} that the vectors  
$\vec{J}_{k_a+1, a} - \vec{J}_{k_{0a}}$ and $\vec{J}_{k_b+1, b}-\vec{J}_{k_{0b}}$ are zero in all coordinates except the one corresponding to the 
(Type II) 
coefficient $e_n - \prod^n_{i=1}k_{i+1, i}$. So, only one coordinate of the desired equality~\eqref{eq:lin-depen-col} remains to be checked.
In fact, we already saw that the equalities in~\eqref{eq:partial-elemen-minus-cycle}, from the prior case, hold also in Case~2.
This implies that, as desired, the final coordinate of~\eqref{eq:lin-depen-col}
%and so~\eqref{eq:lin-depen-col} 
holds for Case~2.

Finally, the linear dependence~\eqref{eq:lin-depen-col} implies, by Proposition~\ref{prop:rank-matrix-for-identifiability}, 
that $\mathcal{M}$ is unidentifiable.
\end{proof}

We end this section by using Lemma~\ref{lem:unIdentifiable-before-output}, which considered the case of models with one input and one output, 
to prove the general case (Proposition~\ref{prop:non-leak-interlacing-summary}).

\begin{proof}[Proof of Proposition~\ref{prop:non-leak-interlacing-summary}]
Consider an $n$-compartment directed-cycle model
$\mathcal{M} = (G, \In, \Out, \Leak)$, with  
 $Leak = \{\ell_1, \ell_2, ..., \ell_z \}$, where $1 \leq \ell_1 < \ell_2 < \dots < \ell_z \leq n$.  
Assume that $\mathcal M$ is non-leak-interlacing.  
If $\mathcal{M}$ is an exceptional model (Definition~\ref{def:exceptional-model}), then part~(2) of Proposition~\ref{prop:cycle-1-in-1-out} implies that $\mathcal{M}$ is unidentifiable.

For the rest of the proof, we consider the remaining case (when $\mathcal{M}$ is not an exceptional model).   
This case, by Definition~\ref{def:interlacing}, is when 
$\lvert Leak \rvert \geq 2$, and, 
additionally,
        there exists $\alpha \in \{1,2,\dots,z \}$
        such that 
        $\{\ell_{\alpha}+1, \ell_{\alpha}+2, \dots, \ell_{\alpha+1} \} \cap (\In \cup \Out) = \varnothing$. 
Let $a:=\ell_{\alpha}$ and $b:=\ell_{\alpha+1}$.        

In the remainder of the proof, we use the following notation.  For a model $\mathcal{N}$, we let $J(\mathcal{N})$ denote the Jacobian matrix of the coefficient map of $\mathcal{N}$.
Also, given a parameter
 $k_{qr}$ of $\mathcal{N}$, let 
$\overrightarrow{J(\mathcal{N})}_{k_{qr}}$ denote the column of $J(\mathcal{N})$ that corresponds to $k_{qr}$.

By Proposition~\ref{prop:rank-matrix-for-identifiability},
it suffices to show the following linear dependence  (cf.\ equation~\eqref{eq:lin-depen-col}):
\begin{align}
    \label{eq:lin-depen-col-general}        
    k_{a+1,a} \left( \overrightarrow{J(\mathcal{M})}_{k_{0a}} - \overrightarrow{J(\mathcal{M})}_{k_{a+1, a}} \right)
\quad = \quad
k_{b+1,b} \left( \overrightarrow{J(\mathcal{M})}_{k_{0b}} - \overrightarrow{J(\mathcal{M})}_{k_{b+1, b}} \right)~.
   \end{align}

Next, by Proposition~\ref{prop:coeff-many-in-or-out}, each coefficient of the coefficient map of $\mathcal{M}$ is the coefficient of some model $\mathcal{M}_{ij} := 
(G, \{i\}, \{j\}, \Leak)$, 
where $i\in \In$ and $j \in \Out$. 
Hence, showing the equality~\eqref{eq:lin-depen-col-general} reduces to showing the following, for all $i\in \In$ and $j \in \Out$:
\begin{align}
    \label{eq:lin-depen-col-general-i-j}         k_{a+1,a} \left( \overrightarrow{J(\mathcal{M}_{ij})}_{k_{0a}} - \overrightarrow{J(\mathcal{M}_{ij})}_{k_{a+1, a}} \right)
\quad = \quad
k_{b+1,b} \left( \overrightarrow{J(\mathcal{M}_{ij})}_{k_{0b}} - \overrightarrow{J(\mathcal{M}_{ij})}_{k_{b+1, b}} \right)~.
\end{align}

We aim to use Lemma~\ref{lem:unIdentifiable-before-output} to prove equation~\eqref{eq:lin-depen-col-general-i-j}.  Accordingly, we relabel the compartments of $\mathcal{M}_{ij}$ so that the unique input is in compartment-$1$ (rather than compartment-$i$). 

Call the resulting model $\mathcal{M}_{ij}'$, and let $a'$, $b'$, and $p$ denote the relabelings of $a$, $b$, and $j$, respectively.  

By construction of $a$ and $b$ (in $\mathcal{M}$), the input in compartment-$1$ (in $\mathcal{M}_{ij}'$) is not in $\{a'+1,a'+2,\dots, b'\}$.  Hence:
\begin{align} \label{eq:inequalities-proof-1}
    1 \leq a' < b' \leq n~.
\end{align}
Similarly, the output in compartment-$p$
(in $\mathcal{M}_{ij}'$)
is also not in 
$\{a'+1,a'+2,\dots, b'\}$, so:
\begin{align} \label{eq:inequalities-proof-2}
    \mathrm{ either}
    \quad
    p \leq a'
    \quad
    \mathrm{ or}
    \quad
    b' \leq p-1~,
\end{align}
where we are also using the inequalities~\eqref{eq:inequalities-proof-1} here. 

The inequalities~(\ref{eq:inequalities-proof-1}--\ref{eq:inequalities-proof-2}) allow us to apply Lemma~\ref{lem:unIdentifiable-before-output}, which implies the following:
\begin{align}
    \label{eq:lin-depen-col-general-i-j-relabeled}         k_{a'+1,a'} \left( \overrightarrow{J(\mathcal{M}_{ij}')}_{k_{0a'}} - \overrightarrow{J(\mathcal{M}_{ij}')}_{k_{a'+1, a'}} \right)
\quad = \quad
k_{b'+1,b'} \left( \overrightarrow{J(\mathcal{M}_{ij}')}_{k_{0b'}} - \overrightarrow{J(\mathcal{M}_{ij}')}_{k_{b'+1, b'}} \right)~.
\end{align}
Finally, by relabeling back to the original labels of the compartments, we obtain from~\eqref{eq:lin-depen-col-general-i-j-relabeled} the desired equality~\eqref{eq:lin-depen-col-general-i-j}.
\end{proof}

\subsection{Leak-interlacing models}  \label{sec:cycle-leak-interlacing}
In the prior section, we proved one implication of Theorem~\ref{thm:main-cycle}. 
We prove the remaining implication in this subsection (Proposition~\ref{prop:leak-interlacing} below).
The proof of Proposition~\ref{prop:leak-interlacing} requires two definitions, as follows.

\begin{definition} \label{def:family}
 A directed-cycle model $\mathcal{M}$ is in the {\bf exceptional family} if 
 there exists an exceptional model $\mathcal{M}'$ (Definition~\ref{def:exceptional-model}) 
 such that $\mathcal{M}$ is obtained from $\mathcal{M}'$ 
 by adding at least one input and/or at least one output.
\end{definition}

Recall that the idea behind leak-interlacing is that between any two leaks, there exists at least one input or output.  The next definition strengthens this requirement so that 
between any two leaks, there exists \uline{exactly} one input or output (and not both).

\begin{definition} \label{def:minimal}
    Let $n \geq 3$.  
    An $n$-compartment directed-cycle model $\mathcal{M} = (G, \In, \Out, \Leak)$ 
    with $\lvert Leak \rvert \geq 2$
    is {\bf minimally leak-interlacing} if the following hold:
    \begin{enumerate}
        \item $\mathcal{M}$ is leak-interlacing, and 
        \item if  $Leak = \{\ell_1, \ell_2, ..., \ell_z \}$, where $1 \leq \ell_1 < \ell_2 < \dots < \ell_z \leq n$, 
        then 
        for all $\alpha \in \{1,2,\dots,z \}$, 
        exactly one of the following holds: 
            \begin{enumerate}
                \item 
                $ \lvert \{\ell_{\alpha}+1, \ell_{\alpha}+2, \dots, \ell_{\alpha+1} \} \cap \In  \rvert = 1 $ 
                and 
                $ \lvert \{\ell_{\alpha}+1, \ell_{\alpha}+2, \dots, \ell_{\alpha+1} \} \cap \Out  \rvert = 0 $, or 
               \item 
                $ \lvert \{\ell_{\alpha}+1, \ell_{\alpha}+2, \dots, \ell_{\alpha+1} \} \cap \In  \rvert = 0 $ 
                and 
                $ \lvert \{\ell_{\alpha}+1, \ell_{\alpha}+2, \dots, \ell_{\alpha+1} \} \cap \Out  \rvert = 1 $.                 
            \end{enumerate}
    \end{enumerate}
\end{definition}

The next result states that, from certain leak-interlacing models, we can obtain a minimally leak-interlacing model by removing inputs and/or outputs.

\begin{lemma}[Minimally leak-interlacing submodels] \label{lem:minimal}
    Let $\mathcal{M}= (G, \In, \Out, \Leak)$ be a leak-interlacing directed-cycle model with 
    $\lvert \In \rvert \geq 1$, 
    $\lvert \Out \rvert \geq 1$, and 
    $\lvert \Leak \rvert \geq 2$. 
    If $\mathcal{M}$ is \uline{not} in the exceptional family, 
    then there exist subsets $\In' \subseteq \In $ and $\Out' \subseteq \Out$, with 
    $\lvert \In' \rvert \geq 1$, 
    $\lvert \Out' \rvert \geq 1$,
     and
    $\In' \cap \Out' = \varnothing$,
    such that the model $\mathcal{M}'= (G, \In', \Out', \Leak)$ is minimally leak-interlacing.
\end{lemma}

\begin{proof}
The model $\mathcal{M}$ is leak-interlacing, so, for every directed path 
(in the sense of Remark~\ref{rem:path})
$(\ell_\alpha +1)$ to $\ell_{\alpha +1}$, where $\alpha \in [z]$, there must be at least one input or at least one  output. 
Also, at least two such paths exist, because $\lvert \Leak \rvert \geq 2$. 
We consider two cases.

{\bf Case 1:} Each path contains at least one input. In this case, choose a path with at least one output (which exists because $\lvert \Out \rvert \geq 1$), and place this output in $Out'$.  
For each remaining path, pick one input and place it in $In'$.

{\bf Case 2:} There is a path that does not contain inputs. 
In this case, pick an output from this path and place it in $\Out'$.  Next, pick a path that contains an input (which exists because $\lvert \In \rvert \geq 1$), and choose one such input to place in $\In'$.  For each remaining path, pick an input or an output and place it in $\In'$ or $\Out'$, respectively. 

By construction, each path contains either exactly one input or exactly one output.  Additionally,
$\lvert \In' \rvert \geq 1$, 
    $\lvert \Out' \rvert \geq 1$,
    and
    $\In' \cap \Out' = \varnothing$. 

Next, we claim that the resulting model $\mathcal{M}'= (G, \In', \Out', \Leak)$ is not an exceptional model.  To see this, we consider three cases.  
% Case A
First, consider the case when $|\Leak | \geq 3$.  By definition, exceptional models have only two leaks, so $\mathcal{M}'$ is not an exceptional model.  
% Case B
Next, consider the case when $|\Leak | =2$ and, additionally, $| \In| \geq 2$ or $|\Out| \geq 2$.  In this case, $\mathcal{M}$ is obtained from $\mathcal{M}'$ by adding at least one input and/or at least one output.  By assumption, $\mathcal{M}$ is not  in the exceptional family and so (by definition) $\mathcal{M}'$ is not an exceptional model.
% Case C
We consider the final case, which is when $|\Leak | =2$ and $| \In| =|\Out| = 1$.
In this case,  $\mathcal{M}' = \mathcal{M}$ and (by assumption) $\mathcal{M}$ is leak-interlacing, so (by definition) $\mathcal{M}'$ cannot be an exceptional model.

We conclude that $\mathcal{M}'$
is minimally leak-interlacing.  
\end{proof}

Recall that adding inputs and/or outputs preserves identifiability (Proposition~\ref{prop:add-in-out}).  So, in the context of Lemma~\ref{lem:minimal}, 
if a minimally leak-interlacing model $\mathcal{M}'$ is generically locally
identifiable, so is $\mathcal{M}$.  We use this fact in the proof of the next result.

Additional ideas that underlie  
the proof of the next result are illustrated in the following example.

\begin{example}[Example~\ref{ex:coeff-map}, continued] \label{ex:ideas-in-big-proof}
For the model in Figure~\ref{fig:cycle-running-example}, 
we computed earlier, in~\eqref{eq:coeff-map-for-main-ex},  
 the coefficient map 
$\msc : \R^{5} \rightarrow \R ^5$, 
and its
Jacobian matrix
(in Example~\ref{ex:jacobian-matrix}). Next, we consider the following modified version of the coefficent map: 
$\mathsf{d} : \R^{5} \rightarrow \R ^5$, given by $(\mathsf{d}_1,\mathsf{d}_2,\mathsf{d}_3,\mathsf{d}_4,\mathsf{d}_5):=(\msc_1,\msc_2,\msc_5/\msc_4, \msc_4, \msc_3)$. 
The Jacobian matrix of $\mathsf{d} $ 
--
where the columns correspond to the parameters 
$ k_{21}, k_{32}, k_{13}, k_{01}, k_{03}$, 
respectively
-- 
is the following $5 \times 5$ matrix:
    \begin{align*}        
    \begin{pmatrix}
        1 & 1 & 1 & 1 & 1 \\
        k_{03}+k_{13}+k_{32}
            & k_{01}+k_{03}+k_{13}+k_{21}
            & k_{01}+k_{21}+k_{32}
            &
          k_{03}+k_{13}+k_{32} 
        & k_{01}+k_{21}+k_{32} 
            \\
        0 & 0 & 1 & 0 & 1  \\
        1 & 0 & 0 & 0 & 0 \\
            % Formerly row 3
        k_{03} k_{32}
        &
        k_{01}(k_{03}+k_{13}) + k_{03}k_{21}
        &
        k_{01}k_{32}
        &
        k_{32}(k_{03}+k_{13}) 
        &
        (k_{01}+k_{21})k_{32}
\\
    \end{pmatrix}~.
    \end{align*}  
In the above matrix, we subtract column-$1$ from column-$4$, and column-$3$ from column-$5$.  This yields a matrix with the following lower-triangular block structure:
    \begin{align} \label{eq:matrix-jacobian-example}
    \left( \begin{array}{ccc|cc}
        1 & 1 & 1 & 0 & 0 \\
        k_{03}+k_{13}+k_{32}
            & k_{01}+k_{03}+k_{13}+k_{21}
            & k_{01}+k_{21}+k_{32}
            &
            0
            &0
                \\
        0 & 0 & 1 & 0 & 0  \\
        \hline 
        1 & 0 & 0 & -1 & 0 \\
            % Formerly row 3
        k_{03} k_{32}
        &
        k_{01}(k_{03}+k_{13}) + k_{03}k_{21}
        &
        k_{01}k_{32}
        &
        k_{32}k_{13} 
        &
        k_{21}k_{32}
\\
    \end{array} \right)~.
    \end{align}  
This block matrix makes it straightforward to confirm that the Jacobian matrix of $\mathsf{d}$ is full rank.  This confirms, via Proposition~\ref{prop:rank-matrix-for-identifiability} and Lemma~\ref{lem:modify-coeffs-by-division}, what we saw earlier: this model is generically locally
identifiable.
\end{example}

In Example~\ref{ex:ideas-in-big-proof}, we performed several steps.  First, we modified the coefficient map (by reordering the coefficients and replacing one coefficient by a quotient of two coefficents).  Next, 
we performed column operations on the  Jacobian matrix of the 
modified coefficient map.  Finally, we observed that the resulting matrix is lower triangular.  Versions of these steps are the key ideas in the proof of the following result; for instance,  the general block matrix appears below in~\eqref{eq:jac-after-col-opers}.

\begin{proposition} \label{prop:leak-interlacing}
If $\mathcal{M}$ is a directed-cycle model 
that is leak-interlacing, then $\mathcal{M}$ is generically locally identifiable.
\end{proposition}
\begin{proof}
Assume that $\mathcal{M}=
(G, \In, \Out, \Leak)$ is a leak-interlacing, directed-cycle model with $n$ compartments.  If $\lvert \Leak \rvert \leq 1$, then Proposition~\ref{prop:cycle-1-in-1-out}~(1) implies that $\mathcal{M}$ is generically locally identifiable.  So, for the remainder of the proof, assume that $\lvert \Leak \rvert \geq 2$.
We consider two cases, based on whether or not $\mathcal{M}$ is in the exceptional family (Definition~\ref{def:family}).

{\bf Case 1:} $\mathcal{M}$ is \uline{not} in the exceptional family.
In this case, 
Lemma~\ref{lem:minimal} implies that there exists a minimally leak-interlacing model  $\widetilde{\mathcal{M}}$ (with at least one input and at least one output) that is obtained from $\mathcal{M}$ by removing input(s) and/or output(s) (or removing none if $\mathcal{M}$ is already minimally leak-interlacing). 
Additionally, no compartment of $\widetilde{\mathcal{M}}$ is both an input and an output.

Some facts noted in the prior paragraph -- namely, $\widetilde{\mathcal{M}}$ has at least one input and at least one output, and has no compartment that has both an input and an output -- allow us to relabel the compartments, if needed, so that 
$1$ is an input compartment
and also that there are ``no inputs after the final output''.
More precisely, this relabeled model, which we denote by 
$\widetilde{\mathcal{M}}'
= (G, \widetilde{\In}', \widetilde{\Out}', \widetilde{\Leak}')
$,
satisfies the following:
    \begin{enumerate}
        \item
        $\widetilde{\In}' = \{i_1=1, i_2, ..., i_x\}$, where $2 \leq i_2< \dots < i_x$ if $x \geq 2$ (if $x \geq 1$, then $\widetilde{\In}' = \{i_1=1\}$),   
    \item
    $\widetilde{\Out}' = \{p_1, p_2, ..., p_y\}$, for some $2 \leq p_1 < \dots < p_y \leq n$, with $y \geq 1$, and
    \item
    $i_x \leq p_y - 1$.
    \end{enumerate}
Write the leak set as $\widetilde{\Leak}' = \{\ell_1, \ell_2, \dots , \ell_z\}$, where 
    $1 \leq \ell_1 < \ell_2 < \dots < \ell_z \leq n$.  
We know that 
$z = x+y \geq 2$,
because $\widetilde{\mathcal{M}}'$ is minimally leak-interlacing and has at least two leaks.

Let $J$ denote the Jacobian matrix of the coefficient map of $\widetilde{\mathcal{M}}'$. 
Let $\widetilde J$ denote the matrix obtained from $J$ by the following column operations: for all $\ell \in \widetilde{\Leak}'$, we replace 
the column 
$\vec{J}_{0\ell}$ 
by 
$ \vec{J}_{k_{0\ell}} - \vec{J}_{k_{\ell + 1, \ell}} $, 
where (like in earlier proofs) 
$\vec{J}_{0\ell}$ and
$\vec{J}_{k_\ell + 1, \ell}$, 
denote the columns of $J$ that correspond to 
the parameters
 $k_{0\ell}$ and 
$ k_{\ell + 1, \ell}$ 
 (respectively).
    To show that $\widetilde{\mathcal{M}}'$ is generically locally
identifiable, it suffices (by Proposition~\ref{prop:rank-matrix-for-identifiability} and elementary linear algebra) to show that an $(n+z)\times (n+z)$ submatrix of $\widetilde J$ has nonzero determinant. 

    We begin constructing our $(n+z)\times (n+z)$ submatrix of $\widetilde J$, which we call $M$, as follows.  
    % ---------------
    % COLUMNS OF MATRIX
    % ---------------
    The first $n$ columns correspond to the edge-parameters $k_{21}, k_{32}, \dots , k_{1 n}$, and the remaining $z$ columns  
    are the aforementioned
    ``column-operated'' ones (that is, the columns $ \vec{J}_{k_{0\ell}} - \vec{J}_{k_{\ell + 1, \ell}} $, for $\ell \in \widetilde{\Leak}'$), in some order to be specified later in the proof.
    
    % ---------------
    % ROWS OF MATRIX
    % ---------------
    As for the rows of $M$, 
    these essentially correspond to coefficients of $\widetilde{\mathcal{M}}'$, as given in Proposition~\ref{prop:specific-coefficients}.  
    The first $n$ rows correspond to the coefficients $e_1,e_2, ..., e_{n-1},$
    and $ e_1^{*}(i, j)$, for some $i \in \widetilde{\In}'$ and $j \in \widetilde{\Out}'$ which will be specified later in the proof.
    The next $z-1$ rows correspond to a subset of the coefficients of the form $\kappa(i,j)$, with $i \in \widetilde{\In}'$ and $j \in \widetilde{\Out}'$ (this choice of coefficients will be specified later). 
    Finally, the last row of $M$ corresponds to the coefficient $e_n - \prod^n_{i=1}k_{i+1, i}$.

    The plan for the remainder of the proof (for Case~1) is to show that 
    the matrix $M$ has the following block-diagonal form\footnote{For the model in Figure~\ref{fig:cycle-running-example}, the corresponding block-diagonal matrix was shown earlier in~\eqref{eq:matrix-jacobian-example}.}:
    \begin{align} 
    \label{eq:jac-after-col-opers}
    M \quad = \quad 
    \left(
    	\begin{array}{ccc|cccc}
    	&&&    0 &  0 &\hdots&0\\	
    	& A & &  \vdots & \vdots& \ddots & \vdots   \\
    	&&&    0  & 0 &\hdots&0 \\	
        \hline
        &&&& &&\\
        &C&&&& D &\\
        &&&& &&\\
    	\end{array}
    \right)~,
    \end{align}
where $A$ is an $n \times n$ matrix, $D$ is a $z \times z$ matrix, and both $A$ and $D$ have nonzero determinant.
%, $C$ is $z \times n$, 
Specifically, we proceed as follows.  We first define the rows and columns of $D$, and then prove that $D$ has nonzero determinant.  Next, we specify the last row of $A$, and then prove that $A$ also has nonzero determinant.  Finally, we show that the upper-right %$n \times z$ 
submatrix of $M$ consists of zeroes.

We begin with the matrix $D$.  
Recall that 
$\widetilde{\mathcal{M}}' $
is minimally leak-interlacing, and that there are no inputs after the final output.  Therefore, the inputs, outputs, and leaks ``interlace'' as follows\footnote{For the model in Figure~\ref{fig:cycle-running-example}, the corresponding inequalities are $i_1=1 \leq \ell_{11}=1 < p_{11}=2 \leq \ell_{12}=3$.}: 
\begin{align} \label{eq:interlace-proof}
    \notag
    i_1=1 & \leq \ell_{11} < p_{11} \leq \ell_{12} < p_{12} \leq \dots \leq \ell_{1,j_1} < p_{1,j_1} \leq \ell_{1,j_1+1}< \\
    \notag
    i_2 & \leq \ell_{21} < p_{21} \leq \ell_{22} < p_{22} \leq \dots \leq \ell_{2,j_2} < p_{2,j_2} \leq \ell_{2,j_2+1}< \\
    & \dots < 
    \\
    \notag
    i_{x-1} 
    & \leq \ell_{x-1,1} < p_{x-1,1} \leq \ell_{x-1,2} < p_{x-1,2} \leq \dots \leq \ell_{x-1,j_{x-1}} < p_{x-1,j_{x-1}} \leq \ell_{x-1,j_{x-1}+1}< 
    \\
    \notag
    i_{x} 
    & \leq \ell_{x1} < p_{x1} \leq \ell_{x2} < p_{x2} \leq \dots \leq \ell_{x,j_{x}} < p_{x,j_{x}} \leq \ell_{x,j_{x}+1} \leq n
\end{align}
The last output and leak are, respectively, $p_{x,j_{x}} = p_y$ and $\ell_{x,j_{x}+1}=\ell_z$.  Additionally, 
the first leak after each input 
necessarily exists (i.e., the leaks $\ell_{11},\ell_{21},\dots, \ell_{x1}$), because the model is minimally leak-interlacing.  However, the outputs and leaks ``after'' those first leak in a given row might not exist;
for instance, if $j_1=0$, then the outputs and leaks labeled by 
$p_{11} ,~ \ell_{12} ,~ p_{12} ,~ \dots ,~ \ell_{1,j_1} ,~ p_{1,j_1} ,~ \ell_{1,j_1+1}$ in the first row of 
\eqref{eq:interlace-proof} do not exist.

Before specifying the rows of $D$, we need some notation.  For any output $p \in \widetilde{\Out}'$ that is not the last output (i.e., $p \neq p_y$), let $\nu(p)$ denote the ``next'' output, that is, $\nu(p)=p_{i+1}$ if $p=p_i$.

We choose the $z$ rows of $D$ to correspond to the following (order of) coefficients, which can be viewed as starting from the bottom line 
of~\eqref{eq:interlace-proof}
and then moving up one line at a time:
    \begin{enumerate}
        \item $
        \kappa(i_x,  p_{x1} ) ,~
        \kappa(i_x, p_{x2} ) ,~
        \dots
        ,~
        \kappa(i_x, p_{x,j_{x}}) 
            $
        \item 
        $
        \kappa(i_{x-1},  p_{{x-1},1} ) ,~
        \kappa(i_{x-1}, p_{{x-1},2} ) ,~
        \dots
        ,~
        \kappa(i_{x-1}, p_{{x-1}, j_{{x-1}}}),~
        \kappa(i_{x-1}, p_{x1})
        $
        \\
        $\vdots$ 
        \item [$(x{-}1)$] 
                $
        \kappa(i_2,  p_{21} ) ,~
        \kappa(i_2, p_{22} ) ,~
        \dots
        ,~
        \kappa(i_2, p_{2, j_{2}}),~
        \kappa(i_2, \nu(p_{2, j_{2}}))
        $, 
        \item [$(x)$] 
                $
        \kappa(i_1,  p_{11} ) ,~
        \kappa(i_1, p_{12} ) ,~
        \dots
        ,~
        \kappa(i_1, p_{1, j_{1}}),~
        \kappa(i_1, \nu(p_{1, j_{1}}))
        $,
        \item[$(x{+}1)$] $e_n - \prod^n_{i=1}k_{i+1, i}$ (this coefficient was specified earlier in the proof).
    \end{enumerate}
At least one coefficient exists in each of the above lists, namely, the coefficients 
$\kappa(i_x,  p_{x1} )$, 
$\kappa(i_{x-1}, p_{x1})$, 
\dots,
$\kappa(i_2, \nu(p_{2, j_{2}}))$, 
$\kappa(i_1, \nu(p_{1, j_{1}}))$, 
$e_n - \prod^n_{i=1}k_{i+1, i}$.

We choose the $z$ ``column-operated'' columns of $D$ to correspond to the following order of the $z$ leaks, which (like above) can be viewed as 
proceeding from the bottom line to the top of~\eqref{eq:interlace-proof}:
    \begin{enumerate}
        \item 
          $\ell_{x1} ,~  \ell_{x2}  ,~ \dots  ,~ \ell_{x,j_{x}} $
        \item $\ell_{x-1,1} ,~ \ell_{x-1,2} ,~ \dots ,~ \ell_{x-1,j_{x-1}} ,~ \ell_{x-1,j_{x-1}+1}$
        \\
        $\vdots$ 
        \item [$(x{-}1)$] 
            $\ell_{21} ,~ \ell_{22}  ,~ \dots  ,~ \ell_{2,j_2}  ,~ \ell_{2,j_2+1}$ 
        \item [$(x)$] 
            $\ell_{11} ,~ \ell_{12} ,~ \dots ,~ \ell_{1,j_1} ,~ \ell_{1,j_1+1}$
        \item[$(x{+}1)$] 
            $\ell_{x,j_{x}+1}= \ell_z$.
    \end{enumerate}
The first parameter above, $\ell_{x1}$, labels the unique leak along the path 
(in the sense of Remark~\ref{rem:path}) 
from input $i_x$ to output $p_{x1}$, and this path  corresponds to the coefficient $\kappa(i_x,  p_{x1} )$ chosen for the first row of $D$.  Similarly, the next leak, $\ell_{x2}$, is the right-most leak along the path from input $i_x$ to output $p_{x2}$, which corresponds to the coefficient $\kappa(i_x,  p_{x2})$ defining the second row of $D$. This pattern continues (by construction).

In particular, the rows and columns of $D$ are ordered so that each path (arising from some row of $D$) contains the leak of the corresponding column of $D$ and contains leaks of some columns to the left, but none to the right.  Therefore, by applying 
Lemma~\ref{lem:partial-deriv}, 
we conclude that the matrix $D$ is lower-triangular, and the first $z-1$ main-diagonal entries are nonzero.  As for the last diagonal entry (the lower-right entry of $D$), it is straightforward to check that this entry is $-\frac{ k_{21} k_{32} \dots k_{n,n-1} k_{1n} }{k_{\ell_z+1,\ell_z}}$ 
(and so is nonzero).
We conclude that $D$ has nonzero determinant.  

Before moving on to the matrix $A$, we state a claim that asserts that $p_{x1}$ (the output immediately after the last input, $i_x$) can not also be only one compartment ``behind'' $i_x$ (in Case~1).

{\bf Claim:} In the inequalities~\eqref{eq:interlace-proof}, we have $p_{x1} \neq i_x - 1$ (mod $n$).

We prove this claim as follows.  
Assume for contradiction the equality $p_{x1} = i_x - 1$ (mod $n$).  
It is straightforward to check from the inequalities~\eqref{eq:interlace-proof} that we must have $x=1$ ($\widetilde{\mathcal{M}}'$ has only one input) and $p_{x1}=n$ ($\widetilde{\mathcal{M}}'$ has only one output), and the inequalities~\eqref{eq:interlace-proof} reduce as follows:
\begin{align*}
    1 = i_x \leq \ell_{x1} < p_{x1}=n = \ell_{z}~.
\end{align*}
It follows (by Definition~\ref{def:exceptional-model}) that $\widetilde{\mathcal{M}}'$ is an exceptional model.  Thus, as 
$\widetilde{\mathcal{M}}$
is obtained from $\widetilde{\mathcal{M}}'$ by relabeling and/or adding input(s) and/or output(s), we conclude that $\mathcal{M}$ is in the exceptional family.  This contradicts our assumption for Case~1, and concludes our proof of the claim.

   Our next aim is to finish constructing the $n \times n$ matrix $A$, and then to show that $A$ has nonzero determinant. 
    Only the last row of $A$ remains to be specified; we pick it to correspond to the coefficient
    $e_1^*(i_x, p_{x1})$, as in Proposition~\ref{prop:specific-coefficients}. 
    The fact that this coefficient exists follows from our claim above and Proposition~\ref{prop:specific-coefficients}(3).
    (Technically, 
     $e_1^*(i_x, p_{x1}) \kappa(i_x, p_{x1})$ is a coefficient of the coefficient map, and $e_1^*(i_x, p_{x1})$ is not, 
     but Lemma~\ref{lem:modify-coeffs-by-division} implies that it is permissible to replace the former term by the latter, as $\kappa(i_x, p_{x1})$ 
     is a coefficient.) 
   
   Let $\widetilde A$ denote the matrix obtained from $A$, after setting all leak-parameters to zero: $k_{0\ell_1}=k_{0\ell_2} =\jordy{\cdots}=k_{0 \ell_z}=0$.
    By construction, $\widetilde A$
     is the Jacobian matrix (with respect to $k_{21}, k_{32}, \dots, k_{n,n-1}, k_{1n}$) of the following polynomials (which are obtained from $e_1,e_2,\dots, e_{n-1}$, and $ e_1^{*}(i_x, p_{x1})$ by setting $k_{0\ell_1}=k_{0\ell_2} =\jordy{\cdots}=k_{0 \ell_z}=0$): 
     
    \begin{align} \label{eq:restricted-coeffs}
        & e_1|_{k_{0\ell_1} =k_{0\ell_2} =\cdots=k_{0 \ell_z}=0}~,~
        e_2|_{k_{0\ell_1}=k_{0\ell_2} =\cdots=k_{0 \ell_z}=0}~,~
        \dots 
        e_{n-1}|_{k_{0\ell_1}=k_{0\ell_2} =\cdots=k_{0 \ell_z}=0}~,~
                \\        \notag
        & \quad  e_1^{*}(i_x, p_{x1})|_{k_{0\ell_1}=k_{0\ell_2} =\cdots=k_{0 \ell_z}=0} ~.
         \end{align}
    In turn, the polynomials in~\eqref{eq:restricted-coeffs} can be rewritten as follows:         
    \begin{align} \label{eq:restricted-coeffs-rewritten}
         \widetilde{e_1},~
        \widetilde{e_2},~
        \dots,~
        \widetilde{e_{n-1}},~
        k_{p_{x_1}+2,p_{x_1}+1} + 
        k_{p_{x_1}+3,p_{x_1}+2} +
         \dots +  
         k_{i_x, i_x-1} ~,
    \end{align}
    where $\widetilde{e_i}$ is the $i^{th}$ elementary symmetric polynomial on the  edge-parameter  set $Edge$ $:=$ $\{ k_{21}, k_{32},$ $\dots, k_{n,n-1}, k_{1n} \}$. 
The last expression in the list~\eqref{eq:restricted-coeffs-rewritten} is a sum 
of parameters corresponding to
a proper, nonempty subset of $Edge$ (for instance, $k_{i_x, i_x-1}$ is a summand of that term, but not $k_{p_{x1}+1,p_{x1}}$).  
It follows that Lemma~\ref{lem:alg-indep} applies to the $n$ polynomials in~\eqref{eq:restricted-coeffs-rewritten}, and so we conclude that these polynomials are algebraically independent over $\mathbb C$. Hence, Lemma~\ref{lem:alg-indep-general} implies that
    $\det \widetilde{A} \neq 0$.  We conclude that $\det A \neq 0$, as desired.

Finally, we explain how the zeroes in the upper-right $n \times z$ submatrix of $M$ arise.  
 The zeroes in the first $n-1$ rows are due to Lemma~\ref{lem:partial-deriv-1-input}(I).
 
 As for the last row, this corresponds to the coefficient $e_1^*(i_x, p_{\beta})$.  For all $\ell \in \widetilde{\Leak}'$, either (i) $k_{\ell+1,\ell} + k_{0\ell}$ is a summand of $e_1^*(i_x, p_{\beta})$ (and no other summand involves $k_{\ell+1,\ell}$ or $ k_{0\ell}$) 
 or (ii) no summand of $e_1^*(i_x, p_{\beta})$ involves $k_{\ell+1,\ell}$ or $ k_{0\ell}$.  We conclude that the last row is the zero row. 

Thus, $M$ has nonzero determinant, and so 
(by Proposition~\ref{prop:rank-matrix-for-identifiability})
the model $\widetilde{\mathcal{M}}'$ (and also $\widetilde{\mathcal{M}}$) is generically locally
identifiable.  Hence, by Proposition~\ref{prop:add-in-out}, $\mathcal{M}$ also is generically locally
identifiable.

% REMAINING CASE
{\bf Case 2:} $\mathcal{M}$ is in the exceptional family.
 In this case, by definition, $\mathcal{M}$ is obtained by adding at least one input and/or at least one output to an exceptional model.
 Recall that exceptional models have exactly one input and one output (Definition~\ref{def:exceptional-model}). 
 Thus, $\mathcal{M}$
has at least two inputs and/or at least two outputs.

Relabel the compartments of ${\mathcal{M}}$ so that there is an input in compartment-$1$, an output in compartment-$n$, and a leak in compartment-$n$ (this relabeling can be easily seen from the exceptional model, before adding inputs and/or outputs to obtain $\mathcal{M}$).  
Letting $\widetilde {\mathcal{M}} = (\In, \Out, \Leak)$ denote the resulting model, we have: 
\begin{enumerate}
    \item $\In = \{i_1=1, i_2, ..., i_x\}$, where $2 \leq i_2 < \dots < i_x \leq n$ if $x \geq 2$ (if $x \geq 1$, then $\In = \{i_1=1\}$), 
    \item $\Out = \{p_1, p_2, ..., , p_y= n\}$, 
where 
$1 \leq p_1 < p_2 <\dots  < p_y=n$, with $y \geq 1$, and
    \item $\Leak = \{\ell, n\}$, 
where $1 \leq \ell \leq n-1$.
\end{enumerate}

Let $J$ denote the Jacobian matrix of the coefficient map of $\widetilde{\mathcal{M}}$. 
Like in Case~1, we let $\widetilde J$ denote the matrix obtained from $J$ by replacing
the columns 
$\vec{J}_{0\ell}$ and
$\vec{J}_{0n}$ by (respectively) 
$ \vec{J}_{k_{0\ell}} - \vec{J}_{k_{\ell + 1, \ell}} $ and
$ \vec{J}_{k_{0n}} - \vec{J}_{k_{1n}} $, 
where 
$\vec{J}_{0\ell}$, 
$\vec{J}_{k_\ell + 1, \ell}$, 
$\vec{J}_{0n}$, and
$\vec{J}_{k_{1 n}}$ denote the columns of $J$ that correspond to 
the parameters
 $k_{0\ell}$, 
$ k_{\ell + 1, \ell}$, 
$k_{0n}$,
and $k_{1n}$ (respectively).
We reorder the columns, if needed, so that the two modified columns ($ \vec{J}_{k_{0\ell}} - \vec{J}_{k_{\ell + 1, \ell}} $ and
$ \vec{J}_{k_{0n}} - \vec{J}_{k_{1n}} $) are the right-most columns of $\widetilde J$.

Our next aim is to choose $n+2$ rows of $\widetilde J$, with the plan of showing that the resulting (square) submatrix has nonzero determinant.  
The rows that we will choose essentially correspond to coefficients of $\widetilde{\mathcal{M}}$ listed in Proposition~\ref{prop:specific-coefficients}. 
We take the first $n-1$ rows to correspond to the coefficients 
$e_1, e_2, \dots, e_{n-1}$.  

For the moment, we skip the $n$-th row, and instead we choose the remaining two rows to correspond to the coefficients 
$\kappa(1,n)=k_{21}k_{32}\dots k_{n,n-1}$
and $e_n - \prod^n_{i=1}k_{i+1, i}= e_n - k_{21} k_{32} \dots k_{n,n-1} k_{1n}$. 

Returning now to the $n$-th 
row, we consider two subcases:

{\bf Subcase A:} $\widetilde {\mathcal{M}}$ has at least two outputs.  
In this case, the output $p_1$ exists.  We choose the $n$-th row to correspond to $e_1^{*}(1, p_1)$.

{\bf Subcase B:} $\widetilde {\mathcal{M}}$ has at least two inputs.  In this case, 
the input $i_2$ exists.  
We choose the $n$-th row to correspond to 
$e_1^{*}(i_2, n)$. 

Much like in Case~1, the term 
$e_1^{*}(1, p_1)$ 
(respectively, $e_1^{*}(i_2, n)$ in Subcase~B) 
is not a coefficient of the coefficient map.  However,   
$e_1^{*}(1, p_1) \kappa(1, p_1)$ 
and $\kappa(1, p_1)$ 
(respectively, 
$e_1^{*}(i_2, n) \kappa(i_2, n)$)
and $\kappa(i_2, n)$)
are coefficients
(by Proposition~\ref{prop:specific-coefficients}).  Hence, by Lemma~\ref{lem:modify-coeffs-by-division}, the modification we performed on the coefficients is permitted.

We have defined a square submatrix of $\widetilde J$; denote this submatrix by $M$.
Our next aim is to show that $M$ has the same block-form as the matrix (also called $M$) in~\eqref{eq:jac-after-col-opers}, from Case~1. 
In particular, we show that $M$ consists of blocks $A$, $C$, and $D$
of size $n \times n$, $2 \times n$, and $2 \times 2$ (respectively), as in~\eqref{eq:jac-after-col-opers}. 
(The matrix $D$ is $2 \times 2$, because $\widetilde {\mathcal{M}}$ has exactly two leaks.)

We first show that the $n \times n$ upper-left submatrix of $M$, which we call $A$, has nonzero determinant.  The proof closely mirrors that of Case~1 (for the matrix called $A$ in that case).  As before, we let $\widetilde A$ denote the matrix obtained from $A$ after setting both leak parameters to zero ($k_{0\ell}=k_{0n}=0$).  
As in~\eqref{eq:restricted-coeffs-rewritten}, $\widetilde A$ is the Jacobian matrix  of the following:
    \begin{align} \label{eq:restricted-coeffs-rewritten-case-2}
        \widetilde{e_1},~
        \widetilde{e_2},~
        \dots,~
        \widetilde{e_{n-1}},~
         \alpha ~,
    \end{align}
    where $\widetilde{e_i}$ is the $i$-th elementary symmetric polynomial on $Edge=\{ k_{21}, k_{32}, \dots, k_{n,n-1}, k_{1n} \}$ and 
    \begin{align*}
        \alpha ~:= ~
        \begin{cases}
            k_{p_{1}+2,p_{1}+ 1} + k_{p_{1}+3,p_{1}+2} + \dots +  k_{n,n-1} + k_{1n}
                    & \mathrm{if~Subcase~A~(so,~}1 \leq p_1 \leq n-1\mathrm{)} \\
            k_{21} + k_{32}+ \dots + k_{i_2,i_2-1}
                    & \mathrm{if~Subcase~B~(so,~}2 \leq i_2 \leq n\mathrm{)} \\
        \end{cases}
    \end{align*}
In either subcase, $\alpha$ is a sum of a proper, nonempty subset of $Edge$, and so (like in Case~1) Lemmas~\ref{lem:alg-indep-general} and~\ref{lem:alg-indep} apply to the $n$ polynomials in~\eqref{eq:restricted-coeffs-rewritten-case-2}.
We conclude that
    $\det \widetilde{A} \neq 0$, and thus, $\det A \neq 0$.

Next, we show that the upper-right $n \times 2$ block of $M$, which we call $B$, is the zero matrix.  The fact that the first $n-1$ rows are $0$ 
 is due to Lemma~\ref{lem:partial-deriv-1-input}(I).
  
As for the last row of $B$, this corresponds to the coefficient $e_1^{*}(1, p_1)$ (in Subcase~A) or $e_1^{*}(i_2, n)$ (in Subcase~B).  The coefficient $e_1^{*}(1, p_1)$ is an elementary symmetric polynomial on a set $\Sigma$ that contains $k_{1n}+ k_{0n}$ (and no other element of the set involves $k_{1n}$ or $k_{0n}$); while the coefficient $e_1^{*}(i_2, n)$ 
 is an elementary symmetric polynomial on a set $\Sigma'$ in which no element involves $k_{1n}$ or $k_{0n}$.  Also, for each of the sets $\Sigma$ and $\Sigma'$, the set either 

contains $k_{\ell+1,\ell}+ k_{0 \ell}$ (and no other element of the set involves $k_{\ell+1,\ell}$ or $k_{0 \ell}$) 
or 
% else
no element of the set involves $k_{\ell+1,\ell}$ or $k_{0 \ell}$.  
We conclude that the last row of $B$, and thus all of $B$, consists of zeroes.

Finally, it is straightforward to check that the $2 \times 2$ submatrix $D$ is lower triangular, as follows:
\begin{align*} 
D~=~
\left(
	\begin{array}{cc}
        k_{21} k_{32} \dots k_{\ell,\ell-1} k_{\ell+2,\ell} \dots k_{n,n-1} & 0 \\
        * & -k_{21} k_{32} \dots k_{n,n-1}\\
	\end{array}
\right)~.
\end{align*}

We conclude that $M$ has nonzero determinant, as desired.  Thus,  $\widetilde J$ has rank $n+2$, and so $J$ (the Jacobian matrix of the coefficient map of $\widetilde{\mathcal{M}}$) does too.
    It now follows (from Proposition~\ref{prop:rank-matrix-for-identifiability}) that the model $\widetilde{\mathcal{M}}$ -- and therefore $\mathcal{M}$ as well -- is generically locally
identifiable.
\end{proof}

\section{The singular locus of directed-cycle models} \label{sec:SL}
In the prior section, we 
proved that a directed-cycle model is generically locally identifiable if and only if it is leak-interlacing.  Recall that ``generically'' here means that identifiable models have a measure-zero set of parameter vectors for which the 
Jacobian matrix of the coefficient map is singular.  
This set was called the {\em singular locus} by Gross {\em et al.}, who analyzed the singular locus of several classes of directed-cycle models and bidirected-tree models~\cite{singularlocus}. 

\begin{definition} \label{def:singular-locus}
    Let $\mathcal{M} = (G, In, Out, Leak)$, where $G=(V,E)$, be a strongly connected linear compartmental model that is generically locally identifiable. 
    Let $\msc$ denote the coefficient map of~$\mathcal{M}$. 
    The \textbf{singular locus} of $\mathcal{M}$ is the subset of $\mathbb{R}^{|E| + |Leak|}$ at which the Jacobian matrix of $\msc$ is \uline{not} full rank.
\end{definition}

\begin{example}[Example~\ref{ex:ideas-in-big-proof}, continued] \label{ex:singular-locus}
We revisit the directed-cycle model in Figure~\ref{fig:cycle-running-example}, which has three compartments, with $In=\{1\}$, $Out=\{2\}$, and $Leak = \{1, 3\}$.  For this model, we saw in Example~\ref{ex:jacobian-matrix} that the $5 \times 5$ Jacobian matrix of the coefficient map
has the following determinant: $
k_{21}^2  k_{32} (k_{01} + k_{21} - k_{32})$.  It follows that the vanishing of this determinant defines the singular locus of this model; hence, the singular locus is the union of the following three hyperplanes:
\begin{enumerate}[label=(\alph*)]%[(a)]
    \item $\{k_{21}=0\}$,
    \item $\{k_{32}=0\}$, and
    \item $\{ k_{01} + k_{21} = k_{32} \}$.
\end{enumerate}
The existence of such hyperplanes within the singular locus is a more general phenomenon, as we will see in the next result.  In particular, 
the hyperplane in part~(1) (respectively, the hyperplanes in part~(2)) of Theorem~\ref{thm:singular-locus-cycle} below
corresponds to the hyperplane (a) (respectively, the hyperplanes~(a) and (b)).
Additionally, hyperplane (c) is one of the hyperplanes appearing later in Conjecture~\ref{conj:s-locus}. 
\end{example}
The remainder of this section is organized as follows.  Our results on the singular locus are stated in Section~\ref{sec:s-loc-results}, and the proofs are given in Section~\ref{sec:s-loc-proofs}.

\subsection{Results} \label{sec:s-loc-results}

The first main result of this section, as follows, concerns the singular locus of directed-cycle models. 

\begin{theorem}[Hyperplanes in the singular locus] \label{thm:singular-locus-cycle}
    Let $n \geq 3$, and let $\mathcal{M} = (G, \In, \Out, \Leak)$ be an $n$-compartment directed-cycle model with $\In=\{1\}$.  Assume that $\mathcal{M}$ is generically locally 
    identifiable (or, equivalently, is leak-interlacing). 
    Let $\mathcal{S}_{\mathcal{M}}$ denote the singular locus of $\mathcal M$. 
    \begin{enumerate}
        % part (1)
        \item If there exists a leak $\ell \in \Leak \smallsetminus \{1\}$, then 
        $\mathcal{S}_{\mathcal{M}}$ contains 
        the hyperplane $\{k_{21}=0\}$.
        % part (2)
        \item If 
        $\Leak = \{\ell_1,\ell_2,\dots,\ell_z \}$
        (with $1 \leq \ell_1 < \ell_2 < \dots < \ell_z \leq n$ and $z \geq 1$)
        and 
        $\ell_z \geq p$ for all $p \in \Out$, 
        then  
        $\mathcal{S}_{\mathcal{M}}$ contains 
        the hyperplanes $\{ k_{a+1,a}=0\}$ for all $a \in [n] \smallsetminus \{\ell_z \}$.
    
        % part (3)
        \item  If 
        $\Out=\{p\}$ and $\Leak=\{q,p\}$ with $1 \leq q < p \leq n-1$, 
        then 
        $\mathcal{S}_{\mathcal{M}}$ contains 
         the hyperplane $\{k_{0q} + k_{q+1,q} = k_{0p} + k_{p+1,p}\}$.
        
    \end{enumerate}
\end{theorem}

\noindent We prove Theorem~\ref{thm:singular-locus-cycle} in the next subsection.  The next example illustrates part~(3) of the theorem.

\begin{figure}[ht]
\begin{center}
\begin{tikzpicture}[
roundnode/.style={circle, draw=black, very thick, minimum size=10mm},
arrowbasic/.style={very thick, ->},
]

%Nodes
\node[roundnode](comp1){1};
\node[roundnode](comp2) [right=of comp1] {2};
\node[roundnode](comp3) [below=of comp2] {3};
\node[roundnode](comp4) [left=of comp3] {4};

%Lines

% edges
\draw[arrowbasic] (comp1) -- (comp2) node[pos=.5, above] {\(k_{21}\)};
\draw[arrowbasic] (comp2) -- (comp3) node[pos=.5, right] {\(k_{32}\)};
\draw[arrowbasic] (comp3) -- (comp4) node[pos=.5, above] {\(k_{43}\)};
\draw[arrowbasic] (comp4) -- (comp1) node[pos=.5, left] {\(k_{14}\)};

% input
%\draw[arrowbasic] (0,1.35) -- node[right] {In} (comp1);
\draw[arrowbasic] (-1,1) -- node[above right] {In} (comp1);

% output
\draw(comp3) -- (3.5,-2.03);
\draw(3.6,-2.03) circle (0.1) node[right] {};

% leak
\draw[arrowbasic](comp1) -- node[above] {\(k_{01}\)} (-1.35,0);
%\draw[arrowbasic](comp3) -- node[right] {\(k_{03}\)} (2.05,-3.5);
\draw[arrowbasic](comp3) -- node[below] {\(k_{03}\)} (3.5,-2.75);
\end{tikzpicture}
\end{center}
    \caption{A leak-interlacing directed-cycle model %illustrating part~(3) of Theorem~\ref{thm:singular-locus-cycle}.
    with $\In=\{1\}$, $\Out=\{3\}$, and $\Leak=\{1,{\color{blue} 3}\}$.}
    \label{fig:cycle-4-thm-singular-locus-3}
\end{figure}
\begin{example}\label{ex:singular-locus-thm-3}
Consider the directed-cycle model in Figure~\ref{fig:cycle-4-thm-singular-locus-3}. 
The Jacobian matrix $J$ of the coefficient map is 
a $6{\times}6$ matrix, so the singular locus is defined by the determinant of $J$:
\begin{equation*}
\det(J) ~=~ k_{14}k_{21}^{2}k_{32}^{3}(k_{32}-k_{43}-k_{03})(k_{21}-k_{43}+k_{01}-k_{03})(k_{21} - k_{32} + k_{01})~.
\end{equation*}
\noindent
In particular, the singular locus contains the hyperplane $\{k_{01} + k_{21} = k_{03} + k_{43}\}$, 
which is consistent with part~(3) of
Theorem~\ref{thm:singular-locus-cycle}.
\end{example}

For directed-cycle models with $\In = \Out =
\Leak = 
\{1\}$, 
part~(2) of 
Theorem~\ref{thm:singular-locus-cycle} implies that the singular locus contains the $n-1$ coordinate hyperplanes defined by, respectively, $k_{32}=0$, $k_{43}=0$, \dots, $k_{1n}=0$.  
The next result (Theorem~\ref{thm:1-1-p-singular-locus}) implies that 
an additional $\binom{n}{2}$ hyperplanes fill out the rest of the singular locus. Theorem~\ref{thm:1-1-p-singular-locus} generalizes a result of Gross {\em et al.}~\cite[Theorem~5.3]{singularlocus}; their result is the $\Leak = \{1\}$ case of ours, and our proof (which appears in the next subsection) generalizes theirs.

\begin{theorem}[Singular locus]
    \label{thm:1-1-p-singular-locus}
    Let $n \geq 3$. 
    For the $n$-compartment directed-cycle model $\mathcal{M} = (G, \In, \Out, \Leak)$ with $\In = \Out =\{1\}$ and  $\Leak = \{\ell\}$, where $1 \leq \ell \leq n$, 
    the singular locus is defined by the following equation:
    \begin{align}
    \label{eq:s-locus-in=out=1}        
    \prod_{i\in [n] \smallsetminus \{\ell \} } k_{i+1,i}
    \prod_{2 \leq i < j \leq n} (\widetilde{k}_{i} - \widetilde{k}_{j}) ~=~0~,
    \end{align}
where, for $i \in [n]$, 
$$
\widetilde{k}_{i} ~:=~
\begin{cases} 
    k_{i+1,i}  & \text { if } i \neq \ell
    \\
    k_{\ell+1,\ell}+k_{0 \ell}  & \text { if } i = \ell ~.   
    \end{cases}
$$
\end{theorem}

Before turning to the proofs of the above theorems, we state a conjecture concerning additional hyperplanes in the singular locus.
\begin{conjecture} \label{conj:s-locus}
    Let $n \geq 3$, and let $\mathcal{M} = (G, \In, \Out, \Leak)$ be an $n$-compartment directed-cycle model with $\In=\{1\}$ and $\Out=\{p\}$ (for some $1 \leq p \leq n$).  
Assume that $\mathcal{M}$ is generically locally
    identifiable (or, equivalently, is leak-interlacing). 
Then, for all $a \in ([p-1] \smallsetminus \Leak)$ and for all $\ell \in (\Leak \cap [p])$, the singular locus of $\mathcal M$ contains the hyperplane
$\{k_{\ell +1, \ell } + k_{0 \ell } = k_{a+1,a} \}$.
\end{conjecture}

\subsection{Proofs of Theorems~\ref{thm:singular-locus-cycle} and~\ref{thm:1-1-p-singular-locus}} \label{sec:s-loc-proofs}

The idea behind the proof of Theorem~\ref{thm:singular-locus-cycle} is that certain columns of the Jacobian matrix become equal, for parameter values ``inside'' the relevant hyperplanes (e.g., when $k_{21}=0$).  We illustrate this idea through the following example.

\begin{example}
 [Example~\ref{ex:singular-locus}, continued] \label{ex:columns-become-equal}
Recall that, for 
the directed-cycle model in Figure~\ref{fig:cycle-running-example}, 
the singular locus consists of three hyperplanes, including $\{k_{21}=0\}$ and $\{k_{32}=0\}$.  Earlier, we found these hyperplanes from the determinant of the Jacobian matrix.  Alternatively, we can examine the Jacobian matrix~\eqref{eq:matrix-for-example} and observe that the third and fifth columns of this matrix (which correspond to the parameters $k_{13}$ and $k_{03}$, respectively) become equal when $k_{21}=0$ or $k_{32}=0$.
\end{example}

\begin{proof}[Proof of Theorem~\ref{thm:singular-locus-cycle}]
Let $\mathcal{M} = (G, \In, \Out, \Leak)$ be a generically locally identifiable $n$-compartment directed-cycle model with $\In=\{1\}$.  Let $\msc$ denote the coefficient map of $\mathcal{M} $.

{\bf Part~(1)}.  Let $\ell \in \Leak \smallsetminus \{1\}$.  
Our aim is to show that 
the columns of the Jacobian matrix of $\msc$ corresponding to $k_{\ell+1,\ell}$ and $k_{0\ell}$ are the same when $k_{21} = 0$.  In other words, we must verify the following equality for all coefficients $\msc_i$: 
\begin{align} \label{eq:part-1-proof-equality}
    \left( \frac{\partial}{\partial k_{\ell+1,\ell}}[\msc_i] \right) |_{k_{21}=0} ~=~ 
    \left(
    \frac{\partial}{\partial k_{0 \ell}}[\msc_i] 
        \right)|_{k_{21}=0}
\end{align}

By Proposition~\ref{prop:coeff-many-in-or-out}, each coefficient $\msc_i$ arises from the coefficient map of a directed-cycle model $\mathcal{M}(p) := (G, \In, \{p\}, \Leak)$, where $p \in \Out$.  Therefore, it suffices to restrict our attention to the coefficients~\eqref{eq:coeff-map-cycle} 
of $\mathcal{M}(p)$ given in Lemma~\ref{lem:coefficient-map-cycle}, where $p\in \Out$ is fixed but arbitrary.  We consider each of the four types of coefficients (as in Notation~\ref{notation:types-of-coefficients}).

\uline{Type I.}
By Lemma~\ref{lem:partial-deriv-1-input}(I), $\frac{\partial}{\partial k_{\ell+1,\ell}}[e_i] = \frac{\partial}{\partial k_{0 \ell}}[e_i] $, for $1 \leq i \leq n-1$ (and also for $i=n$).

\uline{Type II.}
Lemma~\ref{lem:partial-deriv-1-input}(II) yields the following equality: 
        \begin{align} \label{eq:proof-part-1-type-II}
        \left(
        \frac{\partial }{\partial k_{\ell+1,+\ell}}
        -
        \frac{\partial }{\partial k_{0 \ell}} 
        \right)
        \left[e_n-\displaystyle{\prod_{i=1}^{n} k_{i+1, i}} \right]
        ~=~
        -
        \displaystyle{\prod_{1 \leq i \leq n,~ i\neq \ell} k_{i+1, i}}~, 
        \end{align}  
and the product on the right-hand side of~\eqref{eq:proof-part-1-type-II}
contains $k_{21}$ as a factor (because $\ell \neq 1$).  Hence, the desired equality~\eqref{eq:part-1-proof-equality} holds for the Type-II coefficient, $e_n-\displaystyle{\prod_{i=1}^{n} k_{i+1, i}} $.

\uline{Type III.} 
Lemma~\ref{lem:partial-deriv-1-input}(III) implies that
$\ds \frac{\partial}{\partial k_{0,\ell}}[\kappa] = 0$ and either 
$\ds \frac{\partial}{\partial k_{\ell+1,\ell}}[\kappa] = 0$ or 
$\ds \frac{\partial}{\partial k_{\ell+1,\ell}}[\kappa] = \frac{ k_{21} k_{32} \dots k_{p,p-1}}{ k_{\ell+1,\ell}}$, 
where $p \geq 2$ and $\ell \leq p-1$.  Hence,  as $\ell \geq 1$, we have $\left( \ds \frac{\partial}{\partial k_{\ell+1,\ell}}[\kappa] \right)|_{k_{21}=0}=0$.  We conclude that the desired equality~\eqref{eq:part-1-proof-equality} holds for the coefficient $\kappa$.

\uline{Type IV.} 
Fix $1 \leq i \leq n-p$.  
We consider three cases.  First, if $\ell \geq p+1$, then the desired equality~\eqref{eq:part-1-proof-equality} follows from Lemma~\ref{lem:partial-deriv-1-input}(IV)(a).  Next, if $\ell = p$, then the desired equality follows from Lemma~\ref{lem:partial-deriv-1-input}(IV)(b).  Finally, assume that $2 \leq \ell \leq p-1$.  In this case, parts (b) and (c) of Lemma~\ref{lem:partial-deriv-1-input}(IV) imply the following:
\begin{align} \label{eq:proof-part-1-type-IV-last-case}
        \left(
        \frac{\partial }{\partial k_{\ell+1,+\ell}}
        -
        \frac{\partial }{\partial k_{0 \ell}} 
        \right)
        \left[e_i^* \kappa \right]
        ~=~
        e_i^* \frac{\kappa}{ k_{\ell+1, \ell}}~,
        \end{align}  
and the right-hand side of~\eqref{eq:proof-part-1-type-IV-last-case}
contains $k_{21}$ as a factor (because $\ell \neq 1$).
Thus, the right-hand side of~\eqref{eq:proof-part-1-type-IV-last-case} becomes zero when $k_{21}=0$ and so the desired equality holds in this case as well.

{\bf Part~(2)}.  
Assume 
        $\Leak = \{\ell_1,\ell_2,\dots,\ell_z \}$
        and 
        $\ell_z \geq p$ for all $p \in \Out$.  Let $a \in [n] \smallsetminus \{\ell_z\}$.
Our aim is to show that 
the columns of the Jacobian matrix of $\msc$ corresponding to $k_{\ell_z+1,\ell_z}$ and $k_{0 \ell_z}$ are the same when $k_{a+1,a} = 0$.  That is, we will show that $\left( \frac{\partial}{\partial k_{\ell_z+1,\ell_z}}[\msc_i] \right)|_{k_{a+1,a} = 0} = \left(  \frac{\partial}{\partial k_{0 \ell_z}}[\msc_i] \right)|_{k_{a+1,a} = 0}$.
As in the proof of part~(1) above, Proposition~\ref{prop:coeff-many-in-or-out} allows us to restrict our attention to the coefficients~\eqref{eq:coeff-map-cycle} 
of $\mathcal{M}(p):= (G, \In, \{p\}, \Leak)$ given in Lemma~\ref{lem:coefficient-map-cycle}, where $p\in \Out$ is fixed but arbitrary. 
As in part~(1), we consider the four types of coefficients.

\uline{Type I.} By Lemma~\ref{lem:partial-deriv-1-input}(I), $\frac{\partial}{\partial k_{\ell_z+1,\ell_z}}[e_i] = \frac{\partial}{\partial k_{0 \ell_z}}[e_i] $, 
for $1 \leq i \leq n-1$ (and for $i=n$, too).

\uline{Type II.} 
Lemma~\ref{lem:partial-deriv-1-input}(II) yields the following equality: 
        \begin{align} \label{eq:proof-part-2-type-II}
        \left(
        \frac{\partial }{\partial k_{\ell_z+1,\ell_z}}
        -
        \frac{\partial }{\partial k_{0 \ell_z}} 
        \right)
        \left[e_n-\displaystyle{\prod_{i=1}^{n} k_{i+1, i}} \right]
        ~=~
        -
        \displaystyle{\prod_{1 \leq i \leq n,~ i\neq \ell_z } k_{i+1, i}}~, 
        \end{align}  
and the product on the right-hand side of~\eqref{eq:proof-part-2-type-II}
contains $k_{a+1,a}$ as a factor (because $a \neq \ell_z $).  Hence, the desired equality holds for this coefficient.

\uline{Type III.} Lemma~\ref{lem:partial-deriv-1-input}(III) implies that $\frac{\partial}{\partial k_{\ell_z+1,\ell_z}}[\kappa] = \frac{\partial}{\partial k_{0 \ell_z}}[\kappa] =0$ (here we use the fact $\ell_z \geq p$). 

\uline{Type IV.}
Lemma~\ref{lem:partial-deriv-1-input}(IV)(a--b) 
implies that $\frac{\partial}{\partial k_{0 \ell_z}}[e_i^*\kappa] =
\frac{\partial}{\partial k_{\ell_z+1,\ell_z}}[e_i^*\kappa]  = 0  $, for all $1 \leq i \leq n - p$.

{\bf Part~(3)}.  Assume $\Out=\{p\}$ and $\Leak=\{q,p\}$, where $1 \leq q < p \leq n-1$.  
Our aim is to show that 
the two ``leak columns'' of the Jacobian matrix of $\msc$ are equal when 
$k_{0q} + k_{q+1,q} = k_{0p} + k_{p+1,p}$.  
More precisely, we show that 
$\left( \frac{\partial}{\partial k_{0p}}[\msc_i] \right)|_{k_{0q} + k_{q+1,q} = k_{0p} + k_{p+1,p}} = \left(  \frac{\partial}{\partial k_{0q}}[\msc_i] \right)|_{k_{0q} + k_{q+1,q} = k_{0p} + k_{p+1,p}}$, where 
$\msc$ is the coefficient map~\eqref{eq:coeff-map-cycle} of $\mathcal{M}$.  As before, we consider the four types of coefficients $\msc_i$.

\uline{Type I.}
Let $1 \leq i \leq n$.
By equation~\eqref{eq:lem-difference-partials} in 
Lemma~\ref{lem:partial-deriv-1-input}(I), the difference
$\frac{\partial}{\partial k_{0q}}[e_i] - \frac{\partial}{\partial k_{0 \ell}}[e_i] $, when $k_{0q} + k_{q+1,q} = k_{0p} + k_{p+1,p}$, is zero. 

\uline{Type II.}
By equation~\eqref{eq:lem-difference-partials-type-II} in Lemma~\ref{lem:partial-deriv-1-input}(II), the difference 
$\left(
\frac{\partial}{\partial k_{0q}}
- \frac{\partial}{\partial k_{0 \ell}} 
\right)
\left[e_n- {\prod_{i=1}^{n} k_{i+1, i}} \right] $, when $k_{0q} + k_{q+1,q} = k_{0p} + k_{p+1,p}$, is zero. 

\uline{Type III.} Lemma~\ref{lem:partial-deriv-1-input}(III) implies that $\frac{\partial}{\partial k_{0p}}[\kappa] = \frac{\partial}{\partial k_{0 q}}[\kappa] =0$.

\uline{Type IV.} Lemma~\ref{lem:partial-deriv-1-input}(IV)(b) 
implies that $\frac{\partial}{\partial k_{0p}}[e_i^*\kappa] =
\frac{\partial}{\partial k_{0q}}[e_i^*\kappa]  = 0  $, for all $1 \leq i \leq n - p$.

\end{proof}

As mentioned above, the proof of Theorem~\ref{thm:1-1-p-singular-locus} closely follows~\cite[proof of Theorem~5.3]{BGMSS}.

\begin{proof}[Proof of Theorem~\ref{thm:1-1-p-singular-locus}]
Assume $\In=\Out=\{1\}$ and $\Leak=\{\ell\}$.
As the output is in the first compartment, 
 Lemma~\ref{lem:coefficient-map-cycle} implies that the Type-III coefficient is $\kappa = 1$.  Therefore, we may omit that coefficient; that is, we consider the remaining $2n-1$ coefficients (from Lemma~\ref{lem:coefficient-map-cycle}): $e_1, e_2, \dots, e_{n-1}, e_{n}{-}\prod_{i=1}^{n}k_{i+1,i}, e_{1}^{*}, e_{2}^{*}, \dots, e_{n-1}^{*}$.  Here, $e_i$ (respectively, $e_{i}^{*}$) is the $i$-th elementary symmetric polynomial on the set 
    $E=\{ \widetilde{k}_1, \widetilde{k}_2,  \dots, \widetilde{k}_n \}$ 
    (respectively, $E^*=\{ \widetilde{k}_2, \widetilde{k}_3,  \dots, \widetilde{k}_n \}$).

Let $J$ denote the resulting Jacobian matrix; and let $J_{0\ell}$ and $J_{\ell+1, \ell}$ denote the columns of $J$ corresponding to the parameters $k_{0 \ell}$ and $k_{\ell+1,\ell}$, respectively. 
We claim that the matrix $J'$ obtained from $J$ by replacing the column $J_{0 \ell}$ by $J_{0 \ell} - J_{\ell+1, \ell}$ has the following form:
\begin{equation} \label{eq:jacobian-1-1-l}
J'
    ~=~ \quad
\begin{pNiceArray}{cc|ccc}[first-col,first-row]
      & k_{0\ell} & k_{21} & k_{32} & \cdots & k_{1n}\\
e_{1} & 0         &   1    &        &        &\\
e_{2} & 0         &   *    &        &        &\\
\vdots&  \vdots   & \vdots &        &   *    &\\
e_{n{-}1}&  0     &   *    &        &        &\\
e_{n}{-}\prod_{i=1}^{n}k_{i+1,i}&  \prod_{i\in[n]\smallsetminus\ell}k_{i+1,i}   &    *    &        &        &\\
\hline
e_{1}^{*}&   0    &   0    &        &        &\\
\vdots& \vdots    & \vdots &        &   M    &\\
e_{n{-}1}^{*}& 0  &   0    &        &        &\\
\end{pNiceArray}~.
\end{equation}
Indeed, the first column in~\eqref{eq:jacobian-1-1-l} was shown in the proof of Theorem~\ref{thm:singular-locus-cycle} (part (2)); and the verification of the form of the second column is straightforward.  

Next, we claim that the singular locus is defined by the vanishing of the following expression:
    \begin{equation}
        \label{eq:minor-1-1-l}
    (\det M)  \prod_{i\in [n] \smallsetminus \{\ell \} }  k_{i+1,i}~,
    \end{equation}
where  $M$ is the $(n-1) \times (n-1)$ lower-right submatrix in~\eqref{eq:jacobian-1-1-l} (we will compute it below).  
Indeed, each $(n+1)\times(n+1)$ minor of $J'$ is a scalar multiple of the expression~\eqref{eq:minor-1-1-l},  so it defines the singular locus.

By comparing the expression~\eqref{eq:minor-1-1-l} to the equation~\eqref{eq:s-locus-in=out=1}, we see that all that remains to be shown is that $\det M$ equals, up to sign, the (Vandermonde) polynomial $ \prod_{2 \leq i < j \leq n} (\widetilde{k}_{i} - \widetilde{k}_{j})$.  

To this end, it is straightforward to check that the matrix $M$ in~\eqref{eq:jacobian-1-1-l} is as follows:
\begin{equation} \label{eq:submatrix-vandermonde}
M
    ~=~ 
\begin{pNiceArray}{cccc}
e_{1}(E^* \smallsetminus \{\widetilde{k}_2\}) & 
    e_{1}(E^* \smallsetminus \{\widetilde{k}_3\}) & 
    \dots
    &
    e_{1}(E^* \smallsetminus \{\widetilde{k}_n\}) 
    \\
e_{2}(E^* \smallsetminus \{\widetilde{k}_2\}) & 
    e_{2}(E^* \smallsetminus \{\widetilde{k}_3\}) & 
    \dots
    &
    e_{2}(E^* \smallsetminus \{\widetilde{k}_n\}) \\
\vdots & \vdots & \ddots & \vdots \\
e_{n-1}(E^* \smallsetminus \{\widetilde{k}_2\}) & 
    e_{n-1}(E^* \smallsetminus \{\widetilde{k}_3\}) & 
    \dots
    &
    e_{n-1}(E^* \smallsetminus \{\widetilde{k}_n\}) 
\end{pNiceArray}
\end{equation}
This matrix~\eqref{eq:submatrix-vandermonde} is exactly what is obtained from the corresponding matrix (also called $M$) in \cite[proof of Theorem 5.3]{singularlocus}, 
after performing the substitution $a_{\ell + 1, \ell} \mapsto \widetilde{k}_{\ell}$. 
The determinant of $M$ given in that proof exactly recovers the desired Vandermonde polynomial, which completes the proof.
\end{proof}

%------------------------
\section{Catenary models} \label{sec:catenary}
%------------------------
In prior sections, we fully characterized identifiability in directed-cycle models and analyzed their singular loci.  In this section, we turn to catenary models.  Recall that 
such models are bidirected-tree models (Definition~\ref{def:types-of-models}) and so -- when there is only one input and one output -- identifiability can be assessed simply by 
counting the number of leaks and 
examining the distance from the input to the output (Proposition \ref{prop:classification-identifiable-models}).
However, the singular locus of (generically locally
identifiable) catenary models is not well understood~\cite{singularlocus}, and additionally we wish to study individual parameters (in possibly unidentifiable models) and their impact on identifiability~\cite{sling}.  (There are some results in this direction, for certain cases of catenary models~\cite{chau1985linear,vicini2000identifiability}.)

One approach to tackling these problems is to start with an explicit formula for the coefficient map of an arbitrary catenary model, much like the formula stated earlier for cycle models (Lemma~\ref{lem:coefficient-map-cycle}) which allowed us to analyze the singular locus of cycle models.  Accordingly, we prove a coefficient-map formula for catenary models, Theorem~\ref{thm:catenary-coefficient-map} below, which is the main result of this section.  

The remainder of this section is organized as follows.  
Section~\ref{sec:catenary-prelim} introduces the notation we need for Theorem~\ref{thm:catenary-coefficient-map}, which we state in Section~\ref{sec:catenary-statment}.  In Section~\ref{sec:catenary-bortner-subsection}, we recall a prior result 
we subsequently use to prove  Theorem~\ref{thm:catenary-coefficient-map} in Section~\ref{sec:catenary-formula-proof}.

%------------------------
\subsection{Preliminaries} \label{sec:catenary-prelim}
%------------------------
In order to state Theorem~\ref{thm:catenary-coefficient-map}, we need some notation.
Consider an $n$-compartment catenary model $\mathcal{M} = (G, \In, \Out, \Leak)$.  
For $1 \leq i\leq j \leq n$, we denote the product of the edge-parameters along the path from compartment-$i$ to compartment-$j$ by
                         \begin{align} \label{eq:kappa-product-path}
                         \kappa(i,j) ~ := ~
                             k_{i+1,i} k_{i+2,i+1} \cdots k_{j,j-1}   ~,      
                        \end{align}
which matches the notation from~\eqref{eq:product-i-to-j}.  
Next, for $1 \leq \ell \leq n$, we let $ \mathrm{Out}_{\mathcal{M}}(\ell) $
denote the sum of the parameters of all outgoing edges and/or leaks at compartment-$\ell$ as follows:
$$ \mathrm{Out}_{\mathcal{M}}(\ell) ~:=~ 
\begin{cases} 
    k_{21}  & \text { if }  \ell = 1  \text{ and } \ell\not\in \Leak  \\ 
    k_{21} + k_{01} & \text { if } \ell = 1 \text{ and }  \ell \in \Leak \\
    k_{\ell - 1,\ell} + k_{\ell+1 ,\ell} & \text { if } 2 \leq \ell \leq n-1 \text{ and } \ell\not\in \Leak \\ 
    k_{\ell - 1,\ell} + k_{\ell+1 ,\ell} + k_{0,\ell} & \text { if } 2 \leq \ell \leq n-1 \text{ and }  \ell \in \Leak \\
    k_{n-1,n}  & \text { if } \ell = n  \text{ and }  \ell\not\in \Leak \\ 
    k_{n-1,n} + k_{0n} & \text { if } \ell = n \text{ and }  \ell \in \Leak~. %\\
    \end{cases}
    $$
For any subset $L \subseteq [n]$, we introduce the set of the corresponding ``outgoing sums'':
\begin{align}
    \label{eq:out-set}
    \mathrm{Out}_{\mathcal{M}}(L) ~:=~ \{ \mathrm{Out}_{\mathcal{M}}(\ell) \mid \ell \in L\}~.
\end{align}
Next, we denote the set of nonempty subsets of $[n-1]$ without consecutive numbers:
\begin{align} \label{eq:sets-with-no-consec}
    \Gamma_n ~:&=~ \{A \subseteq [n-1] \mid 
    A \neq \varnothing ;~ 
    \mathrm{if~}i\in A, \mathrm{~then~} i+1 \notin A \}~.
\end{align}
Finally, for any subset of positive integers $I$, we define the following product of edge-labels of cycles $i \leftrightarrows i+1$:
\begin{align} \label{eq:kappa-product}
    \kappa_I ~:&=~ \prod_{i\in I} 
        k_{i,i+1} k_{i+1,i}~,
    \end{align}
and we also define the following set:
\begin{align} \label{eq:I+1}
    I^+~:=~ I \cup \{j+1 \mid j \in I\}~.
\end{align}

%------------------------
\subsection{Main result} \label{sec:catenary-statment}
%------------------------
The graphs underlying catenary models exhibit reflectional symmetry (symmetry under switching compartment-$i$ with compartment-$(n-i)$ for all $i$).  Therefore, by relabeling compartments, if needed, we may assume that the input is located ``before'' the output. 
We make this assumption in the following theorem, 
which
is the main result of this section.

\begin{theorem}[Coefficient map of catenary model] \label{thm:catenary-coefficient-map}
Let $n \geq 1$. 
Let $\mathcal{M} = (G, \In, \Out, \Leak)$ be an $n$-compartment catenary model. 

Assume that $\mathcal{M}$
has only one input and one output, $\In = \{in\}$ and $\Out = \{out\}$, and that the input appears at or before the output ($in \leq out$).  
Let $d$ denote the length of the (generalized\footnote{The concept of path is generalized to allow length-$0$ paths, as in Remark~\ref{rem:path}.}) path from input to output ($d=out-in$), and let $\mathcal{P}$ denote the set of vertices along this path ($\mathcal{P}=\{in, in+1,\dots, out\}$).  
The coefficient map of $\mathcal{M}$, 
    $$\msc: \R^{2n+\lvert Leak \rvert -2} \longrightarrow \R^{2n-d+1}~,$$ 
where we write 
$\msc=(
    \mathsf{a}_1, 
    \mathsf{a}_2, 
    \dots, \mathsf{a}_n,
    ~
    \widetilde{\mathsf{a}}_d, \widetilde{\mathsf{a}}_{d+1}, \dots, \widetilde{\mathsf{a}}_n) $,    
    is given by (for relevant indices $i$)
{\footnotesize 
\begin{align} \label{eq:coefficient-catenary-LHS}
\mathsf{a}_i ~:=~
        &e_i \left( \mathrm{Out}_{\mathcal{M}}([n]) \right) 
        - 
        \sum_{ \{ I \in \Gamma_n ~:~ 2\lvert I \rvert \leq i \} } \kappa_I ~
                e_{i-2\lvert I \rvert }(\mathrm{Out}_{\mathcal{M}}([n] \smallsetminus I^+))~,
                \\
\label{eq:coefficient-catenary-RHS}
\widetilde{\mathsf{a}}_i ~:=~
        \kappa(in,out)
        &
        \left(
         e_{i-d} 
        \left( 
        \mathrm{Out}_{\mathcal{M}}
        ([n] \smallsetminus \mathcal{P} ) 
        \right) 
        - 
        \sum_{ \{ I \in \Gamma_n ~:~ 2\lvert I \rvert \leq i,~ I \cap \mathcal{P} = \varnothing \} } \kappa_I ~
                e_{i-d-2\lvert I \rvert }(\mathrm{Out}_{\mathcal{M}}([n] \smallsetminus (I^+ \cup \mathcal{P})))
        \right)~,     
\end{align}
} \\ 
where $e_i(A)$ denotes the $i$-th elementary symmetric polynomial of a set $A$, and the remaining notation is as in~\eqref{eq:kappa-product-path}--\eqref{eq:I+1}.
\end{theorem}

The proof of Theorem~\ref{thm:catenary-coefficient-map} appears at the end of this section (Section~\ref{sec:catenary-formula-proof}).

\begin{example} \label{ex:catenary-illustrate-theorem}
    Recall that the catenary model in 
    Figure~\ref{fig: bidirectedcatex} has $n=4$ compartments and $\In=\{1\}$, $\Out=\{2\}$, and $\Leak = \{2,4\}$.  By Theorem~\ref{thm:catenary-coefficient-map}, the coefficient $\mathsf{a}_3$ can be computed as follows (here we start the computation, but do not simplify the expression):
\begin{align} \label{eq:coeff-3}
        \mathsf{a}_3
        ~&=~
        e_3 \left( \{ 
            k_{21}, ~
            k_{12}+k_{32}+k_{02},~
            k_{23}+k_{43},~
            k_{34}+k_{04} 
            \}
            \right) \\
            \notag
        & \quad \quad 
        -
        \big(
        k_{12}k_{21} e_1 
                \left( \{ 
                k_{23}+k_{43},~
                k_{34}+k_{04} 
                \}\right)
        +        
        k_{23}k_{32} e_1 
                \left( \{ 
                k_{21}, ~
                k_{34}+k_{04} 
                \}\right) \\
                \notag
        & \quad \quad \quad \quad 
        +        
        k_{34}k_{43} e_1 
                \left( \{ 
                k_{21}, ~
                k_{12}+k_{32}+k_{02}
                \}\right)
        \big)~.
    \end{align}    
\end{example}

\begin{remark} \label{rem:cat-more-in-out}
    Theorem~\ref{thm:catenary-coefficient-map} pertains to catenary models with one input, one output, and any number of leaks.  Generalizing this result to allow more inputs and/or more outputs is a straightforward application of Proposition~\ref{prop:coeff-many-in-or-out}, and so is omitted.
\end{remark}

%------------------------
\subsection{A formula for coefficients} \label{sec:catenary-bortner-subsection}
%------------------------
Our proof of Theorem~\ref{thm:catenary-coefficient-map} uses a recent formula for coefficient maps given by Bortner {\em et al.} (Lemma~\ref{lem:coeff-formula} below).  We need the following definition.

 \begin{definition}
     \label{def:forest}
     Let $H$ be a directed graph.
     \begin{enumerate}
         \item      A {\bf spanning subgraph} of $H$ is 
     a subgraph of $H$ that involves all vertices of $H$.  
    \item     
    A \textbf{spanning incoming forest}
     of $H$ 
     is a spanning subgraph for which:
        \begin{itemize}
            \item      each node has at most one outgoing edge, and 
            \item     the underlying undirected multigraph
     is a forest, i.e., has no cycles.
        \end{itemize}
     \end{enumerate}
Also, an {\bf m-edge} directed graph is a directed graph with exactly $m$ edges.
 \end{definition}

The following result, which is~\cite[Theorem~3.1]{BGMSS}, is stated for the case of models with one input and one output, which is the situation considered in this section (recall Remark~\ref{rem:cat-more-in-out}).

\begin{lemma}[Coefficients of input-output equations for linear compartmental models~\cite{BGMSS}] 
\label{lem:coeff-formula}
    Let $n \geq 1$.
    Let $\mathcal{M} = (G, \In, \Out, \Leak)$
    be a linear compartmental model with $n$ compartments, one input, and one output: $\In=\{j\}$ and $\Out = \{i\}$.  Write the input-output equation~\eqref{eq:in-out} as follows:
 	\begin{align}
    \label{eq:eq:i-o-c-and-d}
	 y_i^{(n)} + c_{n-1}  y_i^{(n-1)} + \dots  + c_1 y_i' + c_0 y_i ~=~ 
		d_{n-1} u_j^{(n-1)} + %d_{n-2} u_1^{(n-2)} +
			 \dots  + d_{1} u_j' + d_{0} u_j
		~.
	\end{align}   
    Then the coefficients of the input-output equation~\eqref{eq:eq:i-o-c-and-d} are given by:
    \begin{align} \label{eq:bortner-formula-coeff}
    c_k ~&=~ \sum_{F \in \mathcal{F}_{n-k}(\widetilde{G})} \pi_F \quad \text{ for } k=0,1, \ldots , n-1~, \ \text{ and } \\
    \label{eq:bortner-formula-coeff-RHS}
    d_k ~&=~ \sum_{F \in \mathcal{F}^{j,i}_{n-k-1}(\widetilde{G}^*_i)} \pi_F \quad \text{ for } k=0,1,\ldots, n-1~,
    \end{align}
where:  
\begin{itemize}
    \item $\widetilde{G}$ is the directed graph obtained from $G$ by adding (1)~a new vertex, labeled by $0$, and (2)~for all $\ell \in \Leak$, an edge $\ell \to 0$, labeled by $k_{0 \ell}$,
    \item $\widetilde{G}^*_i$ is the directed graph obtained from $\widetilde{G}$ by removing all outgoing edges from vertex-$i$ (the output), 
	\item  $\mathcal{F}_{m} ( \widetilde{G} )$ is the set of all $m$-edge spanning incoming  forests of  $\widetilde{G}$,  
	\item  $ \mathcal{F}_{m}^{j, i } ( \widetilde{G}_i^* )$ is the set of all $m$-edge spanning incoming forests of 
    $\widetilde{G}_i^*$, such that some connected component (of the underlying undirected graph) contains both $j$ and $i$, 
    \item 
    $\pi_F$ is the product of edge labels of a graph $F$, that is,
    $\pi_F := \prod_{ e \in E_F } L(e)~$, where 
    $L(e)$ is the label of edge $e$, and 
    $E_F$ is the set of edges of $F$.
\end{itemize} 
\end{lemma}

\begin{remark} \label{rem:error-in-formula}
In Lemma~\ref{lem:coeff-formula}, the input-output equation~\eqref{eq:eq:i-o-c-and-d} differs from the one in~\cite[Theorem~3.1]{BGMSS}, which erroneously has a factor of $
             (-1)^{i+j} $ on the right-hand side.
\end{remark}

In the next section, we use Lemma~\ref{lem:coeff-formula} to prove Theorem~\ref{thm:catenary-coefficient-map}.  
We illustrate the idea behind this proof through the following example.

\begin{example}[Example~\ref{ex:catenary-illustrate-theorem}, continued]
\label{ex:catenary-illustrate-theorem-2}
We revisit the $4$-compartment catenary model from 
    Figure~\ref{fig: bidirectedcatex}, in which %$\In=\{1\}$, $\Out=\{2\}$, and 
    $\Leak = \{2,4\}$.  By Lemma~\ref{lem:coeff-formula}, the coefficient $c_1$ (which is the coefficient $\mathsf{a}_3$ examined in Example~\ref{ex:catenary-illustrate-theorem}) is the sum over 
    certain spanning subgraphs of the following graph $\widetilde{G}$, namely, the $3$-edge, incoming forests:

\begin{center}
    \begin{tikzpicture}
        \tikzset{node basic/.style={draw, ultra thick, text=black, minimum size=2em}}
        \tikzset{node circle/.style={node basic, circle}}
        \tikzset{line basic/.style={very thick, ->}}
        
        \node[node circle] (0) at (4, 1.75) {\(0\)};
        \node[node circle] (1) at (0, 0) {\(1\)};
        \node[node circle] (2) at (2, 0) {\(2\)};
        \node[node circle] (3) at (4, 0) {\(3\)};
        \node[node circle] (4) at (6, 0) {\(4\)};

        \draw[line basic] (1.north east)--(2.north west) node[midway, above] {\(k_{21}\)};
        \draw[line basic] (2.south west)--(1.south east) node[midway, below] {\(k_{12}\)};
        \draw[line basic] (2.north east)--(3.north west) node[midway, above] {\(k_{32}\)};
        \draw[line basic] (3.south west)--(2.south east) node[midway, below] {\(k_{23}\)};
        \draw[line basic] (3.north east)--(4.north west) node[midway, above] {\(k_{43}\)};
        \draw[line basic] (4.south west)--(3.south east) node[midway, below] {\(k_{34}\)};
        % LEAK
        \draw[line basic] (2.north)--(0.south west) node[midway, above] {\(k_{02}\)};
        \draw[line basic] (4.north)--(0.south east) node[midway, above] {\(k_{04}\)};
        
\end{tikzpicture}
\end{center}
In turn, this sum can be computed by taking the corresponding sum over all 
$3$-edge, spanning subgraphs of $\widetilde{G}$
in which each node has at least one outgoing edge
(regardless of whether they are forests or not), and then subtracting the corresponding sum over such subgraphs that are \uline{not}
forests.  This description of $c_1$ as a difference of sums can be seen to exactly match the computation (of $\mathsf{a}_3$) shown earlier in~\eqref{eq:coeff-3}.    
\end{example}

%------------------------
\subsection{Proof of Theorem~\ref{thm:catenary-coefficient-map}} \label{sec:catenary-formula-proof}
%------------------------

\begin{proof}[Proof of Theorem~\ref{thm:catenary-coefficient-map}]
   Our plan is to apply Lemma~\ref{lem:coeff-formula}.
   To this end, following the notation in Lemma~\ref{lem:coeff-formula},
   let $\widetilde{G}$ denote the graph obtained by adding ``leak edges'' $\ell \to 0$ (for all $\ell \in \Leak$) to the $n$-vertex bidirected-path graph.  (Such a graph $\widetilde{G}$ appeared in Example~\ref{ex:catenary-illustrate-theorem-2}.)  Our next aim is to interpret summands in the desired formula~\eqref{eq:coefficient-catenary-LHS} in terms of subgraphs of $\widetilde{G}$, as follows:

\noindent
   {\bf Claim:} The expressions in~\eqref{eq:coefficient-catenary-LHS} have the following interpretations:
    \begin{enumerate}[label=(\roman*)]%[(i)]
        \item 
        \begin{align*}
        e_i \left( \mathrm{Out}_{\mathcal{M}}([n]) \right)     
        ~=~
        \sum_{F \mathrm{~is~an~}
        i\mathrm{-edge~spanning,~incoming~subgraph~of~} \widetilde{G}
        } \pi_F
        \end{align*} 
        \item 
        {\footnotesize
        \begin{align*}
            \sum_{ \{ I \in \Gamma_n ~:~ 2\lvert I \rvert \leq i \} } \kappa_I 
            ~
            e_{i-2\lvert I \rvert }(\mathrm{Out}_{\mathcal{M}}([n] \smallsetminus I^+))
                ~=~
        \sum_{\substack{F \mathrm{~is~an~}
        i\mathrm{-edge~spanning,~incoming~subgraph~of~} \widetilde{G} \\
       \mathrm{~that~is~not~a~forest}}
        } \pi_F~.
        \end{align*}
        }
    \end{enumerate}
We verify this claim, as follows.  An $i$-edge spanning, incoming subgraph of any graph $H$ is obtained by choosing $i$ vertices of $H$ and then (if possible) choosing one outgoing edge from each of those vertices.  This procedure, applied to the graph $\widetilde{G}$, is exactly captured by the expression on the left-hand side of part~(i) of the claim, and so completes the proof of that part of the claim.  

As for part~(ii), it is straightforward to check that the spanning, incoming subgraphs of $\widetilde{G}$ that are {\em not} forests are exactly those with at least one bidirected edge $j \leftrightarrows j+1$, where $1 \leq j \leq n-1$.  Also, ``adjacent'' bidirected edges $j \leftrightarrows j+1  \leftrightarrows j+2$ are forbidden (as incoming subgraphs prohibit two outgoing edges from $j+1$).  So, the possible collections of such edges, where we use $j$ to refer to $j \leftrightarrows j+1$, are indexed by the set $\Gamma_n$ in~\eqref{eq:sets-with-no-consec}.  From here it is simple to verify part~(ii).

Now the desired formula~\eqref{eq:coefficient-catenary-LHS} follows directly from the above claim and Lemma~\ref{lem:coeff-formula}, specifically, equation~\eqref{eq:bortner-formula-coeff}, where $\mathsf{a}_i=c_{n-i}$.

What remains is to analyze the coefficients $\widetilde{\mathsf{a}}_i$ in~\eqref{eq:coefficient-catenary-RHS}.  We again aim to apply Lemma~\ref{lem:coeff-formula}, so we let $\widetilde{G}^*_{out}$ denote the graph obtained from $\widetilde{G}$ by removing all outgoing edges from the output vertex (which we recall is labeled $out$).  As $out$ has no outgoing edges, the only way for the input and output to be in the same connected component of a subgraph of $\widetilde{G}^*_{out}$ is for the subgraph to contain the directed path $in \to in+1 \to \dots \to out$, which we denote by $P$.  This directed path $P$ explains the term $\kappa(in,out)$ appearing in the formula for $\widetilde{\mathsf{a}}_i$ in~\eqref{eq:coefficient-catenary-RHS}.  

Recall that $\mathcal{P}$ denotes the set of vertices along the directed path $P$.  In addition to the edges in~$P$, the remaining edges in a spanning, incoming subgraph of $\widetilde{G}^*_{out}$ must be outgoing edges of $[n] \smallsetminus \mathcal{P}$.  We must show that the choice of these remaining edges is captured by the expression inside the large parentheses in~\eqref{eq:coefficient-catenary-RHS}.  This expression is easily seen to be analogous to that in~\eqref{eq:coefficient-catenary-LHS} and the proof is analogous to the proof of the claim we gave above.  Now part~(ii) follows from 
Lemma~\ref{lem:coeff-formula}, specifically, equation~\eqref{eq:bortner-formula-coeff-RHS}, where $\widetilde{\mathsf{a}}_i=d_{n-i-1}$.
\end{proof}

\section{Discussion}\label{sec:Conj}
In this work, we completely classified generically locally
identifiable directed-cycle models (Theorem~\ref{thm:main-cycle}).  
Notably, as mentioned earlier, 
Theorem~\ref{thm:main-cycle} is the first result on the identifiability of a family of models with any number of inputs and outputs.
A natural question is whether our results can be generalized to other  models that contain cycles.  
Some partial results in this direction -- for models arising from directed-cycle models by adding bidirected-edges -- are easily obtained by combining Theorem~\ref{thm:main-cycle} with results of Bortner {\em et al.}~\cite[Theorems~4.3--4.4]{BGMSS}.  
However, in general this problem remains open.

Another open direction concerns the identifiability of catenary models.  A classification is known when there is only one input and one output (Proposition~\ref{prop:classification-identifiable-models}), but this problem is unresolved when there are multiple inputs or outputs.  It may be the case that a version of the leak-interlacing condition will play a role, and we expect that our formula for the coefficient map (Theorem~\ref{thm:catenary-coefficient-map}) will play a key role in future proofs in this direction. 

Theorem~\ref{thm:catenary-coefficient-map} also can help us investigate the singular locus of generically locally
identifiable) catenary models, and perhaps lead to a resolution of \cite[Conjecture 5.5]{singularlocus}. For directed-cycle models, a formula for the coefficient map enabled our results on their singular loci (Theorems~\ref{thm:singular-locus-cycle} and~\ref{thm:1-1-p-singular-locus}), but 
our understanding of the singular loci of these models is still quite incomplete (Conjecture~\ref{conj:s-locus}).

Finally, we return to the Question~({\bf Q1})
posed in the Introduction: in bidirected-tree models that are unidentifiable, which individual parameters are identifiable?  It is our hope that some answers to this question for catenary models will be forthcoming from our formula for the coefficient map (Theorem~\ref{thm:catenary-coefficient-map}).  We also could pose the analogous question for directed-cycle models, and we expect that an answer again will involve the leak-interlacing condition.

For readers interested in exploring the open problems posed above, we recommend the 
 {\tt Maple} software {\tt SIAN} (Structural Identifiability ANalyser)~\cite{SIAN}, which is available as a web-based MapleCloud application and is based on results of 
Hong, Ovchinnikov, Pogudin, and Yap~\cite{HOPY}. Another option is a {\tt Julia} package, {\tt StructuralIdentifiability.jl}, that implements theory of 
Dong, Goodbrake, Harrington, and Pogudin~\cite{julia-software}.
We also make note of an investigation of models coming from applications~\cite{sling}, and 
databases of models with three compartments~\cite{norton1982investigation} and up to four compartments~\cite{gogishvili-database}, together with their identifiability properties.  In a similar vein, 
a database containing small directed-cycle and catenary models can be found in Appendix~\ref{appendix:cycle}--\ref{appendix:cat}.  We anticipate that patterns identified through such explorations will generate conjectured answers 
to the questions posed above, 
and, in turn, inspire new theorems.

{\small
\subsection*{Acknowledgements}
Paul Dessauer, Alexis Edozie, Odalys Garcia-Lopez, Tanisha Grimsley, and Viridiana Neri initiated this
research in the 2022 MSRI-UP REU, 
which was hosted by the Mathematical Sciences Research Institue (MSRI, now SL-Math) and is supported by the NSF (DMS-2149642) and the Alfred P. Sloan Foundation. Saber Ahmed, Natasha Crepeau, Jordy Lopez Garcia, and Anne Shiu were mentors of MSRI-UP.  
AS and JLG were partially supported by the NSF (DMS-1752672).
The authors graciously acknowledge Jose Lopez for insights and discussions during the course of this project.  
The authors thank 
Federico Ardila
for helpful discussions and 
for his leadership as on-site director of MSRI-UP.  
The authors are also grateful to 
MSRI for the wonderful experience, and Scribble Together for proving tools for remote collaboration.
AS thanks Dean Baskin for discussions  
that contributed to the proof of Lemma~\ref{lem:alg-indep}(2).
The authors thank Alexey Ovchinnikov and Gleb Pogudin for detailed suggestions on an earlier draft,  and acknowledge two reviewers for helpful comments which improved this work.
}

\bibliographystyle{plain}
\bibliography{references.bib}

\addresseshere

\newpage

\begin{appendix}
\section{Directed-cycle models with two leaks} \label{appendix:cycle}
The next two pages form a database of directed-cycle models with up to three compartments, one input, one output, and two leaks.  
This information was obtained using the software {\tt SIAN} (Structural Identifiability ANalyser)~\cite{SIAN}.  For each model, the input is assumed to be in compartment-$1$, and the following information is provided:
    \begin{itemize}
        \item Number of vertices (compartments)
        \item Location of the output and locations of the two leaks
        \item Number of inputs
        \item Number of outputs
        \item Whether or not the model is (generically locally) identifiable (the rows corresponding to identifiable models are highlighted)
        \item Globally identifiable parameters and initial conditions
        \item Locally identifiable parameters and initial conditions
        \item Non-identifiable parameters and initial conditions
    \end{itemize}
For a definitions pertaining to the identifiability of parameters and initial conditions (as opposed to identifiability of models), we refer the reader to 
\cite[Definition 2.5]{HOPY} or~\cite{sling}.  A model is (generically locally) identifiable if and only if none of its parameters is non-identifiable.

\includepdf[pages=-,angle=90]{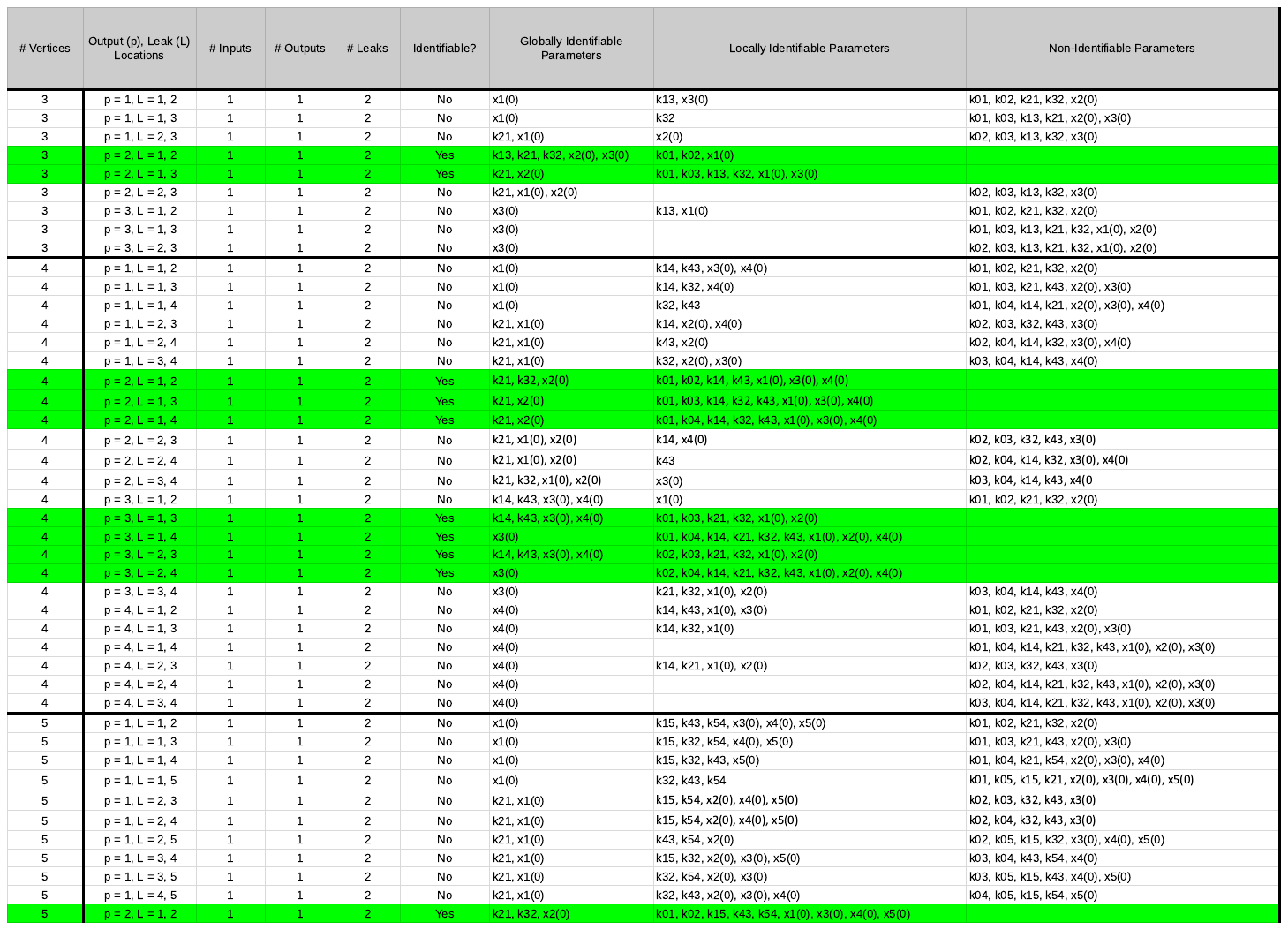}

\section{Catenary models with one input and one output} \label{appendix:cat}

\subsection{Identifiable models}
Figure~\ref{fig:catenary-database} shows a database of (generically locally) identifiable catenary models with three compartments, one input, one output, and up to one leak.  For each model, the following information is provided:
    \begin{itemize}
        \item The model, which is specified by three vectors of the form $(a,b,c)$, one for each the three compartments, with $a \in \{1,0\}$ (respectively, $b$ or $c$) indicating whether the compartment is an input (respectively, output or leak) or not
        \item Number of leaks
        \item A diagram of the model, in which green indicates globally identifiable edges/leaks and blue indicates locally identifiable
    \end{itemize} 

\begin{figure}[htb]
\begin{center}
\includegraphics[scale=0.6,trim={0.5in 0.5in 0 0.5in},clip]{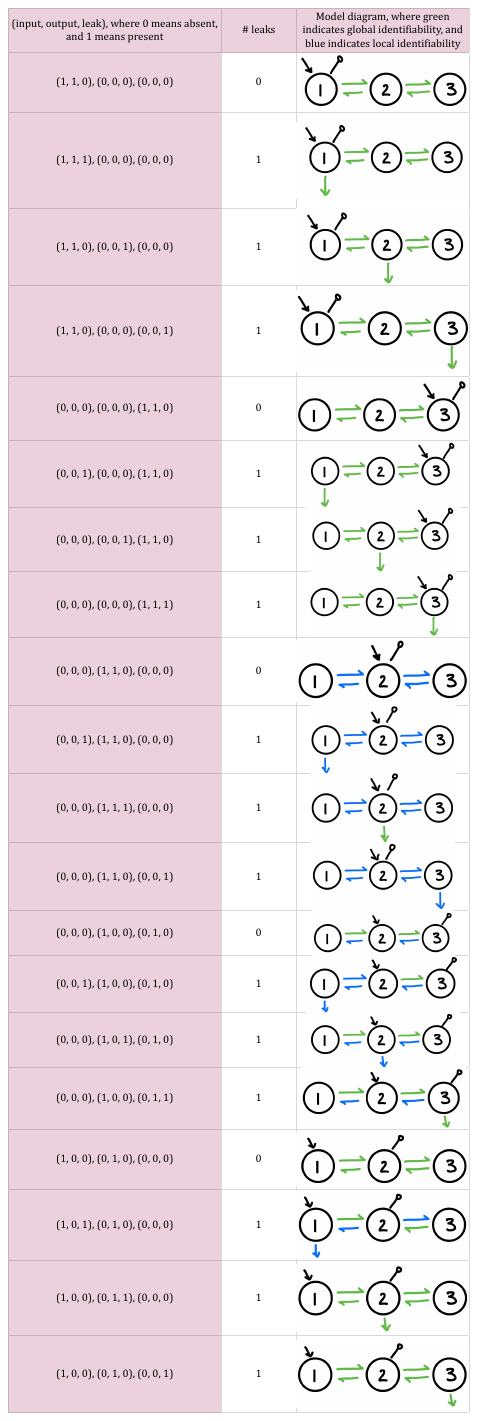} 
\caption{Identifiable catenary models with three compartments, one input, one output, and zero or one leaks. {\tt SIAN} code for these models is available in a supplementary data file and also accessible via the \href{https://github.com/odalys-garcia/Identifiability-of-Directed-Cycle-and-Catenary-Linear-Compartmental-Models}{GitHub repository} \cite{github-repo}.} \label{fig:catenary-database}
\end{center}
\end{figure}

\subsection{All models}

Figures~\ref{fig:catenary-database-2}--\ref{fig:catenary-database-2-5} show a database of catenary models with three compartments, one input, and one output.  For each model, the following information is provided:

    \begin{itemize}
        \item Number of vertices (compartments)
        \item Number of inputs
        \item Number of outputs
        \item Number of leaks
        \item Whether or not the model is (generically locally) identifiable (green [respectively, blue or red] rows correspond to globally [respectively, locally or un-] identifiable models)
        \item Location of input
        \item Location of output
        \item Location of leaks
        \item Globally identifiable parameters
        \item Locally identifiable parameters
        \item Non-identifiable parameters
    \end{itemize} 

\newpage
\begin{figure}[h]
\begin{center}
\includegraphics[scale=0.8, trim={0.5in 0.5in 0 0.5in},clip]{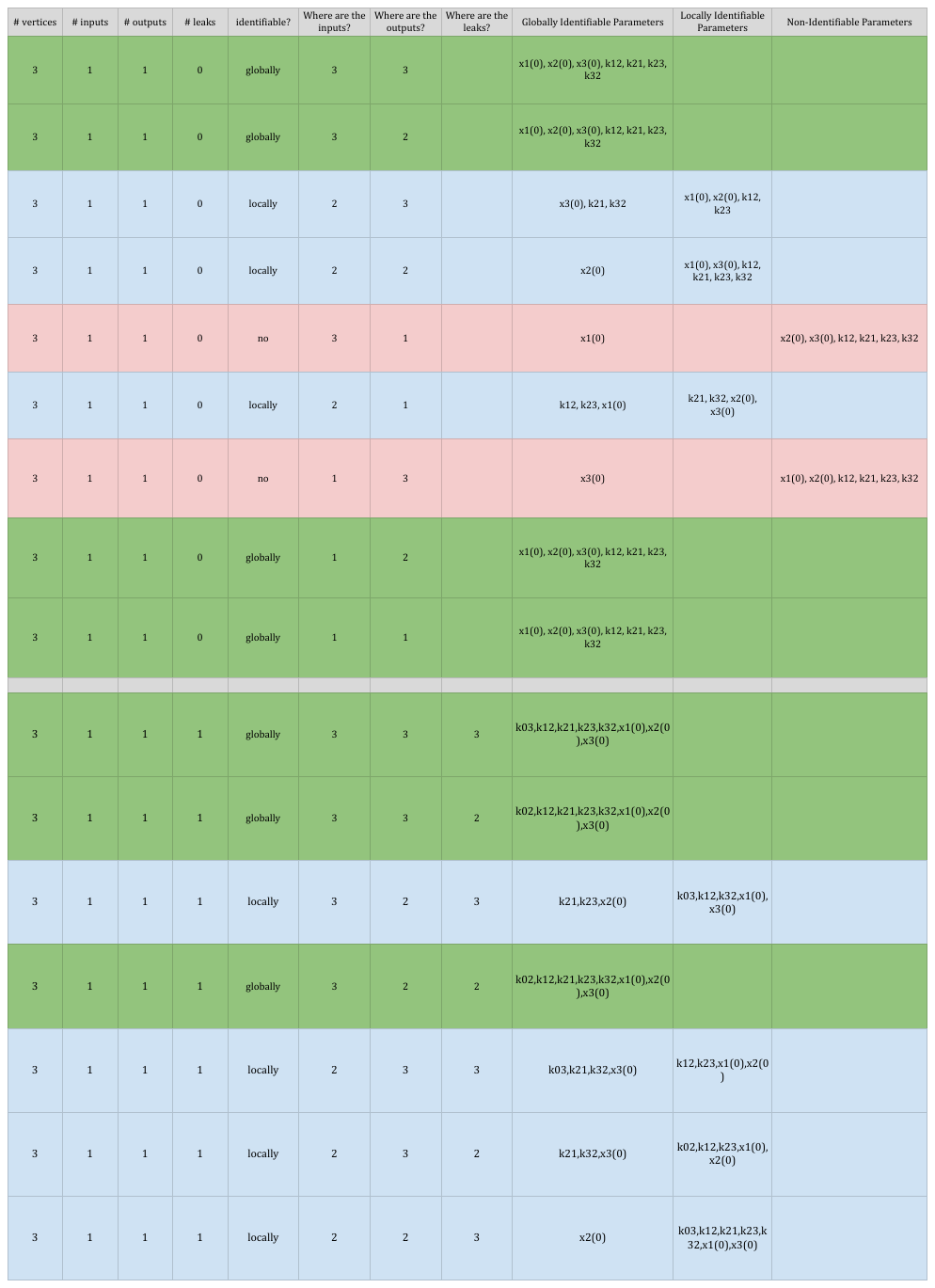}
\caption{Catenary models with three compartments, one input, and one output.
{{\tt SIAN} code for these models is available in a supplementary data file.}
} \label{fig:catenary-database-2}
\end{center}
\end{figure}

% PAGE 2
\newpage
\begin{figure}[h]
\begin{center}
\includegraphics[page=2, scale=0.8, trim={0.5in 0.5in 0 0.5in},clip]{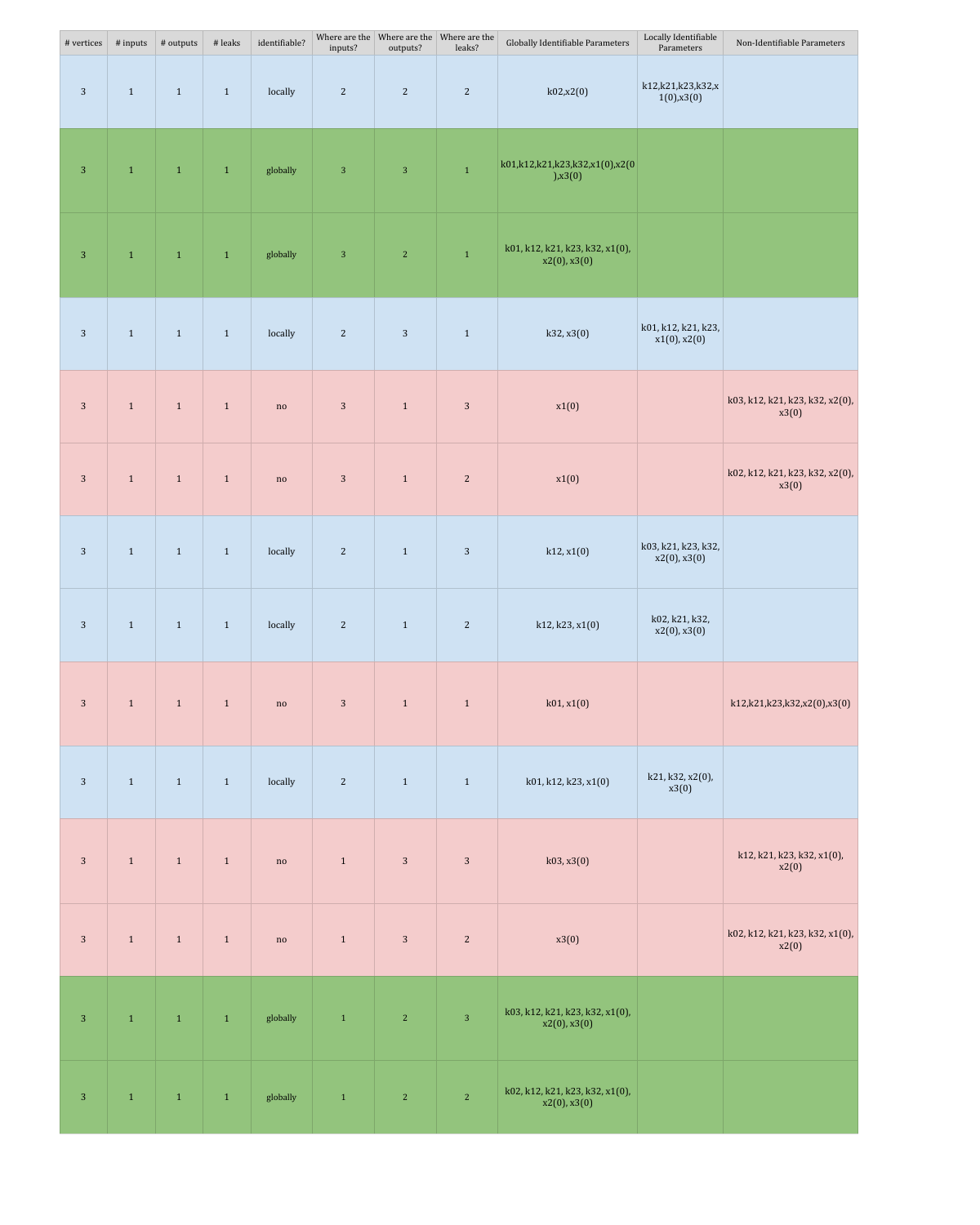}
\caption{Catenary models with three compartments, one input, and one output (page 2).} \label{fig:catenary-database-2-2}
\end{center}
\end{figure}

% PAGE 3
\newpage
\begin{figure}[h]
\begin{center}
\includegraphics[page=3, scale=0.8, trim={0.5in 0.5in 0 0.5in},clip]{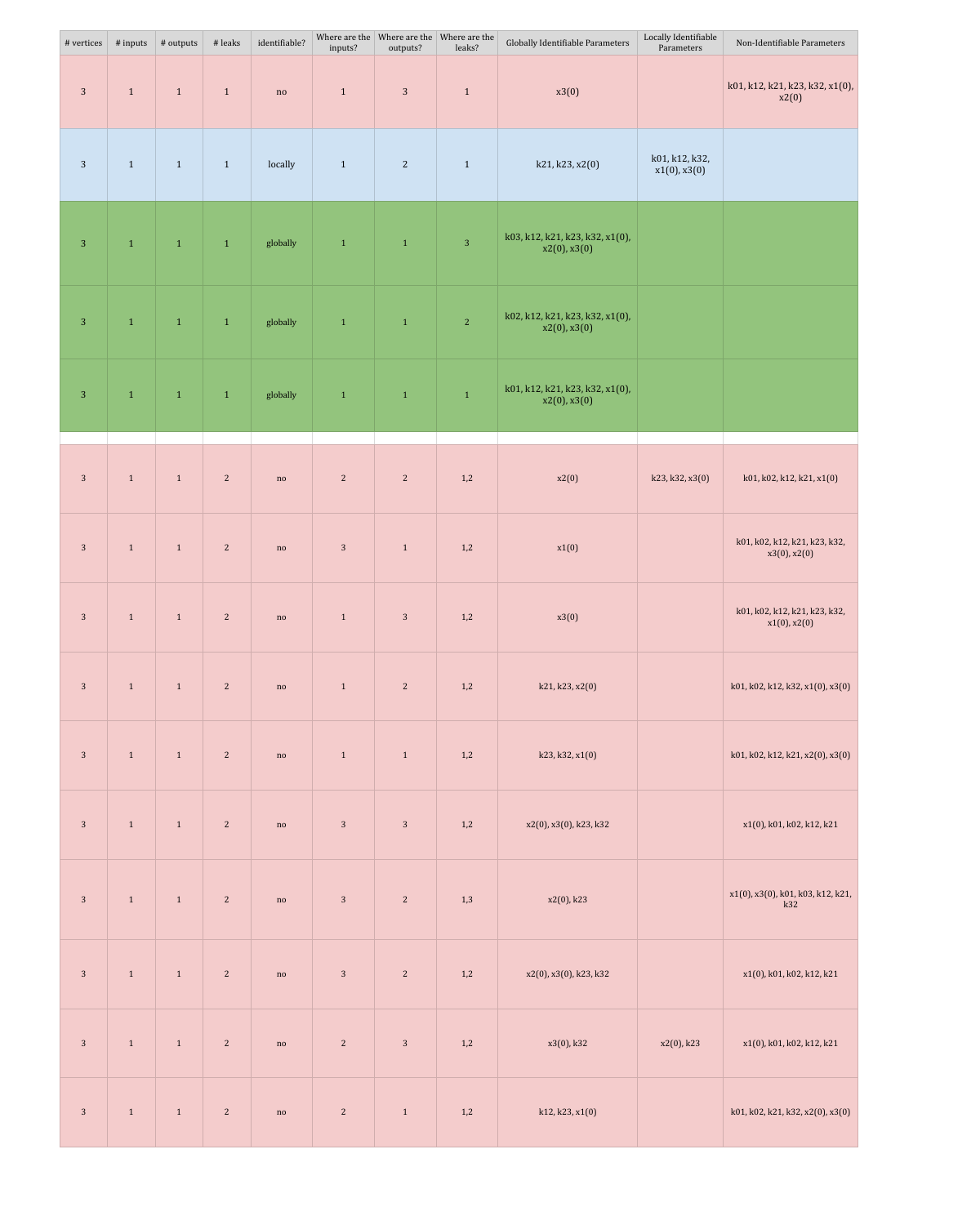}
\caption{Catenary models with three compartments, one input, and one output (page 3).} \label{fig:catenary-database-2-3}
\end{center}
\end{figure}

% PAGE 4
\newpage
\begin{figure}[h]
\begin{center}
\includegraphics[page=4, scale=0.8, trim={0.5in 0.5in 0 0.5in},clip]{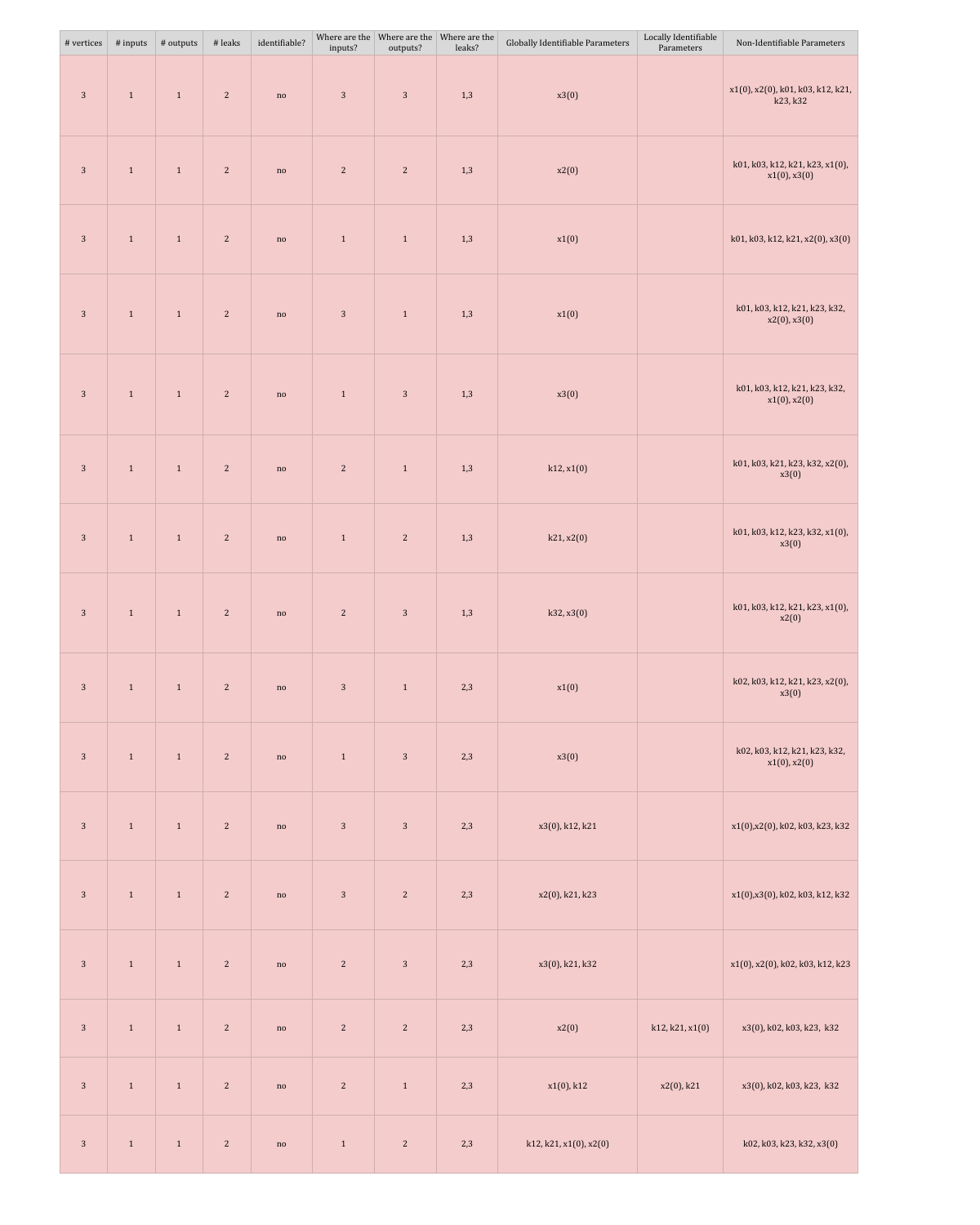}
\caption{Catenary models with three compartments, one input, and one output (page 4).} \label{fig:catenary-database-2-4}
\end{center}
\end{figure}

% PAGE 5
\newpage
\begin{figure}[h]
\begin{center}
\includegraphics[page=5, scale=0.8, trim={0.5in 0.5in 0 0.5in},clip]{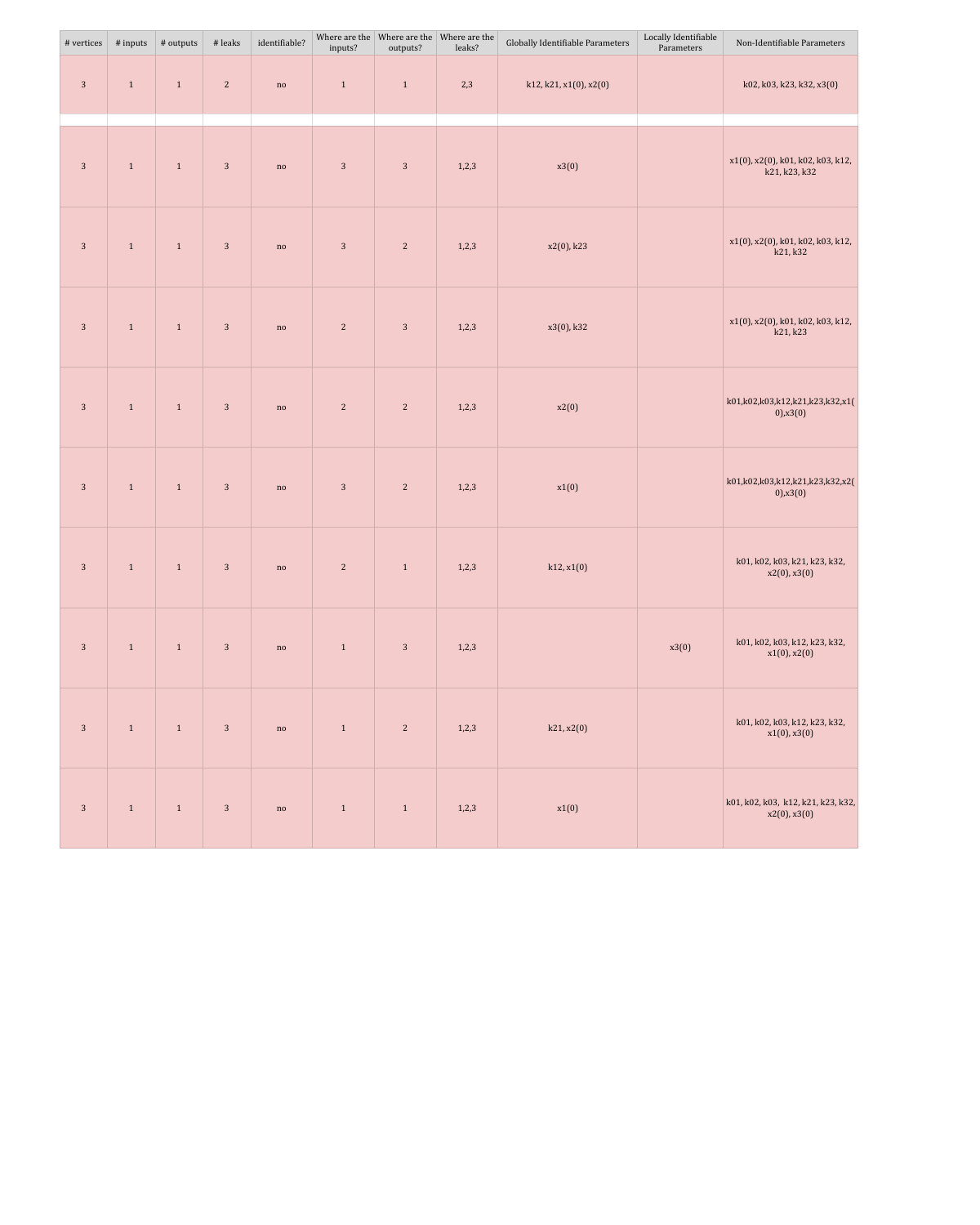}
\caption{Catenary models with three compartments, one input, and one output (page 5).} \label{fig:catenary-database-2-5}
\end{center}
\end{figure}
\end{appendix}
         
\end{document}